\documentclass[11pt,reqno]{amsproc}

\usepackage{amsmath,amsfonts,amssymb,amsthm}
\usepackage[abbrev,lite,nobysame]{amsrefs}
\usepackage{mathrsfs}
\usepackage{graphics,graphicx}
\usepackage[usenames,dvipsnames]{xcolor}
\usepackage{times}
\usepackage{bbm}
\usepackage[margin=1in]{geometry}

\usepackage{enumerate}
\usepackage{enumitem}

\usepackage[]{hyperref}
\hypersetup{
    colorlinks=true,       
    linkcolor=BrickRed,          
    citecolor=black,        
    filecolor=magenta,      
    urlcolor=cyan           
}

\usepackage{tcolorbox}

\usepackage{nicematrix}

\definecolor{blush}{rgb}{0.87, 0.36, 0.51}

\newcommand{\lucas}{\textcolor{blue}}

 \usepackage[toc]{appendix}

\newcommand\e{{\rm e}}

\def\Re{{\rm Re}} 
\def\Im{{\rm Im}}

\newcommand{\R}{\mathbb{R}}
\newcommand{\Z}{\mathbb{Z}}
\newcommand{\T}{\mathbb{T}}

\newcommand{\N}{\mathbb{N}}
\newcommand{\C}{\mathbb{C}}

\renewcommand{\H}{\mathrm{H}}

\newcommand{\eps}{\varepsilon}

\newcommand{\Ld}{\mathrm{L}}
\newcommand{\bcU}{U_s}

\newcommand{\vertiii}[1]{{\left\vert\kern-0.25ex\left\vert\kern-0.25ex\left\vert #1 
		\right\vert\kern-0.25ex\right\vert\kern-0.25ex\right\vert}}

\newtheoremstyle{mystyle}
  {}
  {}
  {\itshape}
  {}
  {\bfseries}
  {.}
  { }
  {}

  \newtheoremstyle{mystyle2}
  {}
  {}
  {}
  {}
  {\bfseries}
  {.}
  { }
  {}

\theoremstyle{mystyle}

\newtheorem{Thm}{Theorem}[section]
\newtheorem{Cor}[Thm]{Corollary}
\newtheorem{Prop}[Thm]{Proposition}
\newtheorem{Lem}[Thm]{Lemma}

\theoremstyle{mystyle2}
\newtheorem{Def}[Thm]{Definition}
\newtheorem{Rem}[Thm]{Remark}

\numberwithin{equation}{section}

\makeatletter

\renewcommand\subsubsection{\@startsection{subsubsection}{3}%
\normalparindent{.5\linespacing\@plus.7\linespacing}{-.5em}
{\normalfont\bfseries}}

\def\@tocline#1#2#3#4#5#6#7{\relax
  \ifnum #1>\c@tocdepth 
  \else
    \par \addpenalty\@secpenalty\addvspace{#2}%
    \begingroup \hyphenpenalty\@M
    \@ifempty{#4}{%
      \@tempdima\csname r@tocindent\number#1\endcsname\relax
    }{%
      \@tempdima#4\relax
    }%
    \parindent\z@ \leftskip#3\relax \advance\leftskip\@tempdima\relax
    \rightskip\@pnumwidth plus4em \parfillskip-\@pnumwidth
    #5\leavevmode\hskip-\@tempdima
      \ifcase #1
       \or\or \hskip 1em \or \hskip 2em \else \hskip 3em \fi%
      #6\nobreak\relax
    \dotfill\hbox to\@pnumwidth{\@tocpagenum{#7}}\par
    \nobreak
    \endgroup
  \fi}
\makeatother

\begin{document} 

\title[Ill-posedness and failure of hydrostatic limit]{\vspace*{-3cm}Ill-posedness of the hydrostatic Euler-Boussinesq equations\\ and failure of hydrostatic limit} 

\author[R. Bianchini]{Roberta Bianchini}
\address{IAC, Consiglio Nazionale delle Ricerche}
\email{r.bianchini@iac.cnr.it}

\author[M. Coti Zelati]{Michele Coti Zelati}
\address{Department of Mathematics, Imperial College London}
\email{m.coti-zelati@imperial.ac.uk}

\author[L. Ertzbischoff]{Lucas Ertzbischoff}
\address{Department of Mathematics, Imperial College London}
\email{l.ertzbischoff@imperial.ac.uk}


 
 



\subjclass[2020]{35Q31, 35Q35, 35R25}

\keywords{Euler-Boussinesq equations, stratified fluids, ill-posedness, hydrostatic limit}

\begin{abstract}
We investigate the hydrostatic approximation for inviscid stratified fluids, described by the two-dimensional Euler-Boussinesq equations in a periodic channel. 
Through a perturbative analysis of the hydrostatic homogeneous setting, we exhibit a stratified steady state violating the Miles-Howard criterion and generating a growing mode, both for the linearized hydrostatic and non-hydrostatic equations. By leveraging long-wave nonlinear instability for the original Euler-Boussinesq system, we demonstrate the breakdown of the hydrostatic limit around such unstable profiles. Finally, we establish the generic nonlinear ill-posedness of the limiting hydrostatic system in Sobolev spaces.
\end{abstract}

\maketitle

\tableofcontents

\medskip

\section{Introduction}\label{SectionIntro-Intro}
In geophysical fluid dynamics, oceanography and atmospheric science, the Euler-Boussinesq equations are often used to describe the behavior of buoyancy-driven incompressible, inviscid and non-homogeneous flows. For $t>0$ and in a thin periodic channel of characteristic size $\eps>0$, the system can be written in terms of suitable rescaled variables as
\begin{equation}\label{eqintro:BE-scaled}
\tag{EB$_{\eps}$}
      \begin{aligned}
      \partial_t \varrho +u \partial_x \varrho+ v \partial_z \varrho&=0, \\
\partial_t u +u \partial_x u + v \partial_z u + \partial_x p&=0,  \\
\eps^2( \partial_t v+ u \partial_x v+v\partial_z v)+\partial_z p +\varrho&=0, \\
\partial_x u + \partial_z v&=0,
\end{aligned}
\end{equation} 
with $(x,z) \in \T \times (-1,1)$ and boundary conditions $v_{\mid z=\pm1}=0$. Here, $(u(t,x,z),v(t,x,z))\in\R^2$ is the velocity field, $p(t,x,z)\in\R$  is the pressure, and $\varrho (t,x,z) \in \R^+$ is the fluid density.
%
%


The parameter $\eps>0$ is often known as the \textit{shallow-water} or \textit{aspect ratio} parameter: it accounts for the signifi\-cant difference between typical (small) vertical and (large) horizontal length scales, and is therefore assumed to be very small. This assumption constitutes a crucial feature in understanding the large-scale dynamics of planetary oceans.
Formally, the \textit{hydrostatic} or \textit{shallow-water} limit $\eps \rightarrow 0 $ yields the so-called \textit{hydrostatic Euler-Boussinesq equations} 
\begin{equation}\label{eqintro:BHyd-E}
\tag{hEB}
      \begin{aligned}
      \partial_t \varrho +u \partial_x \varrho+ v \partial_z \varrho&=0, \\ 
\partial_t u +u \partial_x u + v \partial_z u + \partial_x p&=0,  \\ 
\partial_z p+ \varrho &=0, \\ 
\partial_x u + \partial_z v&=0,
\end{aligned}
\end{equation}
with $(x,z) \in \T \times (-1,1)$ and boundary conditions $v_{\mid z=\pm1}=0$. Here, the former vertical momentum equation has been replaced by the hydrostatic balance between density and vertical variation of pressure. The hydrostatic Euler-Boussinesq equations \eqref{eqintro:BHyd-E}, also known as the inviscid primitive equations, are classical dynamical equations commonly applied in oceanographic and climate studies \cites{KT23, majda2003, Vallis2017}.

However, the approximation $\eps \rightarrow 0$ is a singular limit procedure that is remarkably challenging to justify rigorously. Typically, two related mathematical questions naturally arise:
\begin{enumerate} [label=(Q\arabic*), ref=(Q\arabic*)]
\item\label{item:hydrolimit} Establish or refute the validity of the hydrostatic limit, involving the rigorous derivation of \eqref{eqintro:BHyd-E} from \eqref{eqintro:BE-scaled} as $\varepsilon \rightarrow 0$. As we will see, this can be reformulated as a stability problem for a specific class of solutions of \eqref{eqintro:BE-scaled}.

\item \label{item:illposed} Investigate the limiting system \eqref{eqintro:BHyd-E} in its own right, with the study of the Cauchy theory being one of the initial and fundamental steps.
\end{enumerate}
The main goal of this work is to demonstrate the general failure of the hydrostatic limit and the ill-posedness of the hydrostatic system \eqref{eqintro:BHyd-E}, answering in a negative fashion both \ref{item:hydrolimit} and \ref{item:illposed}. The breakdown of the hydrostatic approximation is attributed to violent instabilities in the dynamics around specific steady states, known as \emph{stratified shear flows}, that take the form
\begin{equation}\label{eq:equilibrium}
(\varrho, u, v, p)(x, z) = (\rho_s(z), U_s(z), 0, p_s(z)), \qquad  \text{with} \qquad  p_s'(z)=-\rho_s(z).
\end{equation}
As elaborated in Section \ref{SectionIntro-Difficulties}, the stability properties of such equilibria constitute the overarching mechanism influencing the hydrostatic approximation, encoded by the particular interplay between shearing and stratification.
Compared to previous results \cites{HKN, ILT}, the main novelty of our work is to address the entire system \eqref{eqintro:BHyd-E} without neglecting the effect of the (stable) stratification. 

\subsection{Main results}\label{SectionIntro-Results}
The classical Miles-Howard criterion \cites{miles1961stability, howard1961note} asserts that such stratified shear flows \eqref{eq:equilibrium} are spectrally stable
for \eqref{eqintro:BE-scaled} and \eqref{eqintro:BHyd-E} provided that 
\begin{equation}\label{eq:Miles-Howard}
 \frac{1}{4} \leq \dfrac{-\rho_s'(z)}{|U_s'(z)|^2}, \qquad \forall z \in [-1,1]. 
\end{equation}
The previous ratio is commonly referred to as the (local) Richardson number in the physics literature. It assesses the (de)stabilizing impact of stratification compared to the effect of the background shear flow. 

Leaning on the construction of stratified shear flows that violate the Miles-Howard criterion \eqref{eq:Miles-Howard}, our first result is concerned with the failure of the hydrostatic limit. 
\begin{Thm}\label{thm-invalidHydrostatlim}
There exists an analytic stationary profile $(\rho_s(z), U_s(z))$, with $\rho_s'<0$, which does not satisfy the Miles-Howard condition \eqref{eq:Miles-Howard}, such that the following holds. For all $p,k \in \N$, there exist $m>0$, a family of smooth solutions $(\varrho_\eps, u_\eps, v_\eps)_{\eps>0}$ of \eqref{eqintro:BE-scaled} and times $T_\eps = O(\eps \vert \log \, \eps \vert) \rightarrow 0$ such that  
\begin{align*}
\Vert (\varrho_\eps, u_\eps, v_\eps)_{\mid t=0}-(\rho_s, U_s,0) \Vert_{\mathrm{H}^p(\T \times (-1,1))} \leq \eps^k,
\end{align*}
while
\begin{align*}
\underset{t \in [0,T_\eps]}{\sup}\Vert (\varrho_\eps, u_\eps, v_\eps)(t)-(\rho_s, U_s,0) \Vert_{\Ld^{2}(\T \times (-1,1))} \geq m.
\end{align*}
\end{Thm}
This theorem precisely tells us that the hydrostatic approximation starting from \eqref{eqintro:BE-scaled} cannot be true in general, even for short times and even if one considers highly regular initial data. As a matter of fact, even if $(\varrho_\eps, u_\eps, v_\eps)_{\mid t=0} \rightarrow (\rho_s, U_s,0)$ at any polynomial speed when $\eps \rightarrow 0$ initially, the corresponding solution $(\varrho_\eps, u_\eps, v_\eps)$ will never converge back to $(\rho_s, U_s,0)$ when $\eps \rightarrow 0$, even for times going to 0. In particular, the hydrostatic limit cannot hold along this solution: no candidate in the limit of $(\varrho_\eps, u_\eps, v_\eps)$ can satisfy an equation having $(\rho_s, U_s,0)$ as a steady state, as \eqref{eqintro:BHyd-E}\footnote{The limiting equations having a \textit{unique} analytic solution starting from the analytic data $(\rho_s, U_s,0)$, see \cite{KTVZ}.}. Note that the density stratification still retains the mild stability condition $\rho_s'<0$, namely that $\rho_s$ is a decreasing function of height.

As we will prove in Section \ref{Section:HydroLimit}, the failure of the hydrostatic limit will be a consequence of a long-wave instability phenomenon that is at stake in the Euler-Boussinesq system \eqref{eqintro:BE-scaled}, for $\eps=1$. As a byproduct, we also obtain a nonlinear Lyapunov instability result for the unscaled Euler-Boussinesq system on a periodic channel (with large enough period), which is a result interesting in itself. We refer to Theorem \ref{thm:nonlin-instabilityEB} for a more precise statement.

\begin{Rem}
Our theorem extends Grenier's result \cite{Grenier-hydroderiv}*{Theorem 1.3} to the case of stratified fluids. Grenier's theorem, addressing the homogeneous $2d$ scenario (see the system \eqref{eq:Hyd-E} in Section \ref{SectionIntro-Difficulties}), is recovered as a byproduct of Theorem \ref{thm-invalidHydrostatlim} since our analysis could essentially be conducted with constant densities.
\end{Rem}

Our second main result is stated in terms of the vorticity of the fluid. In the context of two-dimensional hydrostatic equations, the appropriate scalar quantity serving as the vorticity is given by
$$\omega:=\partial_z u.$$
Applying $\partial_z$ to the equation for $u$ in \eqref{eqintro:BHyd-E}, we obtain the following system on $(\varrho, \omega)$:
\begin{equation}\label{eqintro:BHyd-E-vort}
      \begin{aligned}
      \partial_t \varrho +u \partial_x \varrho+ v \partial_z \varrho&=0, \\[1mm]
\partial_t \omega+u \partial_x \omega + v \partial_z \omega&=\partial_x \varrho,
\end{aligned}
\end{equation}
where the \textit{hydrostatic Biot-Savart law} reads in this context
\begin{align}\label{eq:BiotSavartHydrostat}
u=\partial_z \varphi, \qquad v = -\partial_x \varphi, \qquad  \text{where} \qquad\begin{cases}
        \partial_z^2 \varphi=\omega,\\
        \varphi_{\mid z=\pm 1}=0.
    \end{cases} 
\end{align}
Our result asserts that the system \eqref{eqintro:BHyd-E-vort} is generically nonlinearly ill-posed in Sobolev spaces, even for short times and for an arbitrary loss of derivatives.
\begin{Thm}\label{thm-illposednesNonlin}
There exists an analytic stationary profile $(\rho_s(z), U_s(z))$, with $\rho_s'<0$, which does not satisfy the Miles-Howard condition \eqref{eq:Miles-Howard}, such that the following holds. For any $p,k \in \N$ and $\alpha \in (0,1]$, there exists $0 < \delta_0 < 1$ such that, for any $0<\delta \le \delta_0$, 
there exists a family of solutions $(\varrho_\delta, \omega_\delta)_{\delta<\delta_0}$ to \eqref{eqintro:BHyd-E-vort},  satisfying at initial time 
\begin{align*}
    \big\Vert (\varrho_\delta, \omega_\delta)_{\mid t=0}-\left(\rho_s, U_s'\right)\big\Vert_{\mathrm{H}^p(\T \times (-1,1))} \leq \delta,
\end{align*}
such that, for times $T_\delta=\mathcal{O}(\delta \vert \log \delta \vert) \rightarrow 0$, 
\begin{align*}
\underset{\delta \rightarrow 0}{\lim} \,  \frac{\big\Vert (\varrho_\delta, \omega_\delta)-\left(\rho_s, U_s'\right) \big\Vert_{\Ld^2([0,T_{\delta}]; \H^1(O_\delta))}}{\big\Vert (\varrho_\delta, \omega_\delta)_{\mid t=0}-\left(\rho_s, U_s'\right)\big\Vert_{\mathrm{H}^p(\T \times (-1,1))}^{\alpha}}=+\infty,
\end{align*}
with $O_\delta=B(x_0, \delta^k) \times B(z_0, \delta^k)$ and  where $B(x_0, \delta^k)$ (resp. $B(z_0, \delta^k)$) is an interval centered in $x_0 \in \mathbb T$ (resp. centered in $z_0 \in (-1, 1)$).
Additionally, for any $p'>0$, there exists a universal constant $M>0$ (independent of $\delta$) such that
%
\begin{align*}
\sup_{t \in [0, T_\delta]} \|(\varrho_\delta, \omega_\delta)(t)-\left(\rho_s, U_s'\right)\|_{\H^{1+p'}(O_\delta)}>M.
\end{align*}
\end{Thm}

In essence, this result suggests that conventional fixed-point methods or iterative procedures fall short in constructing strong solutions to \eqref{eqintro:BHyd-E}. Specifically, (Hölder) continuity with respect to the initial data cannot be achieved. The outcome implies that, if the flow exists, it cannot exhibit Hölder continuity from $\mathrm{H}^p$ to $\H^1$ for any $s$ (involving arbitrary loss of derivatives), even for arbitrarily short times. Moreover, it cannot be continuous, for instance, from $\mathrm{H}^p$, with $p \geq 1$, to $\mathrm{H}^{1+}$. Discontinuity of the flow map is usually called \emph{mild ill-posedness}, corresponding to an ill-posedness in the sense of Hadamard. We refer to \cite{elgindi2020} for a discussion on Hadamard ill-posedness in the sense of pathological behavior of the map from the initial data to the solution at later times. A consistent amount of mathematical results in recent years have been devoted to proving several kinds of ill-posedness of hydrodynamic equations. We mention, for instance, \cites{jeong2017, jeong3, bourgain, cordoba1, BHI}, which are based on different strategies. Other weaker forms of ill-posedness exist, such as the loss of Lipschitz continuity of the flow map, see, for instance, \cites{GVD, GVN, GuoNguyen} for the ill-posedness of the Prandtl equation and failure of the boundary-layer expansion around nonmonotonic shear flows in Sobolev setting.  

Our ill-posedness result will mainly come from a linear instability feature: our particular profile of the form \eqref{eq:equilibrium} has an associated linearized operator for the hydrostatic system with an unbounded unstable spectrum, a context in which the general framework developed in \cite{HKN} is particularly fruitful to prove discontinuous dependency of the solution map on the initial data. It would be interesting to see how to make our ill-posedness statement stronger, i.e., without relying solely on this linear mechanism.

\begin{Rem}
An analogous ill-posedness result has been obtained for the $2d$ homogeneous hydrostatic system \eqref{eq:Hyd-E} by Han-Kwan and Nguyen in \cite{HKN} (see below in Section \ref{SectionIntro-Difficulties}). Building on this statement, Ibrahim, Lin, and Titi \cite{ILT} essentially extended it to the hydrostatic $2d$ Euler-Boussinesq system \eqref{eqintro:BHyd-E} but with the assumption of initial zero density, i.e., $\varrho_{\mid t=0}=0$. In this case, \eqref{eqintro:BHyd-E} reduces to the homogeneous system \eqref{eq:Hyd-E}, at least for smooth solutions.


\end{Rem}

\begin{Rem}\label{Rmk-illposedlinear}
At the linearized level, the hydrostatic equations \eqref{eqintro:BHyd-E-vort} are indeed ill-posed in Sobolev spaces and in Gevrey spaces (of order greater than $1$). This conclusion is drawn from the linear analysis conducted in Section \ref{Subsec:UnboundedSpecHydrost}, and a precise statement can be found in Theorem \ref{thm:growingmodHydrostat}.
\end{Rem}

\begin{Rem}
In Theorems \ref{thm-invalidHydrostatlim} and \ref{thm-illposednesNonlin}, the stratified shear flow can be explicitly chosen as
\begin{equation*}
    \rho_s(z) = \alpha (1-\alpha)\left( -\beta \tanh\left(\beta z \right)+\frac{\beta}{3}\tanh^3\left(\beta z\right)\right), \qquad 
     U_s(z) = \tanh\left( \beta z\right),
\end{equation*}
with $\beta > 0$ sufficiently large and for some $\alpha < 1$ close to $1$. Notably, Theorems \ref{thm-invalidHydrostatlim} and \ref{thm-illposednesNonlin} hold \textit{in spite of the presence of a stable stratification}: it is easily observed that $\rho_s'(z) < 0$ for all $z \in [-1,1]$, but that the Miles-Howard criterion \eqref{eq:Miles-Howard} is violated.

It remains unclear whether Theorems \ref{thm-invalidHydrostatlim} and \ref{thm-illposednesNonlin} hold for any smooth profile $(\rho_s(z), U_s(z))$ that does not fulfill \eqref{eq:Miles-Howard}. This uncertainty arises because, to our knowledge, the Miles-Howard stability condition is \textit{not} a necessary and sufficient condition for spectral stability. However, our Theorem \ref{thm-invalidHydrostatlim} (resp. Theorem \ref{thm-illposednesNonlin}) does hold for any smooth stratified shear flow such that the conclusion of Theorem \ref{thm:growingmodEB-largetorus} (resp. Theorem \ref{thm:growingmodHydrostat}) is satisfied.
\end{Rem}

To conclude this section, we highlight the fact that, to our knowledge, the well-posedness of the hydrostatic equation \eqref{eqintro:BHyd-E} and the justification of the hydrostatic limit in finite regularity (under additional structural assumptions on the data) is still an open problem.

\subsection{Main difficulties and context}\label{SectionIntro-Difficulties} 
Theorems \ref{thm-invalidHydrostatlim} and \ref{thm-illposednesNonlin} assert that, in general, justifying the hydrostatic limit from \eqref{eqintro:BE-scaled} and establishing the well-posedness of the hydrostatic system \eqref{eqintro:BHyd-E} (in finite regularity) is hopeless. We highlight some of the obstructions below, and compare our results to existing ones in homogeneous fluids and kinetic theory.




\subsubsection*{Cauchy theory for \eqref{eqintro:BHyd-E}}
The equation for the horizontal velocity $u$ in \eqref{eqintro:BHyd-E} exhibits a Burgers-type structure with $(\partial_t + u\partial_x) u$. Hence, without smoothing in the equation, a blow-up in finite time is naturally expected. On the other hand, the vertical velocity $v$ has no equation on its own, and it is recovered through the divergence-free condition a function of $u$ via
\begin{align*}
v(t, x, z) \sim - \partial_z^{-1} \partial_x u(t, x, z),
\end{align*}
that is
\begin{align*}
  v(t, x, z) =L[u](t, x, z):= -\int_{-1}^z \partial_x u(t, x, z') \, \mathrm{d}z'.
\end{align*}
Here, we have used the boundary conditions $v_{\mid z=-1}=0$. Thus, the equation for $u$ can be expressed as the nonlinear transport-like equation of nonlocal type
\begin{align*}
\partial_t u + u \partial_x u + L[u] \partial_z u + \partial_x p = 0.
\end{align*}
A primary challenge arises from the fact that the operator $L$ is not necessarily bounded or skew-symmetric in $\Ld^2$. This hinders the application of standard energy methods to obtain favorable \textit{a priori} estimates on the solution.This issue is similar to what is encountered in Prandtl equations for boundary layers, with distinct boundary conditions and mechanisms arising from partial dissipation (see \cites{GVD, GuoNguyen, GVN, HGV} and references therein). Considering Remark \ref{Rmk-illposedlinear} and \cite{GVD}, we observe that, at least at the linear level, the loss of derivatives is more pronounced in the hydrostatic equations.
Another structural difference with the Prandtl equations is the pressure term in \eqref{eqintro:BHyd-E}. In fact, as $\partial_z p = - \varrho$,
we expect an additional loss of derivative in the horizontal direction because of the term $-\partial_x \partial_z^{-1} \varrho$ in the equation for $u$.
%
At the vorticity level, the hydrostatic Biot-Savart law \eqref{eq:BiotSavartHydrostat} lacks the usual elliptic regularization in $x$, and entails the formal relations  $u=\partial_z^{-1} \omega$ and $v=-\partial_x \partial_{z}^{-2} \omega$. Therefore, from \eqref{eqintro:BHyd-E-vort}, for any $\ell\geq 0$ we cannot expect better than $\Vert (\varrho, \omega) \Vert_{\H^{\ell}    } \lesssim \Vert (\partial_x \varrho, \partial_x \omega ) \Vert_{\H^{\ell}    }$ in terms of energy estimates. As a consequence, an analytic setting is natural to recover this loss of derivative, while a standard Cauchy theory in Sobolev spaces cannot generally be expected.

To the best of our knowledge, the only result regarding the well-posedness of \eqref{eqintro:BHyd-E} (as well as its tridimensional version) is contained in \cite{KTVZ}. This study establishes the existence of a unique analytic solution locally in time for analytic data, employing a Cauchy-Kovalevskaya type theorem. This approach is a standard method to counter the loss of derivatives, as highlighted above. The impact of rotation on the system has been further investigated in \cite{GILT}. Additional insights into linearizations of the inviscid primitive equations with specific boundary conditions can be found in \cites{RTT, HJT}.

Under certain assumptions on the initial data, some ill-posedness \cite{ILT} and blow-up \cites{CINT, collot2023stable} results have been derived, particularly in connection with $2d$ homogeneous dynamics (as discussed below). To overcome the analytic framework, an alternative approach involves considering regularizing viscous effects, such as partial viscosity or diffusivity. 
After the pioneering works \cites{LTW1, LTW2},  global well-posedness for strong solutions has been obtained in \cite{CT2007} (see \cite{LiTiti-review} for an in-depth exploration of these topics). For the mathematical study of a variant known as the hydrostatic  Navier-Stokes system, which mimics the vertical viscosity structure of the Prandtl equations in the hydrostatic context, see \cite{GVMV}.
\subsubsection*{Hydrostatic limit from \eqref{eqintro:BE-scaled}}
Establishing a satisfactory Cauchy theory for \eqref{eqintro:BHyd-E} is generally as challenging as justifying the hydrostatic limit from the system \eqref{eqintro:BE-scaled} when $\eps \rightarrow 0$. Roughly speaking, a difficulty arises due to a hyperbolic change of variables $$\left(t,x,z\right) \mapsto \left(\frac{t}{\eps}, \frac{x}{\eps},z\right)$$ which transforms the Euler-Boussinesq system \eqref{eqintro:BE-scaled} in the hydrostatic regime into the ``standard'' Euler-Boussinesq system \eqref{eqintro:BE-scaled} with $\eps=1$ (scaling the vertical velocity accordingly). As a consequence, obtaining estimates uniform in $\eps$ for the hydrostatic limit as $\eps \rightarrow 0$ on a timescale $\mathcal{O}(1)$ essentially requires proving similar estimates for the unscaled Euler-Boussinesq system on a timescale  $\mathcal{O}(\eps^{-1})$. This suggests that the hydrostatic limit behaves like a long-time dynamics problem for the unscaled Euler-Boussinesq system.
Interpreting the hydrostatic limit in this way implies that the stability or instability properties of the Euler-Boussinesq equations may play a crucial role. Stationary solutions like stratified shear flows, as in \eqref{eq:equilibrium}, form an important class. Besides local well-posedness in Sobolev spaces \cite{Chae}, the question of global well-posedness near stable density profiles like \eqref{eq:equilibrium} remains an open problem
(see \cites{EW, DLS, BBCZD, BHI} for partial results). This may explain the ongoing challenge in justifying the hydrostatic limit in the context of inviscid stratified flows.

Concerning the justification of the hydrostatic limit in the context of viscous fluids, relevant results can be found in \cites{azerad2001mathematical,li2019primitive}. It is worth noting that both theoretical and experimental studies of geophysical flows suggest relying on viscous effects of turbulent (or eddy) nature instead of molecular ones. In this context, to our knowledge the only work addressing the hydrostatic limit and the well-posedness of the limiting equations (in finite regularity) is presented in \cite{BianchDuch}, where suitable isopycnal diffusivity terms are considered, along with the inclusion of so-called Gent-McWilliams correctors.

The Miles-Howard stability criterion \eqref{eq:Miles-Howard} will be central to our analysis, as it appears naturally in  the context of asymptotic stability of stratified shear flows in the Euler-Boussinesq equations \cites{BCZD, BBCZD, BCZD2, CZN1}.


\subsubsection*{The case of homogeneous flows}
The homogeneous counterpart of \eqref{eqintro:BHyd-E}, namely the hydrostatic system with constant density,
\begin{equation}\label{eq:Hyd-E}
\begin{aligned}
\partial_t u +u \partial_x u + v \partial_z u + \partial_x p&=0, \\
\partial_x u + \partial_z v&=0, \\
\partial_z p&=0,
\end{aligned}
\end{equation}
has been extensively studied and shares some of the difficulties with our system \eqref{eqintro:BHyd-E}. The linear and nonlinear Sobolev ill-posedness results \cites{Renardy,HKN}, as well as the failure of the hydrostatic limit \cite{Grenier-hydroderiv}, rely on the existence of unstable shear flows having inflection points. On the other hand, local well-posedness in finite regularity has been established under convexity assumptions on the initial data \cites{Brenier-Hydroexist,MW2012,kukavica2014local}. The justification of the hydrostatic limit with this condition is obtained in \cite{Grenier-hydroderiv} through energy estimates, with a simplified version in \cite{Brenier-Hydroderiv}. Blow-up solutions are discussed in \cites{wong2015blowup, CINT, collot2023stable}. 


From all these works, it is apparent that the study of the hydrostatic limit for \eqref{eq:Hyd-E} is intimately connected to the hydrodynamic stability of shear flows without inflection points, as they are spectrally stable in view of the Rayleigh criterion \cite{drazin1966hydrodynamic}. This observation motivates the connection of the hydrostatic limit and the long-time stability of the non-hydrostatic equation, as discussed above.

\subsubsection*{Connection with the quasineutral limit} The difficulties encountered in the hydrostatic approximation are very similar to those involved in the quasineutral limit from plasma physics in kinetic theory. This analogy has been pointed out in \cites{Zak, BrenierComp, HKH}. Roughly speaking, the quasineutral regime amounts to consider a typical length of electrostatic interactions (so-called Debye length) that is very small. For electrons (in a fixed background of ions) described by the classical Vlasov-Poisson equation, this approximation formally leads to the so-called kinetic Euler equations (named after Brenier \cites{BreVP1,Bre-Genflow}).

The quasineutral approximation shares similarities with the hydrostatic limit, displaying a singular nature, with some analoguous losses of derivatives in the estimates. Stability issues related to a suitable hyperbolic scaling are governed by the Penrose stability condition \cite{Penrose}, influencing the validity of the quasineutral limit and the well-posedness of the limiting equations. Conjecturally, the Miles-Howard condition \eqref{eq:Miles-Howard} plays a similar crucial role for stratified fluids, supported by our Theorems \ref{thm-invalidHydrostatlim} and \ref{thm-illposednesNonlin}.

Research on the quasineutral limit problem is actively pursued, and a more detailed overview is available in \cite{griffin2021recent}. Rigorous justification of the quasineutral limit began among others with \cites{Gr96,B00,Mas}, primarily in an analytic or monokinetic regime (see also \cites{HKI,HKI2}). Structural assumptions of stability on the initial data, in the sense of Penrose, lead to positive results in Sobolev regularity \cite{HKR2}. Conversely, around a homogeneous profile satisfying an instability condition, the limit does not hold \cite{HKH}, and the kinetic Euler equations are ill-posed in Sobolev spaces \cites{HKN,baradat2020nonlinear}. In the ion case, analogous results are discussed in \cites{JN, HK-quasin, BB1,HKR}.

\subsection{Strategy of the proof and organization of the paper}\label{SectionIntro-Outline}

In Section \ref{Section:GrowingMode}, we first investigate linearizations of the hydrostatic and non-hydrostatic Euler-Boussinesq equations, around the stratified shear flows \eqref{eq:equilibrium}.
We essentially aim at finding a suitable smooth profile such that the corresponding linearized operators (in density-vorticity variables) possesses a non-trivial unstable spectrum. 

Our linear analysis provides the leading mechanism driving the nonlinear dynamics. For our purpose, the guiding lines are the following:
\begin{itemize}
    \item in the hydrostatic case, we obtain an unbounded unstable spectrum that may grow linearly in frequency (Theorem \ref{thm:growingmodHydrostat} in  Section \ref{Subsec:UnboundedSpecHydrost}). This is related to the linear (and subsequent nonlinear) ill-posedness of the equations in Sobolev spaces.
    \item in the non-hydrostatic case and for the unscaled Euler-Boussinesq, we exhibit a maximal unstable eigenvalue related to a long-wave instability in a long periodic channel, see Theorem \ref{thm:growingmodEB-largetorus} in  Section \ref{Subsec:LongwaveInstab}. Due to some space-time scaling, this will lead to instabilities in the hydrostatic limit that may develop on a time scale $\mathcal{O}(\eps^\alpha)$, for any $\alpha \in (0,1)$.
\end{itemize}
By looking at transversal pure-modes that are solutions of the linearized dynamics around $(\rho_s(z), U_s(z),0)$, i.e. proportional to $\e^{ikx} \e^{-ikct}$ with $(k,c) \in \Z \setminus {\lbrace 0 \rbrace} \times \C$, we can see that potential instabilities are governed by the Taylor-Goldstein equation \cites{miles1961stability, howard1961note}:  
\begin{equation}\label{eq:introTG}
      \begin{aligned}
(U_s-c)^2\left(\partial_z^2-(\eps k) ^2\right) \varphi-(U_s-c)U_s'' \varphi&=\rho_s' \varphi, \\
\varphi(\pm1)&=0.
\end{aligned}
\end{equation}
This is an eigenvalue problem on the associated stream function $\varphi(z)$, having the same role as the Rayleigh equations for homogeneous flows (i.e. for $\rho_s'=0$). Seeking for instabilities corresponds to searching for a pair $(c, \varphi_c)$ with  $\mathrm{Im}(c)\neq 0$ and $\varphi_c \neq 0$ a solution to \eqref{eq:introTG}.  Here, the case $\eps=1$ brings to the original (unscaled) Euler-Boussinesq equations, while $\eps \to 0$ refers to the hydrostatic regime. 
In the hydrostatic case $\eps=0$,  \eqref{eq:introTG} does not depend on the transversal frequency $k$, so the exponential growth rate in time of solutions is proportional to $k \Im(c)$, with $c$ independent of $k$.
%
%
%

Following an idea of Friedlander in \cite{friedlander2001nonlinear}, we introduce a parameter $\alpha \in (1/2,1)$ and impose that \eqref{eq:equilibrium} satisfies
\begin{align}
    \dfrac{-\rho_s'(z)}{|U_s'(z)|^2}=\alpha(1-\alpha).
\end{align}
In turn, this enforces the violation of 
the Miles-Howard stability condition \eqref{eq:Miles-Howard}. The limit $\alpha \rightarrow 1$ corresponds to a degenerate homogeneous flow. By performing the change of unknowns 
\begin{align}
    \varphi=(U_s-c)^\alpha \psi_{\eps, \alpha},
\end{align}
we can transform the equation \eqref{eq:introTG} into the new eigenvalue problem
\begin{equation}\label{eq:introTGnew}
      \begin{aligned}
\partial_z \big((U_s-c)^{2\alpha} \partial_z \psi_{\eps, \alpha} \big)=(U_s-c)^{2\alpha}\left(\eps k\right)^2 \psi_{\eps, \alpha} +(\alpha-1)(U_s-c)^{2\alpha-1}U_s'' \psi_{\eps, \alpha}&=0, \\
\psi_{\eps, \alpha}(\pm1)&=0.
\end{aligned}
\end{equation}
When $(\eps,\alpha)=(0,1)$, \eqref{eq:introTGnew} drastically simplifies to 
\begin{align*}
    \partial_z \big((U_s-c)^{2} \partial_z \psi_{0,1} \big)=0.
\end{align*}
This corresponds to the eigenvalue problem associated to the homogeneous hydrostatic Euler equations \eqref{eq:Hyd-E} linearized around $U_s$, studied in \cite{Renardy}. Integrating in $z$ and taking into account the boundary condition $\psi_{0,1}(-1)=0$, it yields a suitable growing mode provided that there exists a $c \in \C$ with $\mathrm{Im}(c) \neq 0$ satisfying the dispersion relation 
\begin{align}
        \int_{-1}^1 \frac{\mathrm{d}z}{(U_s(z)-c)^2}=0.
\end{align}
This former dispersion relation, if satisfied, ensures that the boundary condition $\psi_{0,1}(1)=0$ holds.
Drawing inspiration from Penrose's complex analysis approach \cite{Penrose} for instability in Vlasov-Poisson systems, we can actually find a smooth shear-layer profile $U_s$ to solve the aforementioned equation (see Appendix \ref{Appendix:ZeroFunction}).

The full problem \eqref{eq:introTGnew} is handled through a perturbative strategy inspired by \cites{DDGM, HG} in the context of boundary-layer models. In our case, the homogeneous hydrostatic instability persists for $(\eps,\alpha)$ close to $(0,1)$. The procedure is based on Rouché theorem and essentially amounts to find a good zero $c \in \lbrace \mathrm{Im}> 0 \rbrace$ of the dispersion relation  $c \mapsto \psi_{\eps, \alpha}(1)$ (so that the boundary condition $\psi_{\eps, \alpha}(1)=0$ holds), near the zero of the function $c \rightarrow \psi_{0, 1}(1)$. It is performed thanks to explicit stability estimates in the regime $(\eps,\alpha) \rightarrow (0,1)$.


For the sake of clarity, we first treat the case of $\eps=0$, in the limit $\alpha \rightarrow 1$ in Section \ref{Subsec:UnboundedSpecHydrost}, for the hydrostatic equation. In Section \ref{Subsec:LongwaveInstab}, and after a suitable rescaling in $\eps$ as described in Section \ref{SectionIntro-Difficulties}, we consider the case $\eps=1$ with $k$ replaced by $k/M$ (corresponding to a large box limit), in the regime $M \rightarrow + \infty$ and $\alpha \rightarrow 1$.  

\medskip

Section \ref{Section:HydroLimit} is dedicated to the proof of Theorem \ref{thm-invalidHydrostatlim}. We rely on the linear instability devised in Section \ref{Section:GrowingMode}, and take inspiration from the analogies appearing in the disproval of the derivation of the kinetic Euler equation from the Vlasov-Poisson system \cite{HKH}. Their method is itself based on the celebrated Grenier's scheme for nonlinear instability \cite{Gre-instabScheme}, here used after a rescaling in $\eps$. We crucially use the idea that the regime $\eps \rightarrow 0$ is equivalent to a large box limit, where instabilities can be found. We also deduce a nonlinear Lyapunov instability results for the unscaled Euler-Boussinesq equations.

\medskip

In Section \ref{Section:IllPosed}, we prove Theorem \ref{thm-illposednesNonlin}. Our strategy relies on the abstract framework developed in \cite{HKN}, providing a robust setting to prove mild ill-posedness in finite regularity for some nonlocal equations having unbounded unstable spectrum. In view of the scaling invariance of \eqref{eqintro:BHyd-E-vort}, the main issue consists in studying the growth of the associated spectrum and its impact on the semigroup estimates.

\section{Construction of linear instabilities}\label{Section:GrowingMode}
We proceed to the linear analysis of the hydrostatic and non-hydrostatic Euler-Boussinesq equations around an unstable stratified shear \eqref{eq:equilibrium}. Our goal is twofold:
\begin{enumerate}
\item In Section \ref{Subsec:UnboundedSpecHydrost}, we construct a growing mode for the linearization of \eqref{eqintro:BHyd-E} for any non-zero $k \in \Z$,  which is proportional to $\e^{\lambda_k t}$, with
$\mathrm{Re}(\lambda_k)=\sigma k$ and $\sigma>0$ independent of $k$. This reveals that the linearized operator exhibits an unbounded unstable spectrum, a key point for Section \ref{Section:IllPosed}, where we delve into the ill-posedness of the system \eqref{eqintro:BHyd-E-vort}. The analysis of the eigenvalue problem is conducted through a perturbative approach, relying on explicit computations related to the homogeneous hydrostatic equations.
 
\item In Section \ref{Subsec:LongwaveInstab}, we provide a long-wave instability for the linearized unscaled system \eqref{eqintro:BE-scaled} (i.e., with $\eps=1$). Our perturbative proof follows the approach developed in Section \ref{Subsec:UnboundedSpecHydrost}. This unstable eigenvalue has a crucial role in Section \ref{Section:HydroLimit}, where we prove the failure of the hydrostatic limit and a nonlinear Lyapunov instability of the hydrostatic equations.

\end{enumerate}

\subsection{Unstable unbounded spectrum for the hydrostatic Euler-Boussinesq equations}\label{Subsec:UnboundedSpecHydrost}
We begin the (linear) analysis by linearizing system \eqref{eqintro:BHyd-E} around a stratified shear equilibrium $(\rho_s(z), \bcU(z), 0, p_s(z))$ as given in \eqref{eq:equilibrium}. To this end, we consider the ansatz  
\begin{align*}
\varrho(t,x,z)&=\rho_s(z)+\widetilde{\varrho}(t,x,z), \\
p(t,x,z)&=p_s(z)+\widetilde{p}(t,x,z), \qquad p_s'(z)=-\rho_s(z),\\
u(t,x,z)&=U_s(z)+\widetilde{u}(t,x,z), \\
v(t,x,z)&=0+\widetilde{v}(t,x,z),
\end{align*}
and, plugging it into \eqref{eqintro:BHyd-E}, we obtain the linearized system 
\begin{equation}\label{eq:hEBlinearized}
      \begin{cases}
      \partial_t \widetilde{\varrho} +U_s(z) \partial_x \widetilde{\varrho}+ \widetilde{v} \rho_s'(z)=0, \\ 
\partial_t \widetilde{u} + U_s(z) \partial_x \widetilde{u} + \widetilde{v} U_s'(z) + \partial_x \widetilde{p}=0,  \\ 
\partial_x \widetilde{u} + \partial_z \widetilde{v}=0, \\ 
\partial_z \widetilde{p}+ \widetilde{\varrho} =0.
\end{cases}
\end{equation}
From now on, we omit the tilde to simplify notations.  

As done in \cite{MW2012} for the \emph{homogeneous} hydrostatic Euler equations, we express the divergence-free velocity field in terms of a stream function $\varphi$, which is related to the ``vorticity'' $\omega=\partial_z u$ through the degenerate elliptic equation  
\begin{align}
u=\partial_z \varphi, \qquad  v = -\partial_x \varphi, \qquad  \begin{cases}
        \partial_z^2 \varphi=\omega,\\
        \varphi_{\mid z=\pm 1}=0.
    \end{cases}
\end{align}
This formulation allows to rewrite \eqref{eq:hEBlinearized} as 
\begin{equation}\label{eq:Hydrost-Linearized}
      \begin{cases}
      \partial_t \varrho +U_s(z) \partial_x \varrho - \rho_s'(z) \partial_x \varphi=0, \\ 
\partial_t \omega + U_s(z) \partial_x \omega - U_s''(z)\partial_x \varphi = \partial_x \varrho.   
\end{cases}
\end{equation}
Note that when $\rho_s'=0$ and $\varrho$ is constant, we recover the \emph{homogeneous} linearized hydrostatic equation (see \cite{Renardy}):
\begin{align}\label{eq:Hydrost-Linearized-homo}
    \partial_t \omega + U_s(z) \partial_x \omega - U_s''(z)\partial_x \varphi=0.
\end{align}
The system \eqref{eq:Hydrost-Linearized} above can be compactly expressed as 
\begin{align*}
\partial_t \begin{pmatrix}
\varrho\\
\omega
\end{pmatrix}
=\mathscr{L}_{\rho_s, U_s}\begin{pmatrix}
\varrho\\
\omega
\end{pmatrix},
\end{align*}
where the linearized hydrostatic operator
 $\mathscr{L}_{\rho_s, U_s}$ is defined as
\begin{align}\label{def:LinOp-Hydrostat}
\mathscr{L}_{\rho_s, U_s}\begin{pmatrix}
\varrho\\
\omega
\end{pmatrix}=\begin{pmatrix}
-U_s(z) \partial_x \varrho + \rho_s'(z)\partial_x \varphi\\ -U_s(z) \partial_x \omega + U_s''(z)\partial_x \varphi+ \partial_x \varrho
\end{pmatrix}, \qquad \text{with} \qquad  \begin{cases}
        \partial_z^2 \varphi=\omega, \\ 
        \varphi_{\mid z=\pm 1}=0.
    \end{cases}
\end{align}
Our \emph{linear} ill-posedness result is provided below.
\begin{Thm}\label{thm:growingmodHydrostat}
There exists an analytic stationary profile $(\rho_s(z), U_s(z))$, with $\rho_s'(z)<0$, satisfying
\begin{align*}
0< \frac{-\rho_s'(z)}{|U_s'(z)|^2} < \frac{1}{4},
\end{align*}
and such that the following holds. There exists $c \in \lbrace \mathrm{Im}>0 \rbrace$ such that for any $k \in \Z {\setminus \lbrace 0 \rbrace}$, the hydrostatic linearized equations \eqref{eq:Hydrost-Linearized} exhibit a real-valued growing mode of the form 
\begin{equation}\label{eq:growingmode}
 \begin{aligned}
\varrho^{\mathrm{mod}}(t,x,z)=r(z)  \e^{k \Im (c) t} \cos(k(x-\Re(c) t)-\theta_1(z)), \\
\omega^{\mathrm{mod}}(t,x,z)=w(z) \e^{k \Im(c) t} \cos(k(x-\Re(c) )t-\theta_2(z)),
 \end{aligned}
\end{equation}
where $r, w, \theta_1, \theta_2$ are smooth and real-valued nonzero functions.  

Moreover, the unstable spectrum of $\mathscr{L}_{\rho_s, U_s}$ is unbounded and contains  all the eigenvalues $\lambda=-ikc$ with $k \in \Z$ such that $\mathrm{Re}(\lambda)=k \mathrm{Im}(c)>0$, associated to some smooth eigenmodes of the form $\e^{ikx}(\widehat{\varrho}(z), \widehat{\omega}(z))$ for some non-zero regular functions $(\widehat{r}(z), \widehat{w}(z))$.
\end{Thm}

Let us begin the proof, which unfolds in several steps.

\subsubsection*{ Step 1 - Towards the hydrostatic Taylor-Goldstein equation}
We perform a standard normal mode analysis looking for a solution of the form
\begin{align}\label{eq:mode}
\varrho^{\mathrm{mod}}(t,x,z)=r(z) \e^{ik(x-ct)}, \qquad  \varphi^{\mathrm{mod}}(t,x,z)=\varphi(z) \e^{ik(x-ct)},
\end{align}
with $k \in \Z {\setminus \lbrace 0 \rbrace}$, $t >0$, $c \in \C$ and $r,\varphi \in \C$ (non zero), with $\mathrm{Im}(c) \neq 0$. Here, $\varphi$ is the associated hydrotatic stream function related to $w$ in \eqref{eq:growingmode} through \eqref{def:LinOp-Hydrostat}. We will provide a real growing mode by taking the real part of \eqref{eq:mode}.

\begin{Rem}
Note that if such a solution exists, it immediately yields an unstable eigenfunction for the linearized hydrostatic operator $\mathscr{L}_{\rho_s, U_s}$ in \eqref{def:LinOp-Hydrostat}, associated to the eigenvalue $\sigma_k:=-ikc$.
\end{Rem}

Plugging \eqref{eq:mode} into \eqref{eq:Hydrost-Linearized} yields
\begin{equation*}
      \begin{aligned}
      (-ikc+ikU_s) r-ik  \rho_s'(z) \varphi &=0, \\[1mm]
(-ikc+ikU_s)\partial_z^2 \varphi-ik\varphi U_s''&=ik r.\end{aligned}
\end{equation*}
Note that the equations are \emph{homogeneous} in $k$, reflecting the (hyperbolic) space-time scaling of the hydrostatic equations, which is in contrast with the Euler-Boussinesq system \eqref{eqintro:BE-scaled}. Simplifying the $k$'s,
\begin{equation*}
      r=\frac{\rho_s'(z) \varphi}{U_s-c}, \qquad
(U_s-c)\partial_z^2 \varphi-U_s'' \varphi= \frac{1}{U_s-c}\rho_s'(z) \varphi.
\end{equation*}
Multiplying by $\bcU-c$ in the system above, we obtain the hydrostatic Taylor-Goldstein equation
\begin{equation}\label{eq:TaylorGoldstein}
(U_s-c)^2\partial_z^2 \varphi-(U_s-c)U_s'' \varphi= \rho_s'(z) \varphi,\\
\qquad  \varphi(\pm 1)=0.
\end{equation}
Note that solutions to this one-dimensional problem are regular under the condition $\mathrm{Im}(c)>0$, by standard elliptic regularity. 
Moreover, a crucial observation is that when $\rho_s'(z)=0$ (constant density profile), the latter gives the standard (hydrostatic) Rayleigh equation
\begin{align}\label{eq:RayleighHydrostatic}
          (U_s-c) \partial^2_z \varphi-U_s'' \varphi=0, 
\end{align}
which has been studied in \cite{Renardy}.
 %
In this homogeneous case, the existence of a growing mode for the linearization near $(\varrho, u,v)=(1,U_s(z),0)$ is known, at least for a specific unstable shear flow $\bcU (z)$, see \cites{Renardy,HKN}.
%

A necessary condition to identify a growing mode for the linearized hydrostatic equation \eqref{eq:Hydrost-Linearized} is given below.
\begin{Lem}\label{lem:necessCond-growingmod}
The following holds:
\begin{enumerate} [label=(G\arabic*), ref=(G\arabic*)]
\item\label{item:G1} If a growing mode as \eqref{eq:growingmode} for the  linearized hydrostatic equations \eqref{eq:Hydrost-Linearized} exists, then $\mathrm{Re}(c) \in U_s([-1,1])$.

\item \label{item:G2} If a growing mode for the homogeneous  linearized hydrostatic equations \eqref{eq:Hydrost-Linearized-homo} exists, then $\mathrm{Re}(c) \in U_s([-1,1])$.

\item \label{item:G3} If the profile $(\rho_s(z), U_s(z))$ is stable in the sense of Miles-Howard as in \eqref{eq:Miles-Howard}, then the  linearized hydrostatic equations \eqref{eq:Hydrost-Linearized} do not have any growing mode.
\end{enumerate}
\end{Lem}
\begin{proof}
Let us assume the existence of a growing mode, namely a non-trivial solution of the hydrostatic Taylor-Goldstein equation \eqref{eq:TaylorGoldstein} with $\mathrm{Im}(c) \neq 0$. To obtain \ref{item:G1}, we first observe that the new unknown $\phi$ defined by 
\begin{align*}
\varphi = (U_s-c) \phi
\end{align*}
satisfies the equation
\begin{align*}
    (U_s-c)^2 \partial_z^2 \phi+2 (U_s-c)U_s' \partial_z \phi=\rho_s\phi,
\end{align*}
which rewrites as
\begin{align*}
    \partial_z\left[(U_s-c)^2 \partial_z \phi \right]=\rho_s\phi. 
\end{align*}
Testing against $\overline{\phi}^*$,  integrating by parts and using the boundary condition gives
\begin{align*}
    \int_{-1}^1  (U_s-c)^2 \vert \partial_z \phi \vert^2=-\int_{-1}^1 \rho_s\vert \phi \vert^2.
    \end{align*}
Taking the imaginary part of the previous identity,
\begin{align*}
    2 \mathrm{Im}(c)\int_{-1}^1  (U_s-\mathrm{Re}(c)) \vert \partial_z \phi \vert^2=0,
\end{align*}
which implies the first statement. The proof of \ref{item:G2} is identical. To prove \ref{item:G3}, we rely on the classical trick of Miles, see \cites{miles1961stability, howard1961note, drazin1966hydrodynamic, BCZD2}, which consists in defining a new unknown $\psi$ through
\begin{align*}
\varphi = (U_s-c)^\frac12 \psi.
\end{align*}
The new stream-function $\psi$ satisfies the equation 
\begin{align*}
\partial_z\left[(U_s-c) \partial_z \psi \right]+\left[\frac{1}{2}U_s''- \frac{1}{U_s-c}\left( \frac{|U_s'|^2}{4}+ \rho_s' \right)\right]\psi=0.
\end{align*}
Testing against the complex conjugate $\overline{\psi}^*$ in the interval $(-1,1)$, integrating by parts and using the boundary condition on $\psi$ gives
\begin{align*}
    \int_{-1}^1 (U_s-c) \vert \partial_z \psi \vert^2 =\int_{-1}^1 \left[\frac{1}{2}U_s''- \frac{1}{U_s-c}\left( \frac{|U_s'|^2}{4}+ \rho_s' \right)\right] \vert \psi \vert^2.
\end{align*}
Let us now take the imaginary part of the previous identity, which yields
\begin{align*}
    -\mathrm{Im}(c) \int_{-1}^1  \vert \partial_z \psi \vert^2 =-\mathrm{Im(c)}\int_{-1}^1 \frac{1}{\vert U_s-c \vert^2 }\left( \frac{|U_s'|^2}{4}+\rho_s' \right) \vert \psi \vert^2.
\end{align*}
Owing to the Miles-Howard stability condition \eqref{eq:Miles-Howard}, note that the parenthesis in the right-hand side is non-positive. As $\mathrm{Im}(c)\neq 0$, then $\partial_z \psi=0$ and  $\psi=0$ thanks to the boundary condition. We obtain a contradiction, which ends the proof.
\end{proof}
As a consequence, to identify a growing mode, we must necessarily consider a profile $(\rho_s, U_s)$ that does not satisfy the Miles-Howard stability condition.

\subsubsection*{Step 2 - Change of unknown}
Following Friedlander \cite{friedlander2001nonlinear}, let us now assume that our profile $(\rho_s, U_s)$ satisfies 
\begin{align}\label{def-profilInstab}
-\rho_s'(z)=\alpha (1-\alpha) |U_s'(z)|^2, \qquad \alpha \in (0,1), \quad \alpha \neq 1/2.
\end{align}
This enforces the failure of the Miles-Howard stability criterion \eqref{eq:Miles-Howard}. 
Given the solution $\varphi$ to \eqref{eq:TaylorGoldstein}, follo\-wing \cites{howard1961note, friedlander2001nonlinear}, we set
\begin{align*}
\varphi = (U_s-c)^{\alpha} \psi.
\end{align*}
The new unknown $\psi$ satisfies
\begin{align*}
(U_s-c) \partial_z^2 \psi+2\alpha U_s' \partial_z \psi+\left[(\alpha-1)U_s''+ \frac{\alpha(\alpha-1) |U_s'|^2}{U_s-c}- \frac{\rho_s'}{U_s-c}  \right]\psi=0,
\end{align*}
which, using \eqref{def-profilInstab}, simplifies as
\begin{align*}
(U_s-c) \partial_z^2 \psi+ 2 \alpha U_s' \partial_z \psi +(\alpha-1)U_s'' \psi=0,
\end{align*}
or, equivalently, 
\begin{align}\label{ODE:tosolve}
\partial_z \big((U_s-c)^{2\alpha} \partial_z \psi \big)+(\alpha-1)(U_s-c)^{2\alpha-1}U_s'' \psi=0, 
     \qquad \psi(\pm1)=0.
\end{align}
%

\begin{Rem}\label{Rmk-BacktoHomogeneous}
The case $\alpha=1$ precisely corresponds to $\rho_s'(z)=0$ and leads to the Rayleigh equation \eqref{eq:RayleighHydrostatic}, which can be expressed as
\begin{align}
\partial_z \big((U_s-c)^{2} \partial_z \psi \big)=0,
\qquad \psi(\pm1)=0.
\end{align}
This equation arises for the same reasons within the context of the homogeneous hydrostatic Euler equations \eqref{eq:Hyd-E} (see \cites{Renardy,HKN}). Choosing a constant density profile, i.e., $\rho_s'(z)=0$, leads us to the homogeneous case as a particular instance of the non-homogeneous one. Notably, the subsequent analysis fully covers the case $\alpha=1$ (with the appropriate choice of the shear flow $U_s$), and thus, we recover the result of \cite{Renardy}.
\end{Rem}


To establish the existence of a \textit{genuinely} stratified unstable profile (i.e., with $\rho_s' \neq 0$), we employ a perturbative approach as $\alpha \rightarrow 1$, viewing non-homogeneity as a perturbation of homogeneous fluids. This draws inspiration from \cites{DDGM, HG}, where the authors rely on a similar perturbative method to identify unstable modes in boundary layer models (such as interactive boundary layers and triple decks) by treating viscous modes as a perturbation of inviscid ones.

\subsubsection*{Step 3 - Integral formulation}
We want to find a pair $(c,\psi)$ with $\mathrm{Im}(c) \neq0$ satisfying \eqref{ODE:tosolve}.  Integrating and using the boundary condition $\psi(-1)=0$, we look for $(c,\psi)$ with $\mathrm{Im}(c) \neq0$ such that 
\begin{align}\label{eq:IntegroDiffTosolve}  
\psi(z)&=\int_{-1}^z\frac{\mathrm{d}r}{(U_s(r)-c)^{2\alpha}}+(1-\alpha)\int_{-1}^z\frac{\mathrm{d}r}{(U_s(r)-c)^{2\alpha}}\left\lbrace \int_{-1}^r (U_s(q)-c)^{2\alpha-1}U_s''(q)\psi(q)  \mathrm{d}q\right\rbrace,
\end{align}
with
\begin{align}\label{eq:DispersionRelTosolve}
\psi(1)=0.
\end{align}
It is important to note that if we identify such a function $\psi$, it cannot be the zero function. The reason being that this would imply that
\begin{align*}
\forall z \in (-1,1), \qquad \int_{-1}^z\frac{\mathrm{d}r}{(U_s(r)-c)^{2\alpha}}=0, \quad \text{and}\quad \frac{1}{(U_s(r)-c)^{2\alpha}}=0 \quad \forall r \in (-1,1),
\end{align*}
which is not possible.

Our main idea is as follows: when $\alpha=1$, the integral equation \eqref{eq:IntegroDiffTosolve} simplifies to the form
\begin{align}
    \psi(z)=\int_{-1}^z\frac{\mathrm{d}r}{(U_s(r)-c)^{2}}, \quad \text{with} \quad \psi(1)=0.
\end{align}
This problem admits a solution provided that the relation
\begin{align}
    \int_{-1}^1 \frac{\mathrm{d}r}{(U_s(r)-c)^{2}}=0
\end{align}
holds, and this is precisely the dispersion relation obtained by Renardy in \cite{Renardy} for the case of hydrostatic homogeneous flows (see Remark \ref{Rmk-BacktoHomogeneous}). In Appendix \ref{Appendix:ZeroFunction}, we briefly revisit an example of a smooth shear flow satisfying the given dispersion relation (distinct from the non-smooth ones provided by \cite{Renardy}), although the exact shape is not crucial for our analysis. Leveraging this profile, we can solve the integral equation \eqref{eq:IntegroDiffTosolve} for $\alpha$ close to $1$.

Let us now set our perturbative argument. For $\alpha \in (1/2,1)$ and $c \in \lbrace \mathrm{Im}>0 \rbrace$, we define the operator
\begin{align}\label{def:Op-Salpha}
\mathrm{S}_{\alpha,c}[\Phi](z):=\int_{-1}^z\frac{\mathrm{d}r}{(U_s(r)-c)^{2\alpha}}\left\lbrace \int_{-1}^r (U_s(q)-c)^{2\alpha-1}U_s''(q)\Phi(q)  \mathrm{d}q\right\rbrace,
\end{align}
and the function
\begin{align}\label{def:fctDalpha}
D_{\alpha}^z(c): z \mapsto D_{\alpha}^{z}(c):=\int_{-1}^z\frac{\mathrm{d}r}{(U_s(r)-c)^{2\alpha}}, \qquad z \in (-1,1).
\end{align}
In view of \eqref{eq:IntegroDiffTosolve}, we look for a function $\Phi=\Phi_{\alpha,c}$, vanishing at $z=1$ and satisfying
\begin{align*}
\Phi(z)=\int_{-1}^z\frac{\mathrm{d}r}{(U_s(r)-c)^{2\alpha}}+(1-\alpha)\mathrm{S}_{\alpha,c}[\Phi](z), \qquad  \text{with} \qquad \Phi(1)=0,
\end{align*}
at least for $\alpha$ close to $1$. We would like to formally expand the former equation as
\begin{align*}
\Phi(z)&=\big( \mathrm{I}-(1-\alpha)\mathrm{S}_{\alpha,c} \big)^{-1}[D^z_{\alpha}(c)]=D_{\alpha}^z(c)+\sum_{n \geq 1}(1-\alpha)^n\mathrm{S}_{\alpha,c}^n[D^z_{\alpha}(c)], \qquad 
\text{with }\ \Phi(1)=0.
\end{align*}
To prove this, we must study the operator $\mathrm{S}_{\alpha,c}$. In what follows, the complex number $c \in \lbrace \mathrm{Im}>0 \rbrace$ will be treated as a parameter.  We have the following statement.
\begin{Lem}\label{lem-opnormIntegOp}
Assume that the profile $U_s$ is smooth and strictly monotone. For all $\alpha \in (1/2,1)$ and $c \in \lbrace \mathrm{Im}>0 \rbrace \cap \lbrace \mathrm{Re} \in U_s([-1,1]) \rbrace$, the operator $\mathrm{S}_{\alpha,c}$ is bounded on the Banach space$$Y=\lbrace f \in \mathrm{W}^{1,\infty}(-1,1) \mid f(-1)=0 \rbrace$$ endowed with the Lipschitz norm, with the estimate
\begin{align}\label{ineq:BoundOpT}
\Vert \mathrm{S}_{\alpha,c} \Vert_{Y \rightarrow Y} \leq C  \left( 1+ \frac{4\| U_s\|_{\Ld^\infty}^2}{\vert \mathrm{Im}(c) \vert^{2}}\right)^{\alpha} \left\Vert\frac{U_s''}{U_s'} \right\Vert_{\mathrm{W}^{1,\infty}},
\end{align}
for some universal constant $C>0$ independent of $\alpha$ and $c$.
\end{Lem}
\begin{proof}
Following \cite{zhao2018}*{Definition 4.3}, we can write the above $\mathrm{S}_{\alpha,c}$ as follows
\begin{equation}
\mathrm{S}_{\alpha,c} =T_0 \circ T_{2\alpha, 2\alpha-1},
\end{equation}
where 
\begin{align}
T_0 f(y):= \int_{-1}^y f(z) \, \mathrm{d}z , \qquad T_{m, n}f(y):= \frac{1}{(U_s(y)-c)^{m}}  \int_{-1}^{y} f(z) (U_s(z)-c)^{n} U_s''(z) \, \mathrm{d}z.
\end{align}
Note that $\partial_z \mathrm{S}_{\alpha,c} =T_{2\alpha, 2\alpha-1}$. For $T_0$ above, we directly obtain 
\begin{align}
\Vert T_0 f\Vert_{Y} \le \Vert f\Vert_{Y},
\end{align}
for any $f \in Y$. 
Next, for any $f \in Y$, we have
\begin{align}
\Vert \mathrm{S}_{\alpha,c} f\Vert_{Y}&=\Vert (T_0\circ T_{2\alpha, 2\alpha-1}) f\Vert_{Y} \le  2\Vert T_{2\alpha, 2\alpha-1} f\Vert_{\infty}.
\end{align}
We observe that 
\begin{align*}
T_{2\alpha, 2\alpha-1} f(y) &=  \frac{1}{(U_s(y)-c)^{2\alpha}}  \int_{-1}^{y} f(z) (U_s(z)-c)^{2\alpha-1} U_s''(z) \, \mathrm{d}z \\
&=   \frac{1}{(U_s(y)-c)^{2\alpha}}  \int_{-1}^{y} f(z)  \frac{(U_s(z)-c)^{2\alpha-1} U_s'(z)}{U_s'(z)} U_s''(z) \, \mathrm{d}z \\
&=   \frac{(2\alpha)^{-1}}{(U_s(y)-c)^{2\alpha}}  \int_{-1}^{y} f(z) {\partial_z \big((U_s(z)-c)^{2\alpha} \big)} \frac{U_s''(z)}{U_s'(z)} \, \mathrm{d}z,
\end{align*}
hence by integration by parts
\begin{align*}
T_{2\alpha, 2\alpha-1} f(y) &=  \frac{(2\alpha)^{-1}}{(U_s(y)-c)^{2\alpha}}  \left[ f(z) (U_s(z)-c)^{2\alpha} \frac{U_s''(z)}{U_s'(z)}\right]_{z=-1}^{z=y} \\
& \quad -  \frac{(2\alpha)^{-1}}{(U_s(y)-c)^{2\alpha}} \int_{-1}^{y} f'(z) (U_s(z)-c)^{2\alpha} \frac{U_s''(z)}{U_s'(z)}  \, \mathrm{d}z  \\
& \quad - \frac{(2\alpha)^{-1}}{(U_s(y)-c)^{2\alpha}}  \int_{-1}^{y} f(z)  (U_s(z)-c)^{2\alpha} \partial_z \left(\frac{U_s''(z)}{U_s'(z)}\right) \, \mathrm{d}z   \\
&:= \text{(I)}+\text{(II)}+\text{(III)}.
\end{align*}
Note that as $ \alpha \in (1/2,1)$, then $(2\alpha)^{-1} \leq 1$. Let us now treat each term separately.
For \text{(I)}, we use the condition $ f(-1)=0$ 
to find
\begin{align*}
\text{(I)}&=-\frac{(2\alpha)^{-1}}{(U_s(y)-c)^{2\alpha}}   f(y) (U_s(y)-c)^{2\alpha} \frac{U_s''(y)}{U_s'(y)},
\end{align*}
which implies
\begin{align*}
\vert \text{(I)} \vert  \leq \sup_{y \in [-1,1]}\left|\frac{U_s''(y)}{U_s'(y)} \right| \Vert f \Vert_{\Ld^\infty}.
\end{align*}
The terms arising from \text{(II)} and \text{(III)} are analogous and will be treated in a similar manner. Let us focus on \text{(II)}. We have
\begin{align*}
\vert \text{(II)} \vert & \leq   \frac{1}{\vert U_s(y)-c \vert ^{2\alpha}}\left\vert \int_{-1}^{y}  f'(z)(U_s(z)-c)^{2\alpha} \frac{U_s''(z)}{U_s'(z)}  \, \mathrm{d}z\right\vert \\
& \leq \Vert f \Vert_Y \sup_{y \in [-1,1]}\left|\frac{U_s''(y)}{U_s'(y)} \right| \frac{1}{\vert U_s(y)-c \vert ^{2\alpha}} \int_{-1}^{y}  \vert U_s(z)-c \vert^{2\alpha}  \, \mathrm{d}z  \\
& \leq \Vert f \Vert_Y \sup_{y \in [-1,1]}\left|\frac{U_s''(y)}{U_s'(y)} \right| \frac{\sup_{z \in [-1,y]} \vert U_s(z)-c \vert ^{2\alpha}}{\vert U_s(y)-c \vert ^{2\alpha}}.
\end{align*}
Now note that, as $\mathrm{Re}(c) \in \left[-\| U_s \|_{\Ld^\infty}, \| U_s \|_{\Ld^\infty} \right]$ by Lemma \ref{lem:necessCond-growingmod}, then
\begin{align*}
    |\bcU(z) - c|^2& \le |\bcU(z) - \mathrm{Re}(c)|^2+|\mathrm{Im}(c)|^2 \le 2|\bcU(z)|^2 + 2|\mathrm{Re}(c)|^2+|\mathrm{Im}(c)|^2 \le 4\|\bcU\|_{\Ld^\infty}^2 + |\mathrm{Im}(c)|^2.
\end{align*}
This gives
\begin{align*}
\vert \text{(II)} \vert & \leq  \Vert f \Vert_Y \sup_{y \in [-1,1]}\left|\frac{U_s''(y)}{U_s'(y)} \right| \frac{ (4\| U_s \|_{\Ld^\infty}^2+\vert \mathrm{Im}(c) \vert^2)^{\alpha}}{\vert \mathrm{Im}(c) \vert ^{2\alpha}}  \\
&\leq \Vert f \Vert_Y \sup_{y \in [-1,1]}\left|\frac{U_s''(y)}{U_s'(y)} \right|  \left(1+\frac{4\| U_s \|_{\Ld^\infty}^2}{\vert \mathrm{Im}(c) \vert^2} \right)^{\alpha}.
\end{align*}
For all $\alpha \in (1/2,1)$, $c \in \lbrace \mathrm{Im}>0 \rbrace$ and for $f \in Y$, we have
\begin{align*}
\Vert \mathrm{S}_{\alpha,c} f \Vert_{Y}& \leq 2 \Vert T_{2\alpha, 2\alpha-1} f \Vert_{\infty} \\
& \le C_0  \left(\left\Vert\frac{U_s''}{U_s'} \right\Vert_{\Ld^{\infty}} \Vert f \Vert_X+ \left\Vert\frac{U_s''}{U_s'} \right\Vert_{\mathrm{W}^{1,\infty}} \Vert f \Vert_Y \left(1+\frac{4\| U_s \|_{\Ld^\infty}^2}{\vert \mathrm{Im}(c) \vert^2} \right)^{\alpha} \right)\\
& \leq C  \left( 1+ \frac{4\| U_s \|_{\Ld^\infty}^2}{\vert \mathrm{Im}(c) \vert^{2}}\right)^{\alpha} \left\Vert\frac{U_s''}{U_s'} \right\Vert_{\mathrm{W}^{1,\infty}}\Vert f \Vert_Y,
\end{align*}
for some universal constants $C_0, C>0$ independent of $\alpha$ and $c$. This concludes the proof.
\end{proof}
We can now solve \eqref{eq:IntegroDiffTosolve}.

\begin{Cor}\label{Coro-existsolODE}
Assume that the profile $U_s$ is smooth and strictly monotone. 
Let $r>0$ be fixed. Then there exists $\alpha_r \in (1/2,1)$ such that for any $c \in \lbrace \mathrm{Im}>r\rbrace \cap \lbrace \mathrm{Re} \in U_s([-1,1]) \rbrace$ and any $\alpha \in (\alpha_r, 1)$, the equation \eqref{eq:IntegroDiffTosolve} has  a unique solution $\Phi_{\alpha,c} \in Y$, given by
\begin{align}
    \forall z \in [-1,1], \qquad \Phi_{\alpha,c}(z)=D_{\alpha}^z(c)+\sum_{n \geq 1}(1-\alpha)^n\mathrm{S}_{\alpha,c}^n[D_{\alpha}(c)],
\end{align}
and where $z \mapsto D_{\alpha}^z(c)$ is given by \eqref{def:fctDalpha}.
\end{Cor}
\begin{proof}
Since $z \mapsto D_{\alpha}^z(c)$ belongs to the space $Y$, we only need to invert the operator $\mathrm{I}+(1-\alpha) \mathrm{S}_{\alpha,c}$ on $Y$. Using the bound \eqref{ineq:BoundOpT} on $\mathrm{S}_{\alpha,c}$ from Lemma \ref{lem-opnormIntegOp}, we can choose $\alpha_r$ with $\alpha>\alpha_r$ sufficiently close to $1$, ensuring that
\begin{align*}
C(1-\alpha) \left( 1+ \frac{4\| U_s\|_{\Ld^\infty}^2}{r^{2}}\right)^{\alpha} \left\Vert\frac{U_s''}{U_s'} \right\Vert_{\mathrm{W}^{1,\infty}}<1,
\end{align*}
so that the operator $\mathrm{I}+(1-\alpha) \mathrm{S}_{\alpha,c}$ is invertible on $Y$. Setting
\begin{align*}
    \Phi_{\alpha,c}(z):=\big( \mathrm{I}-(1-\alpha)\mathrm{S}_{\alpha,c} \big)^{-1}[D_{\alpha}^z(c)],
\end{align*}
we obtain a solution $z \mapsto \Phi_{\alpha,c}(z)$ of \eqref{eq:IntegroDiffTosolve} satisfying  $\Phi_{\alpha,c}(-1)=0$, for $\alpha$ close to $1$. The proof follows by Neumann inversion formula.
\end{proof}

\subsubsection*{Step 4 - Eigenvalue problem}

So far, we have only solved the differential equation \eqref{eq:IntegroDiffTosolve} for $\alpha$ close to $1$ and for any $c \in \lbrace \mathrm{Im}>0 \rbrace$. Nevertheless, there is no guarantee that the general solution $\Phi_{\alpha,c}$ satisfies the prescribed right boundary condition \eqref{eq:DispersionRelTosolve} at $z=1$.

\medskip


To establish Theorem \ref{thm:growingmodHydrostat}, the essential step is to solve the eigenvalue problem
\begin{align*}
\widetilde{D}_{\alpha}(c):=\Phi_{\alpha,c}(1)=0, 
\end{align*}
for $\alpha$ close to $1$. Any solution $c$ of this equation in the upper half-space $\lbrace \mathrm{Im}>0 \rbrace$ will yield a desired eigenvalue for the linearized hydrostatic operator \eqref{def:LinOp-Hydrostat}. Therefore, the main goal is to prove the following.
\begin{Prop}\label{Prop:existZERO}
  There exists an analytic strictly monotone profile $U_s(z)$ such that the following holds. There exists $\alpha^{\star} \in (\frac{1}{2},1)$ such that for all $\alpha \in (\alpha^{\star},1)$, there exists $c^{\star}$ with $\mathrm{Im}(c^{\star})>0$ satisfying
\begin{align}
\widetilde{D}_{\alpha}(c^{\star})=0.
\end{align}
\end{Prop}
We need two preliminary results.
\begin{Lem}\label{Prop:HolomorphDgamma}
There exists $\alpha^{\star}_1 \in (\frac{1}{2},1)$ such that if $c_1$ is a zero of $D_1^1(c)=D_{\alpha=1}^{z=1}(c)$ as in \eqref{def:fctDalpha} and $r_1:=\mathrm{Im}(c_1)/2>0$, then the function $\widetilde{D}_{\alpha}$ is analytic on $\mathrm{B}(c_1, r_1)$ for all $\alpha \in (\alpha^{\star}_1,1)$.
\end{Lem}
\begin{proof}
    By Corollary \ref{Coro-existsolODE}, we can express 
\begin{align*}
\widetilde{D}_{\alpha}(c)&=\Phi_{\alpha,c}(1)=D_{\alpha}^1(c)+\sum_{n \geq 1}(1-\alpha)^n\mathrm{S}_{\alpha,c}^n[D_{\alpha}(c)]_{\mid z=1},
\end{align*}
where $D_{\alpha}^z(c)$ is defined in \eqref{def:fctDalpha}.
Since every term in the preceding series is a holomorphic function of the variable $c$, it suffices to demonstrate that the series converges uniformly on any compact set in $\mathrm{B}(c_1, r_1)$, implying the uniform convergence of the series of functions
$$g_n(c)=(1-\alpha)^n\big[\mathrm{S}_{\alpha,c}^n[D_{\alpha}(c)]\big]_{\mid z=1}, \qquad n \geq 1.$$
First,
\begin{align*}
\vert g_n(c) \vert & \leq (1-\alpha)^n \underset{y \in [-1,1]}{\sup} \left\vert  \big[\mathrm{S}_{\alpha,c}^n[D_{\alpha}(c)]\big]_{\mid z=y} \right\vert  \\
& \leq (1-\alpha)^n \| \mathrm{S}_{\alpha,c} \|_{Y \rightarrow Y}^n \left(\underset{y \in [-1,1]}{\sup} \vert  D_{\alpha}^y(c) \vert + \underset{y \in [-1,1]}{\sup} \vert  \partial_y D_{\alpha}^y(c) \vert \right).
\end{align*}
Appealing to \eqref{ineq:BoundOpT} yields the upper bound
\begin{align*}
\| \mathrm{S}_{\alpha,c} \|_{Y \rightarrow Y}^n \leq (C_{U_s})^n \left(1+\frac{1}{\mathrm{Im}(c)} \right)^{\alpha n},
\end{align*}
for some constant $C_{U_s}>0$ depending on $\|\bcU\|_{\Ld^\infty}$, and 
\begin{align*}
\underset{y \in [-1,1]}{\sup} \vert  D_{\alpha}^y(c) \vert + \underset{y \in [-1,1]}{\sup} \vert  \partial_y D_{\alpha}^y(c) \vert \leq  \frac{3}{\vert \mathrm{Im}(c) \vert^{2\alpha}}.
\end{align*}
As $c \in \mathrm{B}(c_1, r_1)$ with $r_1=\mathrm{Im}(c_1)/2$, then $\mathrm{Im}(c) \geq r_1$, so that 
\begin{align*}
\vert g_n(c) \vert & \leq  (1-\alpha)^n (C_{U_s})^n \left(1+\frac{1}{\mathrm{Im}(c)} \right)^{\alpha n}\frac{1}{\vert \mathrm{Im}(c) \vert^{2\alpha}} \leq (1-\alpha)^n (C_{U_s})^n \left(1+\frac{1}{r_1} \right)^{\alpha n}\frac{1}{r_1^{2\alpha}}.
\end{align*}
Choosing $\alpha$ close $1^-$ so that 
\begin{align*}
(1-\alpha) C_{U_s}\left(1+\frac{1}{r_1} \right)^{\alpha}<1
\end{align*}
allows to conclude.
\end{proof}
\begin{Lem}\label{Prop:DiffDgamma} Under the assumptions of Lemma \ref{Prop:HolomorphDgamma}, there exist $\delta \in (0, \min\{1, r_1\})$, $\alpha^{\star}_2(\delta) \in (\frac{1}{2},1)$ and $M(\delta)>0$ such that 
\begin{align}
\forall c \in  \mathrm{B}(c_1, \delta), \qquad  \vert \widetilde{D}_{\alpha}(c)-D_{1}^{1}(c)  \vert \leq M(\delta) (1-\alpha ),
\end{align}
for all $\alpha \in (\alpha^{\star}_2,1)$.
\end{Lem}
\begin{proof}
    By Corollary \ref{Coro-existsolODE}, it follows that
\begin{align*}
 \widetilde{D}_{\alpha}(c)-D_{1}^{1}(c) =D_{\alpha}^{1}(c)-D_{1}^{1}(c)+\sum_{n \geq 1}(1-\alpha)^n \big[\mathrm{S}_{\alpha,c}^n[D_{\alpha}(c)] \big]_{\mid z=1},
\end{align*}
where
\begin{align*}
D_{\alpha}^{1}(c)=\int_{-1}^1\frac{\mathrm{d}r}{(U_s(r)-c)^{2\alpha}}.
\end{align*}
Arguing as in the proof of Lemma \ref{Prop:HolomorphDgamma}, we obtain
\begin{align*}
\vert \widetilde{D}_{\alpha}(c)-D_{1}^{1}(c) \vert &\leq \vert D_{\alpha}^{1}(c)-D_{1}^{1}(c) \vert +  \frac{1}{\delta^{2\alpha}}\sum_{n \geq 1}\left[(1-\alpha)C_{U_s}\left(1+\frac{1}{\delta} \right)^{\alpha} \right]^n.
\end{align*}
Being $\delta <1$ fixed, we can choose $\alpha \in (1/2, 1)$ so that $0<C_{U_s}(1-\alpha)\left(1+\frac{1}{\delta} \right)<3/4$, obtaining 
\begin{align*}
\left\vert \frac{1}{\delta^{2\alpha}}\sum_{n \geq 1}\left[C_{U_s}(1-\alpha)\left(1+\frac{1}{\delta} \right)^{\alpha} \right]^n \right\vert & \leq \frac{1}{\delta^{2\alpha}} \frac{C_{U_s}(1-\alpha)\left(1+\frac{1}{\delta} \right)}{1-C_{U_s}(1-\alpha)\left(1+\frac{1}{\delta} \right)} \\
&\leq \frac{1}{\delta^{2}}\frac{4}{3} (1-\alpha)C_{U_s}\left(1+\frac{1}{\delta} \right).
\end{align*}
Next, we apply the Mean-Value Theorem to the function $x \mapsto (U_s(r)-c)^{-2x}$ on $(\alpha, 1)$ and we obtain
\begin{align*}
\vert D_{\alpha}^{1}(c)-D_{1}^{1}(c) \vert \leq 2( 1-\alpha ) \int_{-1}^1 \underset{x \in [\alpha,1]}{\sup}\dfrac{\mathrm{d}r}{\vert U_s(r)-c \vert^{2x+1}} \leq 2( 1-\alpha ) \underset{x \in [\alpha,1]}{\sup}\dfrac{1}{\delta^{2x+1}}  \leq  \dfrac{2( 1-\alpha )}{\delta^{3}}.
\end{align*}
This concludes the proof.
\end{proof}

\begin{proof}[Proof of Proposition \ref{Prop:existZERO}]
By Proposition \ref{Prop-Appendix} in Appendix \ref{Appendix:ZeroFunction}, for $\alpha=1$ (homogeneous hydrostatic equations), there exists an analytic strictly increasing profile $U_s(z)$ and a $c_1$ with $\mathrm{Im}(c_1)>0$ such that
\begin{align}\label{eq:integHydrostaticHomog-profile}
D_{1}^{1}(c_1)=\int_{-1}^1\frac{\mathrm{d}r}{(U_s(r)-c_1)^{2}}=0.
\end{align}
We can also find an associated solution to the equation \eqref{eq:RayleighHydrostatic} and, appealing to $(2)$ in Lemma \ref{lem:necessCond-growingmod}, deduce that $\mathrm{Re}(c_1) \in U_s([-1,1])$.
Let us introduce $r_1:=\mathrm{Im}(c_1)/2$. First, observe that the function $c \mapsto D_{1}^{1}(c)$ is analytic on $\lbrace \mathrm{Im}(c)>0 \rbrace$, by standard preservation of holomorphy under the integral. 
In fact, for all $\eta>0$, $c \in \lbrace \mathrm{Im}(c) >\eta \rbrace$,
\begin{align*}
\left\vert \frac{1}{(U_s(r)-c)^{2}} \right\vert  \leq \frac{1}{|\mathrm{Im}(c)|^2} \leq \frac{1}{\eta^2} \in \Ld^1(-1,1).
\end{align*} 
It follows that the zeros of $D_1^1$ are isolated and being $c_1$ a zero, there exists $\delta \in (0, \min\{1, r_1\})$ such that
$$ \epsilon:=\underset{c \in \mathrm{B}(c_1, \delta) \setminus \lbrace c_1 \rbrace}{\inf} \vert D_{1}^{1}(c) \vert > 0.$$
Now, setting $\alpha^{\star}=\max(\alpha^{\star}_1,\alpha^{\star}_2, \alpha_{r_1})$, where $\alpha_{r_1}$ is given by Corollary \ref{Coro-existsolODE}, we can conclude by appealing to Lemma \ref{Prop:HolomorphDgamma} and \ref{Prop:DiffDgamma}. In fact, for all $\alpha>\alpha^{\star}$ sufficiently close to $1$ so that $M(\delta) (1-\alpha ) \leq \epsilon/2$, we can apply Rouché's theorem to the family of holomorphic functions $(\widetilde{D}_{\alpha})_{\alpha \in (\alpha^\star,1)}$ to deduce that $\widetilde{D}_{\alpha}$ has a zero in $\mathrm{B}(c_1, \delta)$. 
\end{proof}
Our strategy also provides some information on the localization of the zeros of the dispersion relation.
\begin{Cor}\label{Coro:Gamma0}
 Let $\alpha \in (\alpha^{\star},1)$. The quantity
\begin{align}\label{def:Gamma0}
\gamma_0:= \max \,  \lbrace \mathrm{Im}(c) \mid \widetilde{D}_{\alpha}(c) =0 \ \text{and} \  \mathrm{Im}(c) > 0  \rbrace
\end{align}
is attained and positive.
\end{Cor}
\begin{proof}
By Proposition \ref{Prop:existZERO}, the considered set is non empty. As $\widetilde{D}_{\alpha}$ is continuous on $\lbrace \mathrm{Im}(c)> 0 \rbrace$, it is enough to prove that the set of zeros  $\lbrace c \mid \widetilde{D}_{\alpha}(c) =0 \rbrace $ is bounded. 
By Lemma \ref{Lem:IntegralFPositive} from the Appendix \ref{Appendix:ZeroFunction}, there exists $\overline{R}>0$ such that
\begin{align*}
\forall R \geq \overline{R}, \ \forall c \in \lbrace \vert z \vert =R \rbrace \cap \lbrace \mathrm{Im}(c)\neq 0 \rbrace, \qquad \vert D_{1}^{1}(c) \vert >0.
\end{align*}
Therefore the set of zeros of the function $\vert D_{1}^{1} \vert$ in $\lbrace \mathrm{Im}\neq 0 \rbrace$ is bounded. It means that the hydrostatic Rayleigh equation \eqref{eq:RayleighHydrostatic} has no solution for $\vert c \vert$ large enough. From Propositions \ref{Prop:HolomorphDgamma}--\ref{Prop:DiffDgamma}, we also know that in any sufficiently small neighborhood of a zero of $\vert D_{1}^{1} \vert$, there is a unique zero of $\widetilde{D}_{\alpha}$. This proves that the set of zeros of the function $\widetilde{D}_{\alpha}$ in $\lbrace \mathrm{Im}\neq 0 \rbrace$ is bounded as well, and thus concludes the proof.
\end{proof}


\subsection{Long-wave instability for the Euler-Boussinesq system}\label{Subsec:LongwaveInstab}
Our objective here is to identify an unstable eigenvalue for the linearization of the Euler-Boussinesq system \eqref{eqintro:BE-scaled} with $\eps=1$. Building on previous ideas, we approach this problem as a perturbation in low frequency (i.e. long-wave) of the linearized hydrostatic equations introduced in Section \ref{Subsec:UnboundedSpecHydrost}.

For $M>0$, we recall the notation $$\T_{M}:= \R / (2 \pi M \Z) \simeq [0,2 \pi M).$$ 
For $\eps=1$, we consider \eqref{eqintro:BE-scaled} for $(x,z) \in \T_M \times (-1,1)$. As before, the linearization around perturbations of a shear-stratified equilibrium
gives the system
\begin{equation}
      \begin{cases}
      \partial_t \varrho +U_s(z) \partial_x \varrho -\rho_s'(z) \partial_x \varphi =0, \\
\partial_t \omega + U_s(z) \partial_x \omega- U_s''(z)\partial_x \varphi= \partial_x \varrho,
\end{cases}
\end{equation}
where the streamfunction now satisfies the standard Poisson problem
\begin{align*}
\begin{cases}
        \Delta \varphi=\omega,\\
        \varphi_{\mid z=\pm 1}=0.
    \end{cases}
\end{align*}
%
In compact form, it reads
\begin{align}\label{eq:LinBoussinesq}
\partial_t \begin{pmatrix}
\varrho\\
\omega
\end{pmatrix}
=\mathscr{B}_{\rho_s, U_s}\begin{pmatrix}
\varrho\\
\omega
\end{pmatrix},
\end{align}
where $\mathscr{B}_{\rho_s, U_s}$ is the linearized operator  around $(\rho_s, U_s)$ defined by
\begin{align}\label{def:LinOp-Boussinesq}
\mathscr{B}_{\rho_s, U_s}\begin{pmatrix}
\varrho\\
\omega
\end{pmatrix}:=\begin{pmatrix}
-U_s(z) \partial_x \varrho + \rho_s'(z)\partial_x \varphi \\[1mm]
-U_s(z) \partial_x \omega +U_s''(z)\partial_x \varphi + \partial_x \varrho
\end{pmatrix}.
\end{align}
The existence of an unstable eigenvalue for the linearized operator $\mathscr{B}_{\rho_s, U_s}$ (around a suitably chosen steady state) is then demonstrated in the limit of a large box $M\to\infty$ in the horizontal direction. We refer to this as a linear long-wave transverse instability, which is described as follows.

\begin{Thm}\label{thm:growingmodEB-largetorus}
There exists an analytic stationary profile $(\rho_s(z), U_s(z))$, with $\rho_s'(z)<0$, satisfying
\begin{align*}
0< \frac{\rho_s'(z)}{|U_s'(z)|^2} < \frac{1}{4},
\end{align*}
and such that the following holds. There exists $\overline{M}>0$ such that for any $M>\overline{M}$, the linearized operator $\mathscr{B}_{\rho_s, U_s}$ on $\T_M \times (-1,1)$ has an eigenvalue $\lambda \in \C$ with $\mathrm{Re}(\lambda)>0$. Moreover, there exists an associated eigenfunction of the form
\begin{align*}
    \varrho(x,z)=\widehat{\varrho}(z) \e^{i\frac{k}{M}x}, \qquad  \omega(x,z)=\widehat{\omega}(z) \e^{i\frac{k}{M}x}
\end{align*}
for some smooth nonzero functions $\widehat{\varrho}(z),\widehat{\omega}(z)$ and some $k \in \Z \setminus \lbrace 0 \rbrace$.
\end{Thm}


\begin{Rem}
Considering the same linearized operator $\mathscr{B}_{\rho_s, U_s}$, but now on $\R \times (-1,1)$, we can establish the existence of an unstable eigenvalue in a low-frequency regime. Due to conciseness, we opt not to delve into this direction, as the analysis from Section \ref{Section:HydroLimit} will not be pursued in this context.
\end{Rem}

Our analysis follows the perturbative approach of Section \ref{Subsec:UnboundedSpecHydrost}. We look for a solution $(\varrho, \omega)$ to \eqref{eq:LinBoussinesq} of the form
\begin{align}\label{eq:mode-low}
\varrho(t,x,z)=\widehat{\varrho}(z) \e^{i\frac{k}{M}(x-ct)}, \qquad  \varphi(t,x,z)=\widehat{\varphi}(z) \e^{i\frac{k}{M}(x-ct)}, \qquad \Delta \varphi=\omega,
\end{align}
with 
$k \in \mathbb{Z}$
($k \neq 0$), $t >0$, $c \in \C$ and $\widehat{\varrho}, \widehat{\varphi} \in \C$, with $\mathrm{Im}(c) \neq 0$. As we will see, $c=c[k]$ and $(\widehat{\varrho}, \widehat{\varphi} )=(\widehat{\varrho}_k, \widehat{\varphi}_k)$ depend on the horizontal frequency $k \in \mathbb{Z}$. Hence, if such a solution exists, then 
\begin{align}
    \label{eq:lambdak}
    \lambda=\lambda[k]=-ik c[k]
\end{align}
is an eigenvalue of the operator $\mathscr{B}_{\rho_s, U_s}$ with
\begin{align*}
\mathrm{Re} \, \lambda= k \mathrm{Im} \, c[k],
\end{align*}
and this proves Theorem \ref{thm:growingmodEB-largetorus}.

\medskip

Plugging the ansatz \eqref{eq:mode-low} into the linearized equations \eqref{eq:LinBoussinesq} and arguing as in Section \ref{Subsec:UnboundedSpecHydrost}, we obtain the modified Taylor-Goldstein equation
\begin{align}\label{eq:TGmodified}
(U_s-c)^2\left(\partial_z^2-\left(\frac{k}{M} \right)^2\right) \widehat{\varphi}-(U_s-c)U_s'' \widehat{\varphi}&=\rho_s'(z) \widehat{\varphi},
\qquad \widehat{\varphi}(\pm1)=0.
\end{align}
\begin{Rem}
Note that \eqref{eq:TGmodified} reduces to \eqref{eq:TaylorGoldstein} when $k=0$. However, unlike \eqref{eq:TaylorGoldstein}, \eqref{eq:TGmodified} is not homogeneous in $k$, suggesting that the solution $c=c[k]$ to \eqref{eq:TGmodified} may depend on $k$. Consequently, $\lambda$ in \eqref{eq:lambdak} is not necessarily unbounded for large $k$ - in contrast to the eigenvalue $\lambda$ from Theorem \ref{thm:growingmodHydrostat} - making the instability in the Euler-Boussinesq system milder, and in particular, it does not imply any ill-posedness.
\end{Rem}

Let us now consider the profile $(\rho_s, U_s)$ given by Theorem \ref{thm:growingmodHydrostat}. By Section \ref{Subsec:UnboundedSpecHydrost}, it satisfies
\begin{align*}
-\rho_s'(z)=\alpha (1-\alpha) |U_s'(z)|^2, \qquad \alpha \in (0,1), \ \alpha \neq 1/2.
\end{align*}
The new unknown $\psi(z)$ defined by
\begin{align*}
\widehat{\varphi} = (U_s-c)^{\alpha} \psi
\end{align*}
solves
\begin{align}\label{ODE:tosolve-nonhydro}
\partial_z \big((U_s-c)^{2\alpha} \partial_z \psi \big)-(U_s-c)^{2\alpha}\left(\frac{k}{M} \right)^2 \psi +(\alpha-1)(U_s-c)^{2\alpha-1}U_s'' \psi=0.
\end{align}
Arguing as in Section \ref{Subsec:UnboundedSpecHydrost}, the goal is to find a pair $(c, \psi)$ with $\mathrm{Im}(c) \neq 0$ such that 
\begin{align}\label{eq:IntegroDiffTosolveNONHYDROS}
    \psi(z)&=D_{\alpha}^z(c)+(1-\alpha)\mathrm{S}_{\alpha,c}[\psi](z)+\left(\frac{k}{M} \right)^2  \widetilde{\mathrm{S}}_{2\alpha, 2\alpha}[\psi](z), \\
\label{eq:DispersionRelTosolveNONHYDROS}
   \psi(1)&=0, 
\end{align}
where $\mathrm{S}_{\alpha,c}$ and $D_{\alpha}(c)$ are in \eqref{def:Op-Salpha}--\eqref{def:fctDalpha} and  
\begin{align*}
    \widetilde{\mathrm{S}}_{2\alpha, 2\alpha}[\psi](z):= \frac{1}{(U_s(y)-c)^{2\alpha}}  \int_{-1}^{y} f(z) (U_s(z)-c)^{2 \alpha}  \, \mathrm{d}z.
\end{align*}
To construct a solution, we formally write
$$\psi(z)=\big( \mathrm{I}-\mathbf{S}_{\alpha,c,M}\big)^{-1}[D_{\alpha}^z(c)],$$
where
\begin{align}
    \mathbf{S}_{\alpha,c,M}:= (1-\alpha)\mathrm{S}_{\alpha,c}+\left(\frac{k}{M} \right)^2  \widetilde{\mathrm{S}}_{2\alpha, 2\alpha}.
\end{align}
Our approach essentially mirrors the proof of Corollary \ref{Coro-existsolODE}, but now in the regime where $\alpha \rightarrow 1$ and $M \rightarrow + \infty$, with the discrete frequency $k$ fixed. To handle the new component involving the operator $\widetilde{\mathrm{S}}_{2\alpha, 2\alpha}$, we employ an approach similar to what was used for the terms $\mathrm{(II)}$ and $\mathrm{(III)}$ in the proof of Proposition \ref{lem-opnormIntegOp}. Ensuring that the solution $\psi$ satisfies the dispersion relation $\psi(1)=0$ is achieved by proceeding akin to the proof of Proposition \ref{Prop:existZERO}, relying on Rouché's theorem when $\alpha \rightarrow 1$ and $M \rightarrow + \infty$.

In summary, employing the perturbative approach outlined in Section \ref{Subsec:UnboundedSpecHydrost}, we arrive at the following result.

\begin{Prop}\label{Prop:existLongwave}
There exists an analytic strictly monotone profile $U_s(z)$ such that the following holds. There exists $\alpha^{\star} \in (\frac{1}{2},1)$ and $\overline{M}>0$ such that for all $\alpha \in (\alpha^{\star},1)$ and $M> \overline{M}$, there exists a solution  $(c^{\star}, \psi^\star)$ with $\mathrm{Im}(c^{\star}) \neq 0$ satisfying \eqref{eq:IntegroDiffTosolveNONHYDROS}--\eqref{eq:DispersionRelTosolveNONHYDROS}.
\end{Prop}
This allows to conclude the proof of Theorem \ref{thm:growingmodEB-largetorus}.

\section{Failure of the hydrostatic limit}\label{Section:HydroLimit}

In this section, we present the proof of Theorem \ref{thm-invalidHydrostatlim}. As outlined in the introduction, we aim to demonstrate the failure of the formal limit from \eqref{eqintro:BE-scaled} to \eqref{eqintro:BHyd-E} around certain unstable stationary profiles in the sense of Miles-Howard. Through a suitable rescaling, we establish a connection between the hydrostatic limit and the large-time/large-box limit of the Euler-Boussinesq system. This link allows us to leverage the insights gained from the analysis of linear long-wave instabilities in Section \ref{Subsec:LongwaveInstab}.

\medskip

Let us consider $M>0$ large enough so that Theorem \ref{thm:growingmodEB-largetorus} holds, with the fixed profile $(\rho_s(z), U_s(z))$ given by this theorem. In what follows, and apart from \textit{Step 4}, all the functional spaces are implicitly set on $\T_M \times (-1,1)$.

\subsubsection*{Step 1: Oscillatory problems}
First, we look at the Euler-Boussinesq system of unknowns $(\varrho_\eps', u_\eps',v_\eps',p_\eps')$
\begin{equation}
      \begin{cases}
      \partial_t \varrho_\eps' +u_\eps' \partial_x \varrho_\eps'+ v_\eps' \partial_z \varrho_\eps'=0, \\[1mm] 
\partial_t u_\eps' +u_\eps' \partial_x u_\eps'+ v_\eps' \partial_z u_\eps' =-\partial_x p_\eps',  \\[1mm]
\eps^2 \big( \partial_t v_\eps'+ u_\eps' \partial_x v_\eps'+v_\eps' \partial_z v_\eps' \big)=-\partial_z p_\eps'-\varrho_\eps' , \\[1mm]
\partial_x u_\eps' + \partial_z v_\eps'=0,
\end{cases}
\end{equation} 
with ${v'_{\eps \, \mid z=\pm1}}=0$, and set on $\T_{\eps M}\times (-1,1)$. We actually consider a sequence $\eps_\ell=\frac{1}{\ell M}$ of solutions, with $\ell \in \N \setminus \lbrace 0 \rbrace$ (keeping the subscript $\eps$ for the sake of readability). 
By gluing $(\eps M)^{-1}$ copies of $Y'_{\eps}=(\varrho_\eps', u_\eps',v_\eps')$, we can obtain a solution $Y_{\mathrm{true},\eps}=(\varrho_{\mathrm{true},\eps}, u_{\mathrm{true},\eps},v_{\mathrm{true},\eps})$ to the original system \eqref{eqintro:BE-scaled} on $\T \times (-1,1)$ by writing
\begin{align}\label{eq:gluing}
Y_{\mathrm{true},\eps}(t,x,z)=Y'_{\eps}(t,x-j\eps M,z),
\end{align}
 with
 $$t>0, \qquad  x \in [j\eps M, (j+1)\eps M ) \ \text{for}\  j=0, \cdots, \ell-1, \qquad   z \in (-1,1).$$

\subsubsection*{Step 2: Hyperbolic rescaling and sharp semigroup estimates}
As argued by Grenier in \cite{Grenier-hydroderiv} and in the spirit of \cite{HKH} for the quasineutral limit, we perform a hyperbolic rescaling and look for the equation satisfied by the function $(\varrho, u,v,p)$ defined on $\R^+ \times \T_M \times (-1,1)$ by 
\begin{align}\label{eq:rescaling}
    (\varrho_\eps', u_\eps',v_\eps', p_\eps')(t',x',z)= (\varrho, u, \eps^{-1}v, p)\left(\frac{t'}{\eps}, \frac{x'}{\eps},z \right), \qquad (t',x',z) \in \R^+ \times  \T_{\eps M}  \times (-1,1).
\end{align}
It is a solution of the following system, \textit{independent of $\eps$},
\begin{equation}\label{eq:EB-scaled}
      \begin{cases}
      \partial_t \varrho +u \partial_x \varrho+ v \partial_z \varrho=0, \\ 
\partial_t u +u \partial_x u + v \partial_z u  =-\partial_x p,  \\ 
 \partial_t v+ u \partial_x v+v\partial_z v =-\partial_z p-\varrho , \\ 
\partial_x u + \partial_z v=0,
\end{cases} 
\end{equation} 
set on $\R^+ \times \T_M \times (-1,1)$.
This is precisely the Euler-Boussinesq system, and its linearization around any stratified profile 
\begin{align*}
(\rho_s(z), U_s(z), 0, p_s(z)), \qquad p_s'(z):=-\rho_s(z)
\end{align*}
has been studied in Section \ref{Subsec:LongwaveInstab} and reads in vorticity form
\begin{align}
\partial_t \begin{pmatrix}
\varrho\\
\omega
\end{pmatrix}
=\mathscr{B}_{\rho_s, U_s}\begin{pmatrix}
\varrho\\
\omega
\end{pmatrix},
\end{align}
where $\mathscr{B}_{\rho_s, U_s}$ is defined in \eqref{def:LinOp-Boussinesq}. Note that the full nonlinear system solved by the perturbation can be written as
\begin{align*}
(\partial_t-\mathscr{B}_{\rho_s, U_s})\begin{pmatrix}
\varrho\\
\omega
\end{pmatrix}=- \begin{pmatrix}
\partial_z \Delta^{-1} \omega \partial_x \varrho - \partial_x \Delta^{-1} \omega \partial_z \varrho \\[1mm]
\partial_z \Delta^{-1} \omega \partial_x \omega - \partial_x \Delta^{-1} \omega \partial_z \omega
\end{pmatrix}.
\end{align*}
Before going further, we need to study the semigroup generated by $\mathscr{B}_{\rho_s, U_s}$. First, we decompose the previous operator as
\begin{align*}
\mathscr{B}_{\rho_s, U_s}=B_0+ K
\end{align*}
where 
\begin{align*}
B_0\begin{pmatrix}
\varrho\\
\omega
\end{pmatrix}:=\begin{pmatrix}
-U_s \partial_x \varrho\\[1mm]
-U_s \partial_x \omega + \partial_x \varrho
\end{pmatrix},\qquad 
K\begin{pmatrix}
\varrho\\
\omega
\end{pmatrix}:=\begin{pmatrix}
\rho_s' \partial_x \Delta^{-1} \omega \\[1mm]
 U_s'' \partial_x \Delta^{-1} \omega
\end{pmatrix}.
\end{align*}
Since the profile $(\rho_s,U_s)$ is smooth on $(-1,1)$, and using Rellich's theorem and elliptic regularity, we easily deduce that the operator $K$ is compact on $\mathrm{W}^{k+1,p}(\T_M \times (-1,1)) \times \mathrm{W}^{k,p}(\T_M \times (-1,1))$ for any $k>0$. Furthermore, we have the following property.
\begin{Lem}\label{Lem:PropSemigroupB0}
For all $p \in [1,\infty]$ and $k\geq 0$, the operator $B_0$ generates a strongly continuous semigroup $\e^{t B_0}$ on $\mathrm{W}^{k+1,p} \times \mathrm{W}^{k,p}$ with
for all $(\varrho,\omega) \in \mathrm{W}^{k+1,p} \times \mathrm{W}^{k,p}$, acting as
\begin{align*}
\e^{t B_0}\begin{pmatrix}
\varrho\\
\omega
\end{pmatrix}(x,z)= \begin{pmatrix}
\varrho(x-tU_s(z),z)\\[2mm]
\omega(x-tU_s(z),z)+t \partial_x \varrho(x-tU_s(z),z)
\end{pmatrix}.
\end{align*}
Furthermore, for all $\beta >0$, $p \in [1,\infty]$ and $k\geq 0$, there exists $C_{k,\beta}>0$ such that 
\begin{align}
\Vert \e^{t B_0}(\varrho,\omega)\Vert_{\mathrm{W}^{k+1,p} \times \mathrm{W}^{k,p}} \leq C_{k,\beta}\, \e^{\beta t} \Vert  (\varrho,\omega)\Vert_{\mathrm{W}^{k+1,p} \times \mathrm{W}^{k,p}}.
\end{align}
\end{Lem}
\begin{proof}
The representation formula is a direct consequence of the method of characteristics. For the estimate, we can use the invariance by translation on the torus, the smoothness of $U_s(z)$ and an induction argument to show that for any $k \geq 0$, there exists a real polynomial $P_k^{U_s}$ (in time) of degree $k$ such that for all $t \geq 0$
\begin{align*}
    \Vert \e^{t B_0}(\varrho,\omega)\Vert_{\mathrm{W}^{k+1,p} \times \mathrm{W}^{k,p}} \lesssim P_k^{U_s}(t) \Vert  (\varrho,\omega)\Vert_{\mathrm{W}^{k+1,p} \times \mathrm{W}^{k,p}},
\end{align*}
hence
\begin{align*}
    \Vert \e^{t B_0}(\varrho,\omega)\Vert_{\mathrm{W}^{k+1,p} \times \mathrm{W}^{k,p}} &\leq C_{k,U_s} \left( \sum_{i=0}^k t^i \right) \Vert  (\varrho,\omega)\Vert_{\mathrm{W}^{k+1,p} \times \mathrm{W}^{k,p}},
\end{align*}
for some constant $C_{k,U_s}>0$. Let $\beta>0$ be fixed. Using the fact that
\begin{align*}
    k! \sum_{i=0}^{+\infty} \frac{(\beta t)^i}{i!} \geq \underset{i=0, \cdots,k}{\min}(\beta^i)\sum_{i=0}^k t^i, \qquad  t >0,
\end{align*}
we deduce that 
\begin{align*}
    \Vert \e^{t B_0}(\varrho,\omega)\Vert_{\mathrm{W}^{k+1,p} \times \mathrm{W}^{k,p}} &\leq \frac{k!}{\underset{i=0, \cdots, k}{\min}(\beta^i)} C_{k,U_s} \e^{\beta t}\Vert  (\varrho,\omega)\Vert_{\mathrm{W}^{k+1,p} \times \mathrm{W}^{k,p}}.
\end{align*}
Therefore for all $\beta >0 $, we have
\begin{align*}
\Vert \e^{t B_0}(\varrho,\omega)\Vert_{\mathrm{W}^{k+1,p} \times \mathrm{W}^{k,p}} \leq C_{k, \beta,U_s} \e^{\beta t} \Vert  (\varrho,\omega)\Vert_{\mathrm{W}^{k+1,p} \times \mathrm{W}^{k,p}},
\end{align*}
for some constant $C_{k,\beta,U_s}>0$, which concludes the proof.
\end{proof}
As the set of unstable eigenvalue for $\mathscr{B}_{\rho_s, U_s}$ is not empty, we can deduce the following.
\begin{Prop}\label{prop:semigroupB}
For all $p \in [1,\infty]$ and $k \in \N$, the operator $\mathscr{B}_{\rho_s, {U_s}}=B_0+ K$ generates a strongly continuous semigroup  $\e^{t \mathscr{B}_{\rho_s, U_s}}$ on $\mathrm{W}^{k+1,p} \times \mathrm{W}^{k,p}$ that satisfies the following properties:
\begin{enumerate} [label=(P\arabic*), ref=(P\arabic*)]
    \item\label{item:P1}  For all $\beta>0$, the set $\mathrm{Sp}(\mathscr{B}_{\rho_s, U_s}) \cap \lbrace \mathrm{Re} > \beta \rbrace$ is finite and any point in $\mathrm{Sp}(\mathscr{B}_{\rho_s, U_s}) \cap \lbrace \mathrm{Re} > \beta \rbrace$ is an isolated eigenvalue with finite algebraic multiplicity. 
    \item\label{item:P2} There exists an eigenvalue $\Lambda$ in $\lbrace \mathrm{Re} >0 \rbrace$ of $\mathscr{B}_{\rho_s, U_s}$ with positive maximal real part.
    \item\label{item:P3} For all $\beta > \mathrm{Re} (\Lambda) $, there exists $M_{k, \beta}>0$ such that for all $(\varrho,\omega) \in \mathrm{W}^{k+1,p} \times \mathrm{W}^{k,p}$, we have
\begin{align*}
\forall t \geq 0, \qquad \Vert \e^{t \mathscr{B}_{\rho_s, U_s}}(\varrho, \omega) \Vert_{\mathrm{W}^{k+1,p} \times \mathrm{W}^{k,p}} \leq M_{k, \beta} \e^{\beta t} \Vert (\varrho, \omega) \Vert_{\mathrm{W}^{k+1,p} \times \mathrm{W}^{k,p}}.
\end{align*}
\end{enumerate}

\end{Prop}
\begin{proof}
We rely on Shizuta-Vidav's theorem from spectral theory - see \cites{Vidav, Shizuta}. Let us stress that this technique is powerful in the case of a coupling between transport and elliptic equations (see e.g. \cites{CGG, lemou2020nonlinear}). Relying on the compactness of the bounded operator, it follows that  $\e^{t B_0}  K$ is compact for all $t \in \R^+$ and that the map $t \mapsto \e^{t B_0}  K \in  \mathrm{W}^{k+1,p} \times \mathrm{W}^{k,p}$ is continuous. Thanks to Lemma \ref{Lem:PropSemigroupB0}, we get the claim \ref{item:P1} by means of \cite{Shizuta}*{Theorem 1.2}. Thanks to Theorem \ref{thm:growingmodEB-largetorus}, we also know that the operator $\mathscr{B}_{\rho_s, U_s}$
admits one eigenvalue $\lambda \in \lbrace \mathrm{Re}>0 \rbrace$, associated to a smooth eigenfunction. By \ref{item:P1}, it means that there exists an eigenvalue $\Lambda$ with positive maximal real part, that is \ref{item:P2}. We get \ref{item:P3} by using \cite{Shizuta}*{Theorem 1.3}. The proof is then complete.

\end{proof}

\subsubsection*{Step 3: Nonlinear instability -  high-order scheme à la Grenier}
In the spirit of \cite{HKH}, we now employ Grenier's method from \cite{Gre-instabScheme} to deduce nonlinear instability from linear instability.
Recall that $M>0$ has been taken large enough so that Theorem \ref{thm:growingmodEB-largetorus} applies. We aim at proving the following.
\begin{Thm}\label{thm:nonlin-instabilityEB}
There exists $m>0$ such that for any $\delta>0$ and any $p \in \N$ , there exists a solution $(\varrho, u,v)$ to \eqref{eq:EB-scaled} with
\begin{align*}
\Vert (\varrho, u, v)_{\mid t=0}-(\rho_s, U_s,0) \Vert_{\mathrm{H}^p(\T_M \times (-1,1))} \leq \delta,
\end{align*}
while
\begin{align*}
  \underset{t \in [0,\tau_\delta]}{\sup} \Vert  (\varrho, u, v)_{\mid t}-(\rho_s, U_s,0)  \Vert_{\Ld^2 (\T_M \times (-1,1))} \geq m,
\end{align*}
where $\tau_\delta=\mathcal{O}(\vert \log \delta \vert)$ as $\delta \rightarrow 0$.
\end{Thm}
We show in Step 4 how to deduce the Theorem \ref{thm-invalidHydrostatlim} from Theorem \ref{thm:nonlin-instabilityEB}. The proof will be divided in several  steps, but we first explain the main ideas and set notations.

Let $N \geq 1$ be an integer to be fixed later on. Thanks to Proposition \ref{prop:semigroupB}, we can consider an unstable eigenvalue $\Lambda$ of the operator $\mathscr{B}_{\rho_s, {U_s}}$ defined in \eqref{def:LinOp-Boussinesq} with maximal real part $\mathrm{Re}(\Lambda)>0$, and an associated infinitely smooth eigenfunction $(\varrho_{\mathrm{eig}}, \omega_{\mathrm{eig}})$. Without loss of generality, we can assume that $\Vert (\varrho_{\mathrm{eig}}, \omega_{\mathrm{eig}}) \Vert_{\Ld^2 \times \H^{-1}}=1$. Let us also assume for now that $\Lambda$ is real and that $(\varrho_{\mathrm{eig}}, \omega_{\mathrm{eig}})$ is real valued.

Roughly speaking, our goal is to show that there exists an $m>0$ independent of $\delta$ such that the solution $(\varrho, \omega)$ of \eqref{eq:EB-scaled} with initial data $(\rho_s,U_s')+\delta (\varrho_{\mathrm{eig}}, \omega_{\mathrm{eig}})$ (with $\delta>0$ small enough) satisfies the following:
\begin{itemize}
    \item  it exists on the interval of time $[0,T_m]$ with
    \begin{align*}
        T_m=-\dfrac{\log \delta}{\Lambda}+\dfrac{\log  2m}{\Lambda},
    \end{align*}
    so that we have
    \begin{align*}
        \delta \e^{T_m \Lambda} =2m.
    \end{align*}
    \item on $[0,T_m]$, it holds
    \begin{align*}
        \Vert (\varrho, \omega)(t)-(\rho_s,U_s')-\delta \e^{\Lambda t}(\varrho_{\mathrm{eig}}, \omega_{\mathrm{eig}})\Vert_{\Ld^2 \times \H^{-1}} \leq \frac{\delta}{2}\e^{\Lambda t}. 
    \end{align*}
\end{itemize}
This implies, by triangle inequality, that
\begin{align*}
    \Vert (\varrho, \omega)_{\mid t=T_m}-(\rho_s,U_s') \Vert_{\Ld^2 \times \H^{-1}} \geq m,
\end{align*}
while for any $k \geq 0$
\begin{align*}
    \Vert (\varrho, \omega)_{\mid t=0}-(\rho_s,U_s') \Vert_{\H^{k+1} \times \H^k} = \delta \Vert (\varrho_{\mathrm{eig}}, \omega_{\mathrm{eig}}) \Vert_{\H^{k+1} \times \H^k},
\end{align*}
which can be made arbitrarily small.

A naive idea to find some instability according to the previous strategy is to consider
\begin{align*}
\begin{pmatrix}
\varrho_{\mathrm{app}}\\
\omega_{\mathrm{app}}
\end{pmatrix}
= 
\begin{pmatrix}
\rho_s\\
U_s'
\end{pmatrix}
+ \delta \begin{pmatrix}
\varrho_{\mathrm{eig}}\\
\omega_{\mathrm{eig}}
\end{pmatrix} \e^{\Lambda t}, \qquad \delta>0
\end{align*}
as an approximate solution to \eqref{eq:EB-scaled}. In fact, by definition of the linearized operator $\mathscr{B}_{\rho_s, {U_s}}$, the ansatz above satisfies the Euler-Boussinesq equation, up to a remainder of the order $\delta^2 \e^{2 \Lambda t }$. Hence, it is natural to compare the approximate solution with the true solution $(\varrho, \omega)$ of \eqref{eq:EB-scaled} with initial data
\begin{align*}
\begin{pmatrix}
\varrho\\
\omega
\end{pmatrix}_{\mid t=0}
= 
\begin{pmatrix}
\rho_s\\
U_s'
\end{pmatrix}
+ \delta \begin{pmatrix}
\varrho^{\mathrm{eig}}\\
\omega^{\mathrm{eig}}
\end{pmatrix}.
\end{align*}
However, if $\Lambda$ is not big enough, it is not clear that $(\varrho, \omega)-(\rho_s, U_s')$ and $\delta (\varrho_{\mathrm{eig}}, \omega_{\mathrm{eig}})e^{\Lambda t}$ can be compared on the same interval of time.
Grenier's iteration method is precisely a way of considering high-order approximation of the solution in order to avoid such issue. For $N \geq 1$ a given integer, recalling that $(\varrho_{\mathrm{eig}}, \omega_{\mathrm{eig}})$ is the unstable eigenmode, we consider an approximate solution
\begin{align}
\begin{pmatrix}
\varrho_{\mathrm{app}}^{(N)}\\
\omega_{\mathrm{app}}^{(N)}
\end{pmatrix}
&:= 
\begin{pmatrix}
\rho_s\\
U_s'
\end{pmatrix}
+ \sum_{j=1}^N \delta^j\begin{pmatrix}
\varrho_j\\
\omega_j
\end{pmatrix}, \label{def-SolAppGrenier}\\
\begin{pmatrix}
\varrho_1\\
\omega_1
\end{pmatrix}
&:=\begin{pmatrix}
\varrho_{\mathrm{eig}}\\
\omega_{\mathrm{eig}}
\end{pmatrix}  \e^{\Lambda t}, \label{def-SolApp1Grenier}
\end{align}
where the $(\varrho_j,
\omega_j)$ (for $j \geq 2$) are defined inductively, satisfy $\Vert (\varrho_j,
\omega_j) \Vert \lesssim \e^{\Lambda j t}$, and such that the higher order approximation $(\varrho_{\mathrm{app}}^{(N)},
\omega_{\mathrm{app}}^{(N)})$ solves the Euler-Boussinesq \eqref{eq:EB-scaled} equations with a remainder 
\begin{align}\label{est:rem-growth}
R_{\mathrm{app}}= \mathcal{O}\left(\delta^{N+1} \e^{\Lambda(N+1)t} \right),
\end{align}
for some suitable (Sobolev) norms. This procedure will allow to estimate the difference $(\varrho-\varrho_{\mathrm{app}}^{(N)}, \omega- \omega_{\mathrm{app}}^{(N)})$ provided that we choose $N$ large enough so that
\begin{align*}
(N+1)\Lambda \gg 1.
\end{align*}
In fact, if the remainder satisfies \eqref{est:rem-growth}, applying Grönwall inequality yields 
\begin{align*}
  \Vert (\varrho, \omega)(t) -(\varrho_{\mathrm{app}}^{(N)}, \omega_{\mathrm{app}}^{(N)}) (t) \Vert \lesssim \delta^{N+1}  \e^{\max\left( (N+1)\Lambda, 1+C_{\rho_s,U_s} \right)t},
\end{align*} 
for some constant $C_{\rho_s,U_s}>0$ depending on the profile $(\rho_s, U_s)$. If $N$ is large enough so that $(N+1)\Lambda> 1+C_{\rho_s,U_s}$, then we can estimate
\begin{align*}
    \Vert (\varrho, \omega)(t) -(\varrho_{\mathrm{app}}^{(N)}, \omega_{\mathrm{app}}^{(N)}) (t) \Vert \leq \left( \delta \e^{\Lambda t} \right)^{N+1},
\end{align*}
which allows to compare the difference $(\varrho, \omega) -(\varrho_{\mathrm{app}}^{(N)}, \omega_{\mathrm{app}}^{(N)})$ with $(\varrho_{\mathrm{app}}^{(N)}, \omega_{\mathrm{app}}^{(N)})-(\rho_s,U_s')= \mathcal{O}(\delta \e^{\Lambda t})$ on the same time interval.
\medskip

\textbf{Definition by induction}: For $2 \leq j \leq N$, we define inductively $(\varrho_j,
\omega_j)$ as the solution of the linear problem
\begin{align*}
(\partial_t -\mathscr{B}_{\rho_s, U_s})\begin{pmatrix}
\varrho_j\\
\omega_j
\end{pmatrix}
&=\sum_{k=1}^{j-1}
\begin{pmatrix}
 u[\omega_k]\partial_x\varrho_{j-k}+v[\omega_k]\partial_z\varrho_{j-k} \\[2mm]
 u[\omega_k]\partial_x\omega_{j-k}+v[\omega_k]\partial_z\omega_{j-k}
\end{pmatrix}=\sum_{k=1}^{j-1}
\begin{pmatrix}
 W[\omega_k] \cdot \nabla \varrho_{j-k} \\[2mm]
  W[\omega_k] \cdot \nabla\omega_{j-k}
\end{pmatrix}, \\
(\varrho_j,
\omega_j)_{\mid t=0}&=(0,0),
\end{align*}
where $$W[\omega]=(u[\omega],v[\omega])=(\partial_z \Delta^{-1} \omega,-\partial_x \Delta^{-1} \omega).$$ 
The approximate solution $(\varrho_{\mathrm{app}}^{(N)},
\omega_{\mathrm{app}}^{(N)})$ satisfies
\begin{align*}
      \partial_t \varrho_{\mathrm{app}}^{(N)} +W_{\mathrm{app}}^{(N)} \cdot \nabla \varrho_{\mathrm{app}}^{(N)}&=\sum_{\substack{1 \leq k, k' \leq N \\  N+1 \leq k+k' \leq 2N }} \delta^{k+k'}W_{k} \cdot \nabla \varrho_{k'}:= R^{(N)}_{1,\mathrm{app}},  \\[1mm]
\partial_t \omega_{\mathrm{app}}^{(N)} +W_{\mathrm{app}}^{(N)} \cdot \nabla \omega_{\mathrm{app}}^{(N)} +\partial_x \varrho_{\mathrm{app}}^{(N)}&=\sum_{\substack{1 \leq k, k' \leq N \\ N+1 \leq k+k' \leq 2N }}\delta^{k+k'}W_{k} \cdot \nabla \omega_{k'}:=R^{(N)}_{2,\mathrm{app}},
\end{align*} 
where $W_k=W[\omega_k]$.

\subsubsection*{Step 3a - Estimate of $(\varrho_j,
\omega_j)$ } Let $\ell \geq 3$ be fixed. By induction and by using the sharp growth estimate of the semigroup in Proposition \ref{prop:semigroupB}, we want to prove:
\begin{align}\label{estimate-iteratesemigroupGrenier}
\forall 1 \leq j \leq N, \qquad \Vert (\varrho_j(t),
\omega_j(t)) \Vert_{\H^{N-j+\ell+1} \times \H^{N-j+\ell}} \leq  C_{N,j} \e^{\Lambda j t },
\end{align}
for some constant $C_{N,j} \geq 0$.
\begin{proof}  We proceed by induction on $1 \leq j \leq N$. The estimate is trivial for $j=1$. We now assume it is true until $j$. Let $\gamma \in [\Lambda, 2\Lambda)$. Since $j+1 \geq 2$, we can use Duhamel’s integral formula to write
    \begin{align}
    \begin{pmatrix}
\varrho_{j+1}(t)\\
\omega_{j+1}(t)
\end{pmatrix}=-\sum_{k=1}^j \int_0^t \e^{\mathscr{B}_{\rho_s, U_s}(t-\tau)}
\begin{pmatrix}
 W[\omega_k] \cdot \nabla \varrho_{j+1-k}(\tau) \\[2mm]
  W[\omega_k] \cdot \nabla\omega_{j+1-k}(\tau) 
\end{pmatrix}\, \mathrm{d}\tau.
    \end{align}
    Relying on the sharp semigroup estimate from Proposition \ref{prop:semigroupB}, we have
    \begin{align*}
        \left\Vert \begin{pmatrix}
\varrho_{j+1}(t)\\
\omega_{j+1}(t)
\end{pmatrix} \right\Vert_{\H^{N-j+\ell} \times \H^{N-j+\ell-1}}
&\leq \sum_{k=1}^j \int_0^t \left\Vert  \e^{\mathscr{B}_{\rho_s, U_s}(t-\tau)}\begin{pmatrix}
 W[\omega_k] \cdot \nabla \varrho_{j+1-k}(\tau) \\[2mm]
  W[\omega_k] \cdot \nabla\omega_{j+1-k}(\tau) 
\end{pmatrix} \right\Vert_{\H^{N-j+\ell} \times \H^{N-j+{\ell-1}}} \, \mathrm{d}\tau \\
&\leq M_{N-j}\sum_{k=1}^j \int_0^t   \e^{\gamma (t-\tau)}\left\Vert\begin{pmatrix}
 W[\omega_k] \cdot \nabla \varrho_{j+1-k}(\tau) \\[2mm]
  W[\omega_k] \cdot \nabla\omega_{j+1-k}(\tau) 
\end{pmatrix} \right\Vert_{\H^{N-j+\ell} \times \H^{N-j+\ell-1}} \, \mathrm{d}\tau.
    \end{align*}
    We will estimate the last norm using the bound
    \begin{align}
        \forall r \geq 0, \eps>0, \qquad \Vert W[\omega]  \Vert_{\mathrm{W}^{r, \infty}} \lesssim \Vert \omega \Vert_{\mathrm{H}^{r+\eps}},
    \end{align}
    which is obtained by Sobolev embedding and elliptic regularity.

\medskip 

$\bullet$ \textbf{Estimate for} $\Vert W[\omega_k] \cdot \nabla\omega_{j+1-k} \Vert_{\H^{N-j+\ell-1}}$: by algebra property, we have
\begin{align*}
    \Vert W[\omega_k] \cdot \nabla\omega_{j+1-k} \Vert_{\H^{N-j+\ell-1}} 
    &\lesssim \Vert W[\omega_k]  \Vert_{\H^{N-j+\ell-1}} \Vert \omega_{j+1-k} \Vert_{\H^{N-j+\ell}}    \\
    & \lesssim \Vert \omega_k  \Vert_{\mathrm{H}^{N-j+\ell-2}} \Vert \omega_{j+1-k} \Vert_{\H^{N-(j+1-k)+\ell-k+1}} \\
    & \lesssim \Vert \omega_k  \Vert_{\mathrm{H}^{N-k+\ell}} \Vert \omega_{j+1-k} \Vert_{\H^{N-(j+1-k)+\ell}},
\end{align*}
        where $1 \leq k \leq j$. Using the induction hypothesis, we get
        \begin{align*}
         \Vert W[\omega_k] \cdot \nabla\omega_{j+1-k}(\tau) \Vert_{\H^{N-j+2}} \lesssim   C_k C_{j+1-k} \e^{\Lambda(j+1)\tau }.
        \end{align*}
        
\medskip 

$\bullet$  \textbf{Estimate for} $\Vert W[\omega_k] \cdot \nabla\varrho_{j+1-k} \Vert_{\H^{N-j+\ell}} $: we first use the product estimate
\begin{align*}
    \Vert W[\omega_k] \cdot \nabla\varrho_{j+1-k} \Vert_{\H^{N-j+\ell}}   \lesssim \Vert W[\omega_k] \Vert_{\Ld^{\infty}}\Vert \varrho_{j+1-k} \Vert_{\H^{N-j+\ell+1}} + \Vert W[\omega_k] \Vert_{\H^{N-j+\ell}}\Vert \nabla\varrho_{j+1-k} \Vert_{\Ld^{\infty}}.
\end{align*}
For the first term, it holds
\begin{align*}
    \Vert W[\omega_k] \Vert_{\Ld^{\infty}}\Vert \varrho_{j+1-k}\Vert_{\H^{N-j+\ell+1}}\lesssim \Vert \omega_k\Vert_{\H^{1}} \Vert \varrho_{j+1-k} \Vert_{\H^{N-(j+1-k)+\ell+1}} \leq C_k C_{j+1-k} \e^{\Lambda(j+1)s }.
\end{align*}
For the second term, we have for all $\eps>0$
\begin{align*}
    \Vert W[\omega_k] \Vert_{\H^{N-j+\ell}}\Vert \nabla\varrho_{j+1-k} \Vert_{\Ld^{\infty}} &\lesssim \Vert \omega_k\Vert_{\H^{N-j+\ell-1}} \Vert \varrho_{j+1-k} \Vert_{\H^{2+\eps}} \\
    & \lesssim \Vert \omega_k\Vert_{\H^{N-k+\ell}} \Vert \varrho_{j+1-k} \Vert_{\H^{N-(j+1-k)+\ell+1 +(j+1-k-N+\eps+2-\ell-1)}},
\end{align*}
where we have used elliptic regularity for the first factor and Sobolev embedding for the second one. Now observe that for $\eps>0$ small enough
\begin{align*}
    j+1-k-N+\eps+2-\ell-1 \leq 2-k-\ell+\eps \leq 0.
\end{align*}
As before, we get
\begin{align*}
    \Vert W[\omega_k] \Vert_{\H^{N-j+\ell}}\Vert \nabla\varrho_{j+1-k}(\tau) \Vert_{\Ld^{\infty}} &\lesssim C_k C_{j+1-k} \e^{\Lambda(j+1)\tau }.
\end{align*}
We thus obtain
\begin{align*}
            \Vert W[\omega_k] \cdot \nabla\varrho_{j+1-k} (\tau)\Vert_{\H^{N-j+\ell}} 
            & \lesssim C_k C_{j+1-k} \e^{\Lambda(j+1)\tau }.
\end{align*} 
We now come back to the estimate of $\left\Vert (
\varrho_{j+1}(t),
\omega_{j+1}(t)) \right\Vert_{\H^{N-j+\ell} \times \H^{N-j+\ell-1}}$. We have 
\begin{align*}
        \left\Vert \begin{pmatrix}
\varrho_{j+1}(t)\\
\omega_{j+1}(t)
\end{pmatrix} \right\Vert_{\H^{N-j+\ell} \times \H^{N-j+\ell-1}} 
&\lesssim M_{N-j}\left(\sum_{k=1}^j C_k C_{j+1-k} \right) \int_0^t   \e^{\gamma (t-\tau)} \e^{\Lambda(j+1)\tau } \, \mathrm{d}\tau.
    \end{align*}
    Since $j\geq 1$ and $\gamma < 2 \Lambda$, we have
    $-\gamma+(j+1) \Lambda>0$ and then
    \begin{align*}
        \left\Vert \begin{pmatrix}
\varrho_{j+1}(t)\\
\omega_{j+1}(t)
\end{pmatrix} \right\Vert_{\H^{N-j+\ell} \times \H^{N-j+\ell-1}} 
\lesssim \frac{M_{N-j}}{(j+1) \Lambda -\gamma}\left(\sum_{k=1}^j C_k C_{j+1-k} \right)  \e^{\Lambda(j+1)t } = C_{N,j+1 } \e^{\Lambda(j+1)t  }.
    \end{align*}
This concludes our induction argument and hence the proof.

\end{proof}

Let us now introduce the following quantities. We define $\delta>0$ and $T_N>0$ such that for all $t \in [0,T_N]$
\begin{align*}
     \delta \e^{\Lambda t} \leq \min \left( \frac{1}{2}, \frac{1}{2K_N}\right), \qquad \text{where} \qquad K_N:= \underset{1 \leq j \leq N}{\max}C_{N,j}.
\end{align*}
\subsubsection*{Step 3b - Estimate of $(\varrho_{\mathrm{app}}^{(N)}-\rho_s, \omega_{\mathrm{app}}^{(N)})-U_s')$}

We have the estimate
\begin{align}\label{estimate-diffGrenier}
   \forall t \in [0,T_N], \qquad  \left\Vert\begin{pmatrix}
        \varrho_{\mathrm{app}}^{(N)}(t)-\rho_s \\
        \omega_{\mathrm{app}}^{(N)}(t)-U_s'
    \end{pmatrix} \right\Vert_{\H^{\ell+1} \times \H^{\ell}} 
    &\leq 1.
\end{align}
\begin{proof}
By definition of $(\varrho_{\mathrm{app}}^{(N)}, \omega_{\mathrm{app}}^{(N)})$, we have for all $t \in [0,T_N]$
\begin{align*}
    \left\Vert\begin{pmatrix}
        \varrho_{\mathrm{app}}^{(N)}(t)-\rho_s \\
        \omega_{\mathrm{app}}^{(N)})(t)-U_s'
    \end{pmatrix} \right\Vert_{\H^{\ell+1} \times \H^{\ell}} 
    &\leq \sum_{j=1}^N \delta^j\left\Vert \begin{pmatrix}
\varrho_j(t)\\
\omega_j(t)
\end{pmatrix}\right\Vert_{\H^{\ell+1} \times \H^{\ell}} 
 \leq \sum_{j=1}^N \delta^j C_{N,j} \e^{\Lambda jt } 
 \leq K_N \frac{\delta \e^{\Lambda t }}{1-\delta \e^{\Lambda t }}, 
\end{align*}
where we have used the estimate \eqref{estimate-iteratesemigroupGrenier} on each iterate for $1 \leq j \leq N$. We get  
\begin{align*}
   \left\Vert\begin{pmatrix}
        \varrho_{\mathrm{app}}^{(N)}(t)-\rho_s \\
        \omega_{\mathrm{app}}^{(N)})(t)-U_s'
    \end{pmatrix}\right\Vert_{\H^{\ell+1} \times \H^{\ell}} \leq K_N \frac{\delta \e^{\Lambda t }}{1-\delta \e^{\Lambda t}} \leq K_N \dfrac{1}{2 K_N} \frac{1}{1-\frac{1}{2}} =1.
\end{align*}
This yields the claimed inequality.
\end{proof}

\subsubsection*{Step 3c - Estimate of $R^{(N)}_{\mathrm{app}}$} 

There exists a constant $K'_N \geq 0$ such that
\begin{align}\label{estimate-remainderGrenier}
    \forall t \in [0,T_N], \qquad \Vert R^{(N)}_{\mathrm{app}}(t) \Vert_{\H^\ell \times \H^{\ell-1}} \leq  K'_N (\delta \e^{\Lambda t})^{N+1}.
\end{align}
\begin{proof}
By definition of the remainder  $R^{(N)}_{\mathrm{app}}$, we have
\begin{align*}
    \Vert R^{(N)}_{\mathrm{app}}(t) \Vert_{\H^\ell \times \H^{\ell-1}} \leq \sum_{\substack{1 \leq k, k' \leq N \\  N+1 \leq k+k' \leq 2N }} \delta^{k+k'} \left\Vert \begin{pmatrix}
        W_{k} \cdot \nabla \varrho_{k'} \\
        W_{k} \cdot \nabla \omega_{k'} 
    \end{pmatrix} \right\Vert_{\H^\ell \times \H^{\ell-1}}.
\end{align*}
By mimicking what has been done in Step 3a, we can obtain the bound
\begin{align*}
    \left\Vert \begin{pmatrix}
        W_{k} \cdot \nabla \varrho_{k'} \\
        W_{k} \cdot \nabla \omega_{k'} 
    \end{pmatrix} \right\Vert_{\H^\ell \times \H^{\ell-1}} \leq C_k C_{k'} \e^{(k+k') t  \Lambda}.
\end{align*}
Defining $K'_N :=\underset{N+1 \leq p \leq 2N}{\sup}\frac{1}{2}\sum_{\substack{1 \leq k, k' \leq N \\   k+k'=p }} C_k C_{k'}$, we then deduce that 
\begin{align*}
    \Vert R^{(N)}_{\mathrm{app}}(t) \Vert_{\H^1 \times \Ld^2} \leq \sum_{\substack{1 \leq k, k' \leq N \\  N+1 \leq k+k' \leq 2N }} \delta^{k+k'} C_k C_{k'} \e^{(k+k') t  \Lambda}  \leq \frac{1}{2} K'_N \sum_{p=N+1}^{+\infty} \delta^\ell \e^{\Lambda \ell  t}  \leq K'_N (\delta \e^{\Lambda t})^{N+1}.
\end{align*}
This concludes the proof.
\end{proof}

\medskip

Now, we estimate the difference between the approximate solution $(\varrho_{\mathrm{app}}^{(N)}, \omega_{\mathrm{app}}^{(N)})$ and a solution $(\varrho, \omega)$ of the Euler-Boussinesq system
\begin{align}\label{eq:EB-LyapunovInstab}
\begin{cases}
    \partial_t \varrho +u \partial_x \varrho+ v \partial_z \varrho=0, \\
\partial_t \omega +u \partial_x \omega + v \partial_z \omega =\partial_x \varrho, \\
(u,v)= (\partial_z \varphi, -\partial_x \varphi),
\end{cases} \ \ 
\text{where} \qquad\begin{cases}
        \Delta \varphi=\omega,\\
        \varphi_{\mid z=\pm 1}=0,
    \end{cases} 
\end{align}
on the domain $\T_M \times (-1,1)$ and with initial data $(\rho_s, U_s') + \delta (\varrho_{\mathrm{eig}}, \omega_{\mathrm{eig}})$.

\subsubsection*{Step 3d - Estimate of $(\varrho-\varrho_{\mathrm{app}}^{(N)}, \omega- \omega_{\mathrm{app}}^{(N)})$ }

For all $p \geq 3$, we have 
\begin{align}\label{est:step3d}
    \forall t \in [0,T_N], \qquad \Vert  \omega(t)-\omega_{\mathrm{app}}(t) \Vert_{\H^p}+\Vert  \varrho(t)-\varrho_{\mathrm{app}}(t) \Vert_{\H^{p+1}} \lesssim K''_N(\delta \e^{\Lambda t})^{N+1},
\end{align}
where
$T_N>0$ is chosen so that the following hold:
\begin{align}\label{eq:TN}
    \delta \e^{T_N \Lambda} \leq \eps_0, \qquad \eps_0 \leq \min\left( \eps_0',\frac{1}{2}, \frac{1}{2 K''_N}\right), \qquad K''_N=\max(K_N,K'_N),
\end{align}
where $\eps_0'$ is small enough and the solution $(\varrho, \omega)$ of the Euler-Boussinesq equations \eqref{eq:EB-LyapunovInstab} with initial data $(\rho_s, U_s')+\delta (\varrho_{\mathrm{eig}}, \omega_{\mathrm{eig}})$ exists until time $T_N$.
\begin{proof}[Proof of estimate \eqref{est:step3d}]
Let us set
\begin{align}
    (\widetilde{\varrho}, \widetilde{\omega}):=(\varrho-\varrho_{\mathrm{app}}^{(N)}, \omega- \omega_{\mathrm{app}}^{(N)}).
\end{align}
It satisfies
\begin{equation*}
      \begin{aligned}
      \partial_t \widetilde{\varrho} +W[\omega_{\mathrm{app}}^{(N)}+ \widetilde{\omega}]\cdot \nabla \widetilde{\varrho}&=-W[ \widetilde{\omega}] \cdot \nabla \varrho_{\mathrm{app}}^{(N)} -R^{(N)}_{1,\mathrm{app}},  \\[1mm]
\partial_t \widetilde{\omega} +W[\omega_{\mathrm{app}}^{(N)}+ \widetilde{\omega}] \cdot \nabla \widetilde{\omega} +\partial_x \widetilde{\varrho}&=-W[ \widetilde{\omega}] \cdot \nabla \omega_{\mathrm{app}}^{(N)}-R^{(N)}_{2,\mathrm{app}},
\end{aligned}
\end{equation*} 
with $(\widetilde{\varrho}, \widetilde{\omega})_{\mid t=0}=(0,0)$ and the notation $W[ \Omega]=(\partial_z \Delta^{-1} \Omega,-\partial_x \Delta^{-1} \Omega)$. As the approximate solution is known, we can leverage the standard local-in-time theory of strong well-posedness for the Euler-Boussinesq equations (see e.g., \cite{Chae}, which can be easily adapted to $\T_M \times (-1,1)$ thanks to the proofs in \cite{nersisyan2013stabilization}). This allows us to establish that the aforementioned equation has a unique local-in-time solution in high-order Sobolev space.

Let us now perform an energy estimate on the previous system, in $\H^{p+1} \times \H^p$. Using standard commutator estimates with respect and Sobolev embedding, we can prove that for all $p \geq 3$, it holds
\begin{align*}
    \dfrac{\mathrm{d}}{\mathrm{d}t} \Vert  \widetilde{\omega} \Vert_{\H^p}^2 &\lesssim \left(\Vert \omega_{\mathrm{app}}^{(N)}\Vert_{\H^{p+3}}+\Vert \widetilde{\omega} \Vert_{\H^{p-1}} \right)\Vert \widetilde{\omega}\Vert_{\H^p}^2 + \Vert \widetilde{\varrho}\Vert_{\H^{p+1}}\Vert \widetilde{\omega} \Vert_{\H^p}  +\Vert R^{(N)}_{2,\mathrm{app}} \Vert_{\H^p} \Vert \widetilde{\omega} \Vert_{\H^p},\\
    \dfrac{\mathrm{d}}{\mathrm{d}t} \Vert \widetilde{\varrho} \Vert_{\H^{p+1}}^2 &\lesssim \Vert \varrho_{\mathrm{app}}^{(N)}\Vert_{\H^{p}} \Vert \widetilde{\varrho} \Vert_{\H^{p+1}}^2 + 
    \Vert \varrho_{\mathrm{app}}^{(N)} \Vert_{\H^{p+4}}\Vert \widetilde{\omega} \Vert_{\H^{p}}\Vert \widetilde{\varrho} \Vert_{\H^{p+1}}+\Vert \widetilde{\omega} \Vert_{\H^p}\Vert \widetilde{\varrho}\Vert_{\H^{p+1}}^2+\Vert R^{(N)}_{1,\mathrm{app}} \Vert_{\H^{p+1}} \Vert \widetilde{\varrho} \Vert_{\H^{p+1}},
\end{align*}
where $\lesssim$ only involves universal constants.
Introducing
\begin{align*}
    y(t):=\Vert  \widetilde{\omega}(t) \Vert_{\H^p}+\Vert  \widetilde{\varrho}(t) \Vert_{\H^{p+1}},
\end{align*}
and $\omega_s:=U'_s$, the above inequalities yield 
\begin{align*}
    y'(t) & \leq \left(1+\Vert \omega_{\mathrm{app}}^{(N)}(t) \Vert_{\H^{p+3}}\Vert+ \Vert \varrho_{\mathrm{app}}^{(N)}(t) \Vert_{\H^{p+4}}+\Vert\Vert  \widetilde{\omega}(t) \Vert_{\H^p} \right)y(t)  + \Vert R^{(N)}_{1,\mathrm{app}}(t) \Vert_{\H^{p+1}}+\Vert R^{(N)}_{2,\mathrm{app}}(t) \Vert_{\H^{p}} \\
    & \leq  \left(1+\mathrm{C}_{\rho_s,\omega_s }+\Vert \omega_{\mathrm{app}}^{(N)}(t)-\omega_s\Vert_{\H^{p+3}}\Vert+ \Vert \varrho_{\mathrm{app}}^{(N)}(t)-\rho_s \Vert_{\H^{p+4}}+\Vert  \widetilde{\omega}(t) \Vert_{\H^p} \right)y(t) \\
    & \quad + \Vert R^{(N)}_{1,\mathrm{app}}(t) \Vert_{\H^{p+1}}+\Vert R^{(N)}_{2,\mathrm{app}}(t) \Vert_{\H^{p}},
\end{align*}
for some constant $\mathrm{C}_{\rho_s,\omega_s}$ only depending on the profile. Now, using \eqref{estimate-diffGrenier}-\eqref{estimate-remainderGrenier} with $\ell \geq p+1$, there exists $C>0$ such that 
\begin{align*}
    y'(t) \leq C\left(2+\mathrm{C}_{\rho_s,\omega_s }+\Vert  \widetilde{\omega}(t) \Vert_{\H^p} \right)y(t) + K''_N (\delta \e^{\Lambda t})^{N+1},
\end{align*}
for all times $t \in [0, T_N]$ with $T_N$ as in \eqref{eq:TN}. Let us now set \begin{align*}
    N:=\left\lfloor \dfrac{C(4+\mathrm{C}_{\rho_s,\omega_s})}{\Lambda} \right\rfloor+1.
\end{align*}
We perform a bootstrap argument. Let us introduce 
\begin{align}\label{bootstrap-Grenier}
   \widetilde{T}_N:= \sup \big\lbrace T>0, \qquad \mathrm{sup}_{0 \leq t \leq T} \Vert  \widetilde{\omega}(t) \Vert_{\H^p} \leq K''_N (\delta \e^{\Lambda t})^N \big\rbrace .    
\end{align}
By the local theory, we have $\widetilde{T}_N>0$. Moreover, by definition of $T_N$ with $N \geq 1$, we also get 
\begin{align*}
     K''_N (\delta \e^{ \Lambda t})^N \leq 1.
\end{align*}
If we have $\widetilde{T}_N<T_N$, then by Grönwall's lemma we can infer that for all $t \in [0,\widetilde{T}_N]$
\begin{align*}
    y(t) \leq K''_N \delta^{N+1} \int_0^t \e^{C(4+\mathrm{C}_{\rho_s,\omega_s})(t-\tau)} \e^{ (N+1)\Lambda \tau} \, \mathrm{d}\tau.
\end{align*}
We deduce that 
\begin{align*}
    y(t) & \leq K''_N \delta^{N+1} \e^{N \Lambda t} \int_0^t \e^{\Lambda \tau} \, \mathrm{d \tau} \leq \dfrac{K''_N}{\Lambda} (\delta \e^{\Lambda t})(\delta \e^{ \Lambda t})^N  \leq \dfrac{K''_N}{\Lambda} \eps_0  (\delta \e^{ \Lambda t})^N.
\end{align*}
By employing a continuation argument, we can refine the earlier bound by selecting $\eps_0$ to be sufficiently small, leading to a contradiction with the definition of $\widetilde{T}_N$. Consequently, we conclude that $\widetilde{T}_N \geq T_N$. Revisiting the previous estimate, we eventually arrive at the stated result.
\end{proof}

\subsubsection*{Step 3e - Conclusion} Let us now conclude the proof. To simplify the presentation, let us set $Z:=(\varrho, \omega)$, which is the solution of the Euler-Boussinesq equation \eqref{eq:EB-LyapunovInstab} with initial data $(\rho_s, U_s')+\delta (\varrho_{\mathrm{eig}}, \omega_{\mathrm{eig}})$ (given by the \textit{Step 3d}). By triangle inequality and by definition of the approximate solution $Z_{\mathrm{app}}$ (see \eqref{def-SolAppGrenier}), we have for all $t \in [0,T_N]$, 
\begin{align*}
    \Vert Z(t) - (\rho_s, U_s') \Vert_{\Ld^2 \times \H^{-1}} &= \Vert Z(t) -Z_{\mathrm{app}}(t)+Z_{\mathrm{app}}(t) - (\rho_s, U_s') \Vert_{\Ld^2 \times \H^{-1}} \\
    &\geq  \delta \Vert Z_1(t) \Vert_{\Ld^2 \times \H^{-1}}-\sum_{k=2}^N \delta^k  \Vert Z_k(t) \Vert_{\H^{p+1} \times \H^{p}} -\Vert Z(t)-Z_{\mathrm{app}}(t) \Vert_{\H^{p+1} \times \H^p},
\end{align*}
where $Z_1$ (see \eqref{def-SolApp1Grenier}) is normalized. Relying on the estimates \eqref{estimate-iteratesemigroupGrenier} and \eqref{estimate-diffGrenier} for $\ell \geq  p+1$ and $p$ large enough, we infer that 
\begin{align*}
    \Vert Z(t) - (\rho_s, U_s') \Vert_{\Ld^2 \times \H^{-1}} &\geq  \delta \e^{\Lambda t}-\sum_{k=2}^N C_K \delta^k  \e^{k \Lambda t} -K''_N(\delta \e^{ \Lambda t})^{N+1} \\
    &\geq  \delta \e^{\Lambda t}-2 K''_N (\delta \e^{ \Lambda t})^2 -K''_N(\delta \e^{\Lambda t})^{N+1} \\
    &= \delta \e^{\Lambda t} \left[ 1- 2 K_N'' \delta \e^{\Lambda t}-K_N'' (\delta \e^{ \Lambda t})^{N}  \right].
\end{align*}
Since $N \geq 1$, choosing $\delta_0>0$ small enough so that 
\begin{align*}
    2 K_N'' (4 \delta_0)+K_N''(4 \delta_0)^N \leq \frac{1}{2}, \qquad 4 \delta_0 \leq \min\left( \eps_0,\frac{1}{2}, \frac{1}{2 K''_N}\right), 
\end{align*}
and a time $T_N^\star$ such that
\begin{align}\label{ineq-choosefinaltimeGrenier}
    2 \delta_0 \leq \delta \e^{\Lambda T_N^\star} \leq 4 \delta_0,
\end{align}
it follows that $T_N^\star \leq T_N$ and, for all $t \in [0,T_N^\star]$, the following lower bound holds:
\begin{align*}
    \Vert Z(t) - (\rho_s, U_s') \Vert_{\Ld^2 \times \H^{-1}} \geq \frac{1}{2} \delta \e^{\Lambda t} \geq \delta_0.
\end{align*}
Therefore, there exists $m_0>0$ such that
\begin{align*}
  \underset{t \in [0,\T_N^\star]}{\sup} \Vert  (\varrho, u,v)(t)-(\rho_s, U_s,0)  \Vert_{\Ld^2} \geq m_0, 
\end{align*}
meanwhile for all $p \geq 0$
\begin{align*}
\Vert (\varrho, u,v)_{\mid t=0}-(\rho_s, U_s,0) \Vert_{\mathrm{H}^p} \lesssim \delta.
\end{align*}

\subsubsection*{Step 3f - The complex valued case}: Now, we address the case where $\Lambda$ may not be a real number; for instance, we can assume $\mathrm{Im}(\Lambda) > 0$. Note that $\mathrm{Re}(\Lambda) > 0$ since we consider a maximal unstable eigenvalue. Denoting $Z_{\mathrm{eig}}$ as the associated eigenvector, we can assume without loss of generality that $\mathrm{Re}(Z_{\mathrm{eig}}) \neq 0$ (given that the linearized operator only involves real-valued functions and coefficients).
In this scenario, we replace \eqref{def-SolApp1Grenier} by
\begin{align*}
    Z_1(t) = (\varrho_1, \omega_1)(t) := \mathrm{Re}(\e^{\Lambda t} Z_{\mathrm{eig}}),
\end{align*}
Let us also assume that $\Vert \mathrm{Re}(Z_{\mathrm{eig}}) \Vert_{\Ld^2 \times \H^{-1}} = 1$. We can essentially perform the previous \textit{Steps 1-2-3}. The main changes primarily come from \textit{Step 3}, where we need to consider the estimate from below:
\begin{align*}
    \Vert Z(t) - (\rho_s, U_s') \Vert_{\Ld^2 \times \H^{-1}}
    \geq \delta \e^{\mathrm{Re}(\Lambda)t} \left[ \Vert \mathrm{Re}(\e^{i \mathrm{Im}(\Lambda)t} Z_{\mathrm{eig}}) \Vert_{\Ld^2 \times \H^{-1}} - 2 K_N'' \delta \e^{\mathrm{Re}(\Lambda)t} - K_N'' (\delta \e^{\mathrm{Re}(\Lambda)t})^{N} \right].
\end{align*}
Observe that for times $t$ of the form $t = \frac{2\pi k}{\mathrm{Im}(\Lambda)}$ (with $k \in \N$), the first term inside the brackets above can be bounded from below as previously. We need to ensure that the size of the time interval, satisfying the bounds \eqref{ineq-choosefinaltimeGrenier}, is at least of size greater than $\frac{2\pi}{\mathrm{Im}(\Lambda)}$. Taking $\delta > 0$ small enough allows us to enforce such a condition.

\medskip

\subsubsection*{Step 4: Back to the original variables}
Let us finally show that the instability proven in Theorem \ref{thm:nonlin-instabilityEB} implies the failure of the hydrostatic limit stated in Theorem \ref{thm-invalidHydrostatlim}.

With fixed parameters $p, k > 0$ and $M > 0$ taken sufficiently large, we apply Theorem
\ref{thm:nonlin-instabilityEB} with 
\begin{align*}
    \delta=\eps^r, \qquad \eps \in (0,1), \qquad r>p+k+1
\end{align*}
to obtain a solution $(\varrho_\eps, u_\eps, v_\eps)$ to the Euler-Boussinesq system \eqref{eq:EB-scaled} on $[0,\tau_\eps] \times \T_M \times (-1,1)$ where $\tau_\eps =\mathcal{O}( \vert \log \, \delta \vert)=\mathcal{O}( \vert \log \, \eps \vert)$ and such that
\begin{align*}
\Vert (\varrho_\eps, u_\eps, v_\eps)_{\mid t=0}-(\rho_s, U_s,0) \Vert_{\mathrm{H}^p(\T_M \times (-1,1))} \leq \eps^r, \\[1mm]
  \underset{t \in [0,T_\delta]}{\sup} \Vert  (\varrho_\eps, u_\eps, v_\eps)(t)-(\rho_s, U_s,0)  \Vert_{\Ld^2(\T_M \times (-1,1))} \geq m,
\end{align*}
for some $c>0$ independent of $\eps$. 

Now, let us come back to the original equation defined on $\T \times (-1,1)$ by employing the change of variable/gluing argument from \textit{Step $3$} and \textit{Step $4$}. By using the relations \eqref{eq:gluing}--\eqref{eq:rescaling}, we can first obtain a solution
$(\varrho_\eps^{\mathrm{true}}, u_\eps^{\mathrm{true}}, v_\eps^{\mathrm{true}})$ of the Euler-Boussinesq system on $[0,\eps \tau_\eps] \times \T \times (-1,1)$. We then set $T_\eps:=\eps \tau_\eps=\mathcal{O}( \eps \vert \log \, \eps \vert)$. The following holds:
\begin{align*}
    \Vert (\varrho_\eps^{\mathrm{true}}, u_\eps^{\mathrm{true}})_{\mid t=0}-(\rho_s, U_s) \Vert_{\mathrm{H}^p(\T \times (-1,1))} &\lesssim \eps^{-p}\Vert (\varrho_\eps, u_\eps)_{\mid t=0}-(\rho_s, U_s) \Vert_{\mathrm{H}^p(\T_M \times (-1,1))} \lesssim \eps^{-p+r}, \\[2mm]
    \Vert  {v_\eps^{\mathrm{true}}}_{\mid t=0} \Vert_{\mathrm{H}^p(\T \times (-1,1))} &\lesssim \eps^{-p-1}\Vert {v_\eps}_{\mid t=0}\Vert_{\mathrm{H}^p(\T_M \times (-1,1))} \lesssim \eps^{-p-1+r}.
\end{align*}
As $r>p+1+k$, we obtain
\begin{align*}
\Vert (\varrho_\eps^{\mathrm{true}}, u_\eps^{\mathrm{true}}, v_\eps^{\mathrm{true}})_{\mid t=0}-(\rho_s, U_s,0) \Vert_{\mathrm{H}^p(\T \times (-1,1))} \lesssim \eps^k.
\end{align*}
We also have
\begin{align*}
    \Vert (\varrho_\eps^{\mathrm{true}}, u_\eps^{\mathrm{true}})_{\mid t=T_\eps}-(\rho_s, U_s) \Vert_{\Ld^2  (\T \times (-1,1))} &= \Vert (\varrho_\eps, u_\eps)_{\mid t=\tau_\eps}-(\rho_s, U_s) \Vert_{\Ld^2 (\T_M \times (-1,1))} \geq m, \\[2mm]
    \Vert  {v_\eps^{\mathrm{true}}}_{\mid t=T_\eps} \Vert_{\Ld^2 \cap (\T \times (-1,1))} &= \eps^{-1}\Vert {v_\eps}_{\mid t=\tau_\eps }\Vert_{\Ld^2 (\T_M \times (-1,1))} \geq \eps^{-1} m,
\end{align*}
so that \begin{align*}
\Vert (\varrho_\eps^{\mathrm{true}}, u_\eps^{\mathrm{true}}, v_\eps^{\mathrm{true}})(T_\eps)-(\rho_s, U_s,0) \Vert_{\Ld^{2} (\T \times (-1,1))} \geq c,
\end{align*}
for $\eps$ small enough.
This concludes the proof of Theorem \ref{thm-invalidHydrostatlim}.

\section{Nonlinear ill-posedness for the hydrostatic Euler-Boussinesq system}\label{Section:IllPosed}
Implementing the method established by Han-Kwan and Nguyen in \cite{HKN}, we prove Theorem \ref{thm-illposednesNonlin}, first introducing the following analytic setting.

\begin{Def}
For any $\delta, \delta'>0$, we introduce the space
\begin{align}
X_{\delta, \delta '}= \lbrace f=f(x,z) \mid \Vert f \Vert_{X_{\delta, \delta '}}<\infty \rbrace,
\end{align}
where for any analytic function $f=f(x,z) \in \R$ with $(x,z) \in \T \times (-1,1)$, we define
\begin{align}
\Vert f \Vert_{\delta, \delta '}:=\sum_{n \in \Z} \sum_{k \geq 0} \Vert \partial_z^k f_n \Vert_{\Ld^2(-1,1)} \frac{\vert \delta ' \vert^k}{k!} \e^{\delta n}, \qquad f_n(z)=\langle f(\cdot, z) , \e^{i n \cdot} \rangle_{\Ld^2(\T)}.
\end{align}
We also denote by $X_{\delta '}$ the space of $z$-analytic function $g=g(z)$ such that 
\begin{align}
\Vert g \Vert_{\delta '}:= \sum_{k \geq 0} \Vert \partial_z^k g \Vert_{\Ld^2(-1,1)} \frac{\vert \delta ' \vert^k}{k!}<+\infty.
\end{align} 
\end{Def}
In the above, the parameters $\delta$ and $\delta'$ can be interpreted as radii of convergence. The parameter $\delta$ quantifies the analyticity in $x$, while the parameter $\delta'$ quantifies the one in $z$. We extend the definition of the previous norms to vectors of functions, in the standard manner.

\medskip

We have the following useful properties (see e.g. \cites{HKN,baradat2020nonlinear}).
\begin{Lem}\label{LM:propXanalytic}
Let $0 \leq \delta_1 < \delta_2$ and $0 \leq \delta_1' < \delta_2'$. The following properties hold.
\begin{enumerate}
    \item For any function $f=f(x,z)$, we have
\begin{align*}
\Vert f \Vert_{\delta_1, \delta_1'} \leq \Vert f \Vert_{\delta_2, \delta_2'}, \qquad \Vert \partial_x f \Vert_{\delta_1, \delta_1'}  \leq \frac{\delta_2}{\delta_2-\delta_1}\Vert f \Vert_{\delta_2, \delta_1'}, \qquad \Vert \partial_z f \Vert_{\delta_1, \delta_1'}  \leq \frac{\delta_2'}{\delta_2'-\delta_1'}\Vert f \Vert_{\delta_1, \delta_2'}.
\end{align*}
\item For any functions $f(z), g(z)$, we have
\begin{align}
    \Vert f g \Vert_{ \delta_1'} \leq \Vert f  \Vert_{ \delta_1'}\Vert g \Vert_{ \delta_1'}.
\end{align}
\end{enumerate}
\end{Lem}

Following \cite{HKN}, we linearize the equation around the unstable profile $(\rho_s(z), U_s'(z))$ from Theorem \ref{thm:growingmodHydrostat} in Section \ref{Subsec:UnboundedSpecHydrost}, expressing it in the hyperbolic scaling in the density-vorticity formulation. For $\eps>0$, we set
\begin{align*}
\mathrm{R}(t,x,z)&=\rho_s(z)+ \varrho\left( \frac{t}{\eps},\frac{x}{\eps}, z\right),\\ 
\Omega(t,x,z)&=U_s'(z)+\omega\left( \frac{t}{\eps},\frac{x}{\eps}, z\right), \\ 
\Phi(t,x,z)&=\int_{-1}^z U_s(\theta) \, \mathrm{d}\theta+ \varphi\left( \frac{t}{\eps},\frac{x}{\eps}, z \right),
\end{align*}
where $(R, \Omega, \Phi)$ is any solution to the hydrostatic system \eqref{eqintro:BHyd-E-vort}. We then consider the coordinates 
$$(\tau,\xi,z):=\left( \frac{t}{\eps},\frac{x}{\eps}, z \right).$$ Note that, thanks to the homogeneity of the equations, the following holds
\begin{align*}
\partial_\tau \begin{pmatrix}
\varrho\\
\omega
\end{pmatrix}=\mathscr{L}_{\rho_s,U_s}\begin{pmatrix}
\varrho\\
\omega
\end{pmatrix}+ \begin{pmatrix}
-\partial_\xi\varrho  \partial_z \varphi + \partial_\xi \varphi \partial_z \varrho \\[1mm]
-\partial_z \varphi \partial_\xi \omega+\partial_\xi \varphi \partial_z \omega
\end{pmatrix},
\end{align*}
where $\mathscr{L}_{\rho_s,U_s}$ is defined in \eqref{def:LinOp-Hydrostat}. Since the nonlinearity is quadratic with respect to the derivatives of $(\varrho, \omega)$, we consider the extended unknown
\begin{align*}
Y:=(\varrho, \omega, \partial_\xi \varrho, \partial_\xi \omega, \partial_z \varrho, \partial_z \omega),
\end{align*}  
the extended linear operator defined by 
\begin{align}\label{eq:defextendedOpL}
\mathrm{L}:=
\begin{pNiceArray}{cc|cc|cc}
   \Block{2-2}<\large>{\mathscr{L}_{\rho_s,U_s}} & &   \Block{2-2}<\large>{0_{2 \times 2}} & & \Block{2-2}<\large>{0_{2 \times 2}} & \\
   & & & & & \\ \hline
   \hspace{5mm} \Block{2-2}<\large>{0_{2 \times 2}} \hspace{5mm} & &   \Block{2-2}<\large>{\mathscr{L}_{\rho_s,U_s}} & & \Block{2-2}<\large>{0_{2 \times 2}} & \\
   & & & & & \\ \hline
   \hspace{5mm} \Block{2-2}<\large>{0_{2 \times 2}} \hspace{5mm} & & -U_s' &  \rho_s' \partial_z \varphi[\cdot] +\rho_s'' \varphi[\cdot]  & -U_s \partial_\xi & 0 \\[2mm]
&  & 0 & -U_s'+ U_s'' \partial_z \varphi[\cdot] +U_s''' \varphi[\cdot] & \partial_\xi  & -U_s \partial_\xi
\end{pNiceArray},
\end{align}
and the extended bilinear operator $\mathrm{Q}$ acting on $Y=(\varrho_1, v_1, \varrho_2,v_2, \varrho_3,v_3)$ and $Y'=(r_1, w_1, r_2, w_2, r_3,w_3)$:
\begin{align}\label{eq:defextendedNonlinQ}
\mathrm{Q}(Y,Y'):=\begin{pmatrix} 
-\partial_z \varphi[v_1] r_2 + \varphi[v_2] r_3 \\[1.5mm]
-\partial_z \varphi[v_1] w_2 + \varphi[v_2] w_3 \\[1.5mm]
\partial_z \varphi[v_2] r_2 + \partial_\xi \varphi[v_2] r_3-\varphi[v_3] \partial_\xi r_2+\varphi[v_2] \partial_\xi r_3 \\[1.5mm]
\partial_z \varphi[v_2] w_2 + \partial_\xi \varphi[v_2] w_3-\varphi[v_3] \partial_\xi w_2+\varphi[v_2] \partial_\xi w_3 \\[1.5mm]
-v_1 r_2 +\varphi[v_3] \partial_z r_2 + \partial_z \varphi[v_2] r_3+\varphi[v_2] \partial_z r_3 \\[1.5mm]
-v_1 w_2 +\varphi[v_3] \partial_z w_2 + \partial_z \varphi[v_2] w_3+\varphi[v_2] \partial_z w_3
\end{pmatrix},
\end{align}
where $\varphi[\omega]$ solves
\begin{align}\label{eq:psi-hydro}
\begin{cases}
        \partial_z^2 \varphi[\omega]=\omega, \\  
        \varphi[\omega]_{\mid z=\pm 1}=0.
    \end{cases}
\end{align}
Note that $\partial_\xi \varphi[\omega]=\varphi[\partial_\xi \omega]$, but not necessarily $\partial_z \varphi[\omega]=\varphi[\partial_z \omega]$ (because of the boundary condition). All in all, we get
\begin{align*}
\partial_\tau Y- \mathrm{L}Y=\mathrm{Q}(Y,Y), \qquad Y=(\varrho, \omega, \partial_\xi \varrho, \partial_\xi \omega, \partial_z \varrho, \partial_z \omega).
\end{align*}
To apply the method of \cite{HKN}, we need to verify the following hypotheses, restated here for clarity.
\begin{enumerate} [label=(H\arabic*), ref=(H\arabic*)]
\item\label{hypH1} \textbf{Structural assumption}: $[\partial_\xi, \mathrm{L} ]=0$.
\item\label{hypH2}  \textbf{Strong instability} for the linearized part: the operator $\mathrm{L}$ has an unstable eigenvalue. Namely, the set 
\begin{align*}
    \Sigma^+:= \left\lbrace \lambda \in \C \mid \mathrm{Re}(\lambda)>0 \ \text{and} \ \exists  (\mathrm{\rho}, \mathrm{\omega}) \neq 0,  \mathrm{L}Y=\lambda Y \right\rbrace
\end{align*}
is not empty.
Because of the fast variable rescaling, it amounts to say that the original linearized operator has an unstable unbounded spectrum.
\item\label{hypH3} \textbf{Loss of analyticity in tangential direction}: the semigroup  $\e^{\mathrm{L}s}$ is well-defined on $X_{\delta, \delta'}$ for $\delta'$ small enough and satisfies the following.

There exists $\gamma_0>0$ such that:
\begin{itemize}
    \item For all $\eta>0$, there exists $k_0 \in [1, +\infty) \cap \Z $ and $\lambda_0 \in \Sigma^+$ such that 
    \begin{align}
        \left\vert \frac{\mathrm{Re} \, \lambda_0}{k_0}-\gamma_0 \right\vert \leq \eta.
    \end{align}
    Furthermore, there exists an eigenfunction $Z$ associated to $\lambda_0$ with $\Vert Z \Vert_{\delta, \delta_0'}<\infty$ for any $\delta>0$ and for some $\delta_0'>0$.
    \item For all $\gamma > \gamma_0$, there exists $C_{\gamma}>0$ and $\delta_1'>0$ such that if $\delta-\gamma s>0$ and $\delta' \in (0, \delta_1']$ then
    \begin{align}
        \forall F \in X_{\delta, \delta'}, \qquad \Vert \e^{s \mathrm{L}} F \Vert_{\delta-\gamma s, \delta'} \leq C_{\gamma} \Vert F \Vert_{\delta, \delta'}.
    \end{align}
\end{itemize}
\item\label{hypH4}  \textbf{Structure of the nonlinearity}: $\mathrm{Q}$ is bilinear and for all $\delta, \delta'$ and $f,g \in X_{\delta, \delta'}$ , we have
\begin{align}
     \Vert \mathrm{Q}(f,g) \Vert_{\delta, \delta'} \lesssim \Vert f \Vert_{\delta, \delta'} \left( \Vert \partial_\xi g \Vert_{\delta, \delta'} + \Vert \partial_z g \Vert_{\delta, \delta'}\right) +\Vert g \Vert_{\delta, \delta'} \left( \Vert \partial_\xi g \Vert_{\delta, \delta'} + \Vert \partial_z g \Vert_{\delta, \delta'}\right).
\end{align}
\end{enumerate}
\bigskip

Now, let us check the assumptions in the setting of Theorem \ref{thm-illposednesNonlin}. Firstly, \ref{hypH1} is trivially satisfied. Passing to \ref{hypH2}, consider an eigenfunction $(\mathrm{\rho}, \mathrm{\omega})$ of the operator $\mathscr{L}_{\rho_s,U_s}$ associated with an eigenvalue $\lambda$ as given by Theorem \ref{thm:growingmodHydrostat}.
Then, by \ref{hypH1} we have
\begin{align*}
\mathscr{L}_{\rho_s,U_s}\begin{pmatrix}
\mathrm{\rho}\\
\mathrm{\omega}
\end{pmatrix}=\lambda \begin{pmatrix}
\mathrm{\rho}\\
\mathrm{\omega}
\end{pmatrix}, \qquad \mathscr{L}_{\rho_s,U_s}\begin{pmatrix}
\partial_\xi \mathrm{\rho}\\
\partial_\xi \mathrm{\omega}
\end{pmatrix}=\lambda \begin{pmatrix}
\partial_\xi \mathrm{\rho}\\
\partial_\xi \mathrm{\omega}
\end{pmatrix}.
\end{align*}
Define $Y := (\mathrm{\rho}, \mathrm{\omega}, \partial_\xi \mathrm{\rho}, \partial_\xi \mathrm{\omega}, \partial_z \mathrm{\rho}, \partial_z \mathrm{\omega})$. By differentiating the second-to-last relation in $z$, we obtain $\mathrm{L}Y=\lambda Y$ with $\mathrm{L}$ defined in \eqref{eq:defextendedOpL}. This shows that \ref{hypH2} holds.

About \ref{hypH3}, the bilinear estimates in analytic norms follow from the Banach algebra property (see Lemma \ref{LM:propXanalytic}).

It remains to show that \ref{hypH3} is satisfied. First, note that, from the definition \eqref{def:Gamma0} of $\gamma_0$ in Corollary \ref{Coro:Gamma0}, we directly obtain the existence of $(\lambda_0, k_0)$ required in \ref{hypH3} by Theorem \ref{thm:growingmodHydrostat}. In fact, we can simply take $c_0$ such that $\mathrm{Im} \, c_0 = \gamma_0$ and any $\lambda_0 = -ik_0 c_0$. The analytic regularity of the associated eigenfunction follows from the regularity of the shear flow $U_s$ and the construction outlined in Section \ref{Subsec:UnboundedSpecHydrost}.

Checking the semigroup estimates for \ref{hypH3} requires a bit more work. First, let us recall the following technical lemma from \cite{HKN}.
\begin{Lem}\label{LM:analyticTransport}
Let $f=f(z)$ be an analytic solution of the transport equation
\begin{align*}
\partial_\tau f + inU_s(z)f=S,
\end{align*}
with a given source $S$ and $n \in \Z$. Then for any $\delta' >0$, we have 
\begin{align*}
\frac{\mathrm{d}}{\mathrm{d}\tau}\Vert f \Vert_{\delta'}  \leq \vert n \vert \delta' N(U_{s})_{\delta'}\Vert f \Vert_{\delta'}+ \Vert S \Vert_{\delta'},
\end{align*}
where
\begin{align*}
    N(U_{s})_{\delta'}:=\sum_{k \geq 0} \Vert \partial_z^k g \Vert_{\Ld^\infty(-1,1)} \frac{\vert \delta ' \vert^k}{k!}<\infty.
\end{align*}
\end{Lem}

To obtain \ref{hypH3}, we will need the following result.
\begin{Prop}\label{Prop:SemigroupEstimBigL}
Let $\delta, \delta'>0$ and recall the definition of $\gamma_0$ as in \eqref{def:Gamma0}. If $\delta$ and $\delta'$ are small enough, the operator $\mathrm{L}$ generates a semigroup  $\e^{\mathrm{L}\tau}$ on $X_{\delta, \delta'}$. Furthermore, for any $\Gamma >1+\gamma_0$, there exists $C_\Gamma>0$ such that
\begin{align*}
\forall Y \in X_{\delta, \delta'}, \qquad \Vert \e^{\mathrm{L} \tau} Y \Vert_{\delta-\Gamma \tau,\delta'} \leq C_\Gamma \Vert  Y \Vert_{\delta,\delta'}, 
\end{align*}
as long as
\begin{align*}
0 <\delta' \ll \gamma_0, \qquad \delta-\Gamma \tau >0.
\end{align*}
\end{Prop}
In view of the structure of the matrix operator $\mathrm{L}$, we first study the operator $\mathscr{L}_{\rho_s,U_s}$. Recall that by definition of the analytic norm, we have for $G(\tau)=\e^{(\mathscr{L}_{\rho_s,U_s})\tau} F$
\begin{align*}
\Vert  G(\tau) \Vert_{\delta-\Gamma \tau,\delta'}=\sum_{n \in \Z}  \Vert  G_n(\tau) \Vert_{\delta'}  \e^{(\delta-\Gamma \tau) n}, \qquad G_n(\tau,z)=\langle F(\tau,\cdot, z), \e^{i n \cdot} \rangle_{\Ld^2(\T)}.
\end{align*}
Decoupling in frequency $n \in \Z$, it is enough to focus on 
$$G_n=(\e^{(\mathscr{L}_{\rho_s,U_s})\tau} F)_n=\e^{\mathcal{L}_n \tau} F_n,$$

where for all $n \in \Z$ we use the notation 
$$\mathcal{L}_n:=(\mathscr{L}_{\rho_s,U_s})_n,$$
with the expression
\begin{align*}
\mathcal{L}_n\begin{pmatrix}
\varrho\\
\omega
\end{pmatrix}= \begin{pmatrix}
-in U_s \varrho +in  \rho_s'\varphi\\ 
-in U_s  \omega +in  U_s''(z)\varphi+ i n \varrho
\end{pmatrix}, \qquad \text{where} \qquad   \begin{cases}
        \partial_z^2 \varphi=\omega, \\ 
        \varphi_{\mid z=\pm 1}=0.
    \end{cases}
\end{align*}
We need the following key proposition.
\begin{Prop}\label{Prop:SemigroupEstimL_n}
Let $n \in \Z$, $ \delta'>0$ and $\gamma_0$ defined in \eqref{def:Gamma0}. If $\delta'$ is small enough, the operator $\mathcal{L}_n$ generates a semigroup  $\e^{\mathcal{L}_n \tau}$ on $X_{\delta'}$. Furthermore, for any $\Gamma >1+\gamma_0$, there exists $C_\gamma>0$ such that
\begin{align*}
\forall \tau \geq 0, \qquad \forall F \in X_{\delta'}, \qquad \Vert \e^{\mathcal{L}_n \tau} F \Vert_{\delta'} \leq C_\gamma \e^{\Gamma \vert n \vert  \tau} \Vert  F \Vert_{\delta'}, 
\end{align*}
as long as
\begin{align*}
0 <\delta \ll \gamma_0.
\end{align*}
\end{Prop}
\begin{proof}
We follow the approach of \cite{HKN}, based on resolvent estimates. First, we observe that by homogeneity and a rescaling in time, it is sufficient to consider the case $n=1$. We write
\begin{align*}
\mathcal{L}_1\begin{pmatrix}
\varrho\\
\omega
\end{pmatrix}&=
\mathbb{A}_{U_s}\begin{pmatrix}
\varrho\\
\omega
\end{pmatrix}+\varphi[\omega]\begin{pmatrix}
i\rho_s'  \\
i  U_s'' 
\end{pmatrix},
\end{align*}
with
\begin{align*}
    \mathbb{A}_{U_s}:=i\begin{pmatrix}
- U_s & 0\\
1 & - U_s
\end{pmatrix}.
\end{align*}
The multiplication operator $\mathbb{A}_{U_s}$ obviously generates a semigroup
\begin{align*}
    \e^{\mathbb{A}_{U_s} \tau}=\e^{-i U_s \tau} \begin{pmatrix}
1 & 0\\
is & 1
\end{pmatrix},
\end{align*}
with a bound on its operator norm (on $\Ld^2$) as
\begin{align}
    \Vert \e^{\mathbb{A}_{U_s} \tau} \Vert \leq 1 +\tau.
\end{align}
On the other hand, the operator
\begin{align*}
    K:(\varrho, \omega) \mapsto \varphi[\omega]\begin{pmatrix}
i\rho_s'  \\
i  U_s'' 
\end{pmatrix},
\end{align*}
with $\varphi[\omega]$ in \eqref{eq:psi-hydro}, is bounded on $\Ld^2 \times \Ld^2$. By standard  Hille-Yosida theorems (see \cite{pazy2012semigroups}*{Chapter 3, Theorem 1.1}), we know that the operator $\mathcal{L}_1$ generates a semigroup on $\Ld^2$. Furthermore, by the inverse Laplace transform formula, we have 
\begin{align}\label{eq:LaplaceTransform}
    \e^{\mathcal{L}_1 \tau} F= \frac{1}{2 i \pi} \mathrm{P.V.} \,\int_{\gamma-i \infty}^{\gamma + i \infty} \e^{\lambda \tau} (\lambda- \mathcal{L}_1)^{-1} F \, \mathrm{d}\lambda,
\end{align}
where $\gamma$ is large enough so that the contour of integration is on the right of $\mathrm{Sp}(\mathcal{L}_1)$. 
Since the operator $K : \Ld^2 \times \Ld^2 \rightarrow \Ld^2 \times \Ld^2$ is a compact operator on $\Ld^2(-1,1) \times \Ld^2(-1,1)$ 
and since $\mathbb{A}_{U_s}$  a multiplication operator with numerical range in $i\R$,  we can apply Weyl's theorem to $\mathcal{L}_1=\mathbb{A}_{U_s}+K$  to get
$$\mathrm{Sp}_{\mathrm{ess}}(\mathcal{L}_1)=\mathrm{Sp}_{\mathrm{ess}}(\mathbb{A}_{U_s}) \subset i \R.
$$
The remaining part of the spectrum is composed of either stable discrete spectrum (located in $\lbrace \mathrm{Re} \leq 0 \rbrace$) or unstable discrete spectrum. For the latter, we know it must be situated in $\lbrace 0 < \mathrm{Re} < \gamma_0 \rbrace$ according to Corollary \ref{Coro:Gamma0} since elements of the discrete spectrum are eigenvalues. To conclude, we choose $\mathrm{Re}(\lambda) = \gamma > \gamma_0$ in the integral formula \eqref{eq:LaplaceTransform}.

We now consider the resolvent equation
\begin{align}\label{eq:resolvent}
    \begin{pmatrix}
\varrho\\
\omega
\end{pmatrix}:= (\lambda - \mathcal{L}_1)^{-1} \begin{pmatrix}
f\\
g
\end{pmatrix}, \qquad \lambda=-ic, \qquad\mathrm{Im}(c) \gg 1,
\end{align}
for $\lambda \notin \mathrm{Sp}(\mathcal{L}_1)$ (note that $\mathrm{Re}(\lambda)=\mathrm{Im}(c))$ and for a fixed $(f,g)$. It means that
\begin{align*}
    \begin{cases}\frac{f}{i}=(U_s-c) \varrho - \rho_s'\varphi, \\
    \frac{g}{i}=(U_s-c) \partial_z^2 \varphi - U_s'' \varphi-\varrho,
    \end{cases}  
\qquad\text{with} \; \varphi \; \text{in \eqref{eq:psi-hydro}}.
\end{align*}
By definition of $\gamma_0$ in \eqref{def:Gamma0}, taking $\mathrm{Im}(c)>\gamma_0$ implies that \eqref{eq:resolvent} is well defined and the previous system has a unique solution $(\varrho,\omega)$, determined by
\begin{align*}
        \varrho=\frac{1}{U_s-c}\left[ \frac{f}{i}+ \rho_s'\varphi \right], \qquad 
        \omega=\frac{1}{U_s-c}\left[ \frac{g}{i}+U_s''\varphi  \right]+\frac{1}{(U_s-c)^2}\left[ \frac{f}{i}+ \rho_s' \varphi\right].
\end{align*}
Hence, for $F:=(f,g)$, we get
\begin{align*}
    \e^{\mathcal{L}_1 \tau} F&= \frac{1}{2 i \pi} \mathrm{P.V.} \,\int_{\gamma-i \infty}^{\gamma + i \infty} \e^{\lambda \tau} (\lambda- \mathcal{L}_1)^{-1} F \, \mathrm{d}\lambda \\
    &=\frac{1}{2 \pi} \mathrm{P.V.} \,\int_{i\gamma- \infty}^{i\gamma +  \infty} \e^{-ic \tau} \frac{1}{U_s-c} 
    \begin{pmatrix} \frac{f}{i}+\varphi \rho_s'  \\ 
         \frac{g}{i}+\varphi U_s'' +\frac{1}{(U_s-c)}\left[ \frac{f}{i}+\varphi \rho_s' \right] \end{pmatrix} 
         \, \mathrm{d}c \\
         &=\frac{1}{2 \pi} \mathrm{P.V.} \,\int_{i\gamma- \infty}^{i\gamma +  \infty} \e^{-ic \tau} \frac{1}{U_s-c} 
    \begin{pmatrix} \varphi \rho_s'  \\ 
         \varphi U_s'' +\frac{1}{(U_s-c)}\varphi \rho_s'  \end{pmatrix} 
         \, \mathrm{d}c +
         \begin{pmatrix} \e^{-iU_s \tau} f   \\ 
         \e^{-iU_s \tau } g-i\tau \e^{-iU_s \tau} f \end{pmatrix}.
\end{align*}
We thus have
\begin{align*}
\Vert \e^{\mathcal{L}_1 \tau} F \Vert_{\Ld^2} &\lesssim \max \, (\|U_s'' \Vert_{\Ld^\infty}, \Vert \rho_s'  \Vert_{\Ld^\infty})  \int_{\R} \, \frac{\e^{\gamma \tau}}{\left( \vert U_s-\ell \vert^2 + \gamma^2\right)^{1/2}} 
\left(1+\frac{1}{\left( \vert U_s-\ell \vert^2 + \gamma^2\right)^{1/2}}
\right) \Vert \varphi \Vert_{\Ld^2(-1,1)} 
\mathrm{d}\ell \\
& \quad + (1+\tau) \Vert F \Vert_{\Ld^2(-1,1)}.
\end{align*}
Hence, to conclude we need a bound for $\Vert \varphi \Vert_{\Ld^2} $. We have the following (rough) elliptic estimate:
\begin{align}\label{ineq:elliptiquePhi}
\Vert \varphi \Vert_{\H^1(-1,1)} \leq C \left[ \frac{1}{\sqrt{1+\vert \mathrm{Re}(c)\vert^2}}+\frac{1}{1+\vert \mathrm{Re}(c)\vert^2} \right] \Vert F \Vert_{\Ld^2},
\end{align}
where $C \geq 0$ is a constant only depending on $\rho_s,U_s$ $\gamma_0$. In fact, consider the equation for $\varphi$:
\begin{align*}
\partial_z^2 \varphi - \left[ \frac{U_s''}{U_s-c}+\frac{\rho_s'}{(U_s-c)^2} \right]\varphi=\frac{1}{U_s-c}\frac{g}{i}+\frac{1}{(U_s-c)^2} \frac{f}{i}.
\end{align*}
By standard elliptic regularity estimates, and since 
\begin{align*}
 \vert \mathrm{Im}(c) \vert > \gamma_0 \Longrightarrow  \frac{1}{\vert U_s-c \vert }+\frac{1}{\vert U_s-c\vert^2}  \leq \frac{1}{\gamma_0}+ \frac{1}{\gamma_0^2},
\end{align*}
we have 
\begin{align*}
    \Vert \varphi \Vert_{\H^1(-1,1)} \Vert &\leq C_{\rho_s,U_s,\gamma_0} \left\Vert \frac{1}{\vert U_s-c \vert }+\frac{1}{\vert U_s-c\vert^2}   \right\Vert_{\Ld^\infty(-1,1)} \Vert F \Vert_{\Ld^2(-1,1)} \\
    & \leq C_{\rho_s,U_s,\gamma_0} \left[ \frac{1}{\sqrt{1+\vert \mathrm{Re}(c)\vert^2}}+\frac{1}{1+\vert \mathrm{Re}(c)\vert^2} \right]\Vert F \Vert_{\Ld^2(-1,1)},
\end{align*}
for another constant $C_{\rho_s,U_s,\gamma_0}$.

\medskip

Coming back to our semigroup estimate, we now infer for $\gamma > \gamma_0$
\begin{align*}
&\Vert \e^{\mathcal{L}_1 \tau} F \Vert_{\Ld^2(-1,1)} \\
&\lesssim \Vert F \Vert_{\Ld^2(-1,1)} \int_{\R_+} \, \frac{\e^{\gamma \tau}}{\left( \vert U_s-r \vert^2 + \gamma^2\right)^{1/2}} 
\left(1+\frac{1}{\left( \vert U_s-r \vert^2 + \gamma^2\right)^{1/2}}
\right)  \left[ \frac{1}{\sqrt{1+\vert r \vert^2}}+\frac{1}{1+\vert r\vert^2} \right]
\mathrm{d}r \\
& \quad + (1+\tau) \Vert F \Vert_{\Ld^2} \\
& \leq C_{\gamma} \e^{\gamma \tau}\Vert F \Vert_{\Ld^2(-1,1)}+(1+\tau) \Vert F \Vert_{\Ld^2(-1,1)},
\end{align*}
so that
\begin{align}\label{estimatesemigroupL1-F}
\Vert \e^{\mathcal{L}_1 \tau} F \Vert_{\Ld^2(-1,1)}\leq C_{\gamma}\e^{\gamma \tau}\Vert F \Vert_{\Ld^2(-1,1)},
\end{align}
for a new constant $C_\gamma>0$.

We can now conclude the proof as in \cite{HKN}. We derive some analytic estimate for $H=(R, \Omega):=\e^{\mathcal{L}_1 \tau} F$, that satisfies
\begin{align*}
      \begin{cases}(\partial_\tau +i U_s)R=i \varphi \rho_s',\\
    (\partial_\tau +i U_s)\Omega=i\varphi U_s''+i R,
    \end{cases}  
\qquad \text{with} \qquad   \begin{cases}
        \partial_z^2 \varphi=\Omega, \\ 
        \varphi_{\mid z=\pm 1}=0.
    \end{cases}  
\end{align*}
Invoking Lemma \ref{LM:analyticTransport} for each equation and summing, we get
\begin{align*}
    \frac{\mathrm{d}}{\mathrm{d}\tau} \Vert H \Vert_{\delta'}   \leq \delta' N(U_s)_{\delta'} \Vert H \Vert_{\delta'} + \Vert \varphi \rho_s' \Vert_{\delta'}+ \Vert \varphi U_s'' \Vert_{\delta'}+\Vert R \Vert_{\delta'}.
\end{align*}
Applying the algebra property of Lemma \ref{LM:propXanalytic} and the inequality
\begin{align}\label{eq:ineq-interpoSob/Analytique}
\forall f=f(z), \qquad \Vert f \Vert_{\delta'} \leq (1+ \vert \delta' \vert) \Vert f \Vert_{\H^1}+ \vert \delta' \vert^2 \Vert \partial_z^2 f \Vert_{\delta'},
\end{align}
we deduce
\begin{align*}
    \frac{\mathrm{d}}{\mathrm{d}\tau} \Vert H \Vert_{\delta'}   &\leq \delta' N(U_s)_{\delta'} \Vert H \Vert_{\delta'} + \left( \Vert \rho_s' \Vert_{\delta'}+ \Vert  U_s'' \Vert_{\delta'} \right)\Vert \varphi\Vert_{\delta'}+\Vert R \Vert_{\delta'} \\
    &\leq (1+\delta' N(U_s)_{\delta'}) \Vert H \Vert_{\delta'} + \left( \Vert \rho_s' \Vert_{\delta'}+ \Vert  U_s'' \Vert_{\delta'} \right)\Big((1+ \vert \delta' \vert) \Vert \varphi \Vert_{\H^1}+ \vert \delta' \vert^2 \Vert \partial_z^2 \varphi \Vert_{\delta'}\Big) \\
    &\lesssim (1+\delta' N(U_s)_{\delta'}) \Vert H \Vert_{\delta'} + \left( \Vert \rho_s' \Vert_{\delta'}+ \Vert  U_s'' \Vert_{\delta'} \right)\Big((1+ \vert \delta' \vert) \Vert \Omega \Vert_{\Ld^2}+ \vert \delta' \vert^2 \Vert \Omega  \Vert_{\delta'} \Big),
\end{align*}
where the last line follows by elliptic regularity. We rewrite the previous estimate as
\begin{align*}
\frac{\mathrm{d}}{\mathrm{d}\tau} \Vert H \Vert_{\delta'} &\lesssim \Gamma_{\delta', \rho_s,U_s} \Vert H \Vert_{\delta'}  +  \lambda_{\delta', \rho_s,U_s}\Vert \Omega \Vert_{\Ld^2} 
\end{align*}
with
\begin{align*}
\Gamma_{\delta', \rho_s,U_s}&:=1+\delta' N(U_s)_{\delta'}+\left( \Vert \rho_s' \Vert_{\delta'}+ \Vert  U_s'' \Vert_{\delta'} \right)\vert \delta' \vert^2, \\
\lambda_{\delta', \rho_s,U_s}&:=(1+ \vert \delta' \vert)\left( \Vert \rho_s' \Vert_{\delta'}+ \Vert  U_s'' \Vert_{\delta'} \right),
\end{align*}
and then use Grönwall's lemma that provides
\begin{align*}
\Vert H(\tau) \Vert_{\delta'} \lesssim \e^{\Gamma_{\delta', \rho_s,U_s} \tau} \Vert F \Vert_{\delta'}+ \int_{0}^\tau \lambda_{\delta', \rho_s,U_s}\e^{\Gamma_{\delta', \rho_s,U_s} (\tau-r)} \Vert \Omega(r) \Vert_{\Ld^2} \, \mathrm{d}r.
\end{align*}
We now choose $\delta'>0$ small enough so that
$\Gamma_{\delta', \rho_s,U_s} \leq 1 + \gamma_0$. 
Appealing to the estimate \eqref{estimatesemigroupL1-F} for $\Vert (R,\Omega)(s) \Vert_{\Ld^2}=\Vert \e^{\mathcal{L}_1 s} F \Vert_{\Ld^2}$, we infer
\begin{align*}
\Vert H(\tau) \Vert_{\delta'} &\lesssim \e^{(1+\gamma) \tau} \Vert F \Vert_{\delta'}+C_{\gamma}\Vert F\Vert_{\Ld^2}\int_{0}^\tau \lambda_{\delta', \rho_s,U_s}\e^{\Gamma_{\delta', \rho_s,U_s} (\tau-r)+ \gamma r}  \, \mathrm{d}r \\
& \lesssim \e^{(1+\gamma) \tau} \Vert F \Vert_{\delta'} + C_{\gamma} \e^{(1+\gamma)\tau}\Vert F\Vert_{\Ld^2},
\end{align*}
for all $\Gamma > 1+ \gamma_0$. We obtain the desired estimate as $\Vert F\Vert_{\Ld^2} \leq \Vert F\Vert_{\delta'}$.
\end{proof}
We can now prove Proposition \ref{Prop:SemigroupEstimBigL}.
\begin{proof}[Proof of Proposition \ref{Prop:SemigroupEstimBigL}]
First recall the definition of the extended linearized operator $\mathrm{L}$ from \eqref{eq:defextendedOpL}. In order to estimate
\begin{align*}
(\mathrm{R}_1,\mathrm{W}_1,\mathrm{R}_2,\mathrm{W}_2,\mathrm{R}_3,\mathrm{W}_3):=\e^{\mathrm{L} \tau} Y,
\end{align*}
we set
\begin{align*}
\mathrm{R}_{j,n}(z):=\langle \mathrm{R}_{j}(\cdot, z) , \e^{in \cdot} \rangle_{\Ld^2(\T)} , \ \mathrm{W}_{j,n}(z):=\langle \mathrm{W}_{j}(\cdot, z) , \e^{in \cdot} \rangle_{\Ld^2(\T)}, \qquad  j=1,2,3, \qquad  n \in \Z.
\end{align*}
By the structure of the operator $\mathrm{L}$, we have $(\mathrm{R}_{j,n},\mathrm{W}_{j,n})(\tau)=\e^{\mathcal{L}_n \tau}Y_{j,n}$ for $j=1,2$. Thus, by Proposition \ref{Prop:SemigroupEstimL_n}, we deduce
\begin{align}\label{eq:boundRj12}
\Vert (\mathrm{R}_{j,n},\mathrm{W}_{j,n})(\tau) \Vert_{\delta'} \leq C_\Gamma \e^{ \Gamma \vert n \vert \tau} \Vert  Y_{j,n} \Vert_{\delta'}, \qquad j=1,2,
\end{align}
for all $\Gamma > 1+\gamma_0$ and $\delta'>0$ small enough. By definition of the analytic norm, we directly get
\begin{align*}
 \Vert (\mathrm{R}_{j},\mathrm{W}_{j})(\tau) \Vert_{\delta-\Gamma \tau,\delta'} \leq C_\Gamma \sum_{n \in \Z} \e^{\Gamma \vert n \vert \tau} \Vert  Y_{j,n} \Vert_{\delta'}  \e^{(\delta-\Gamma \tau) \vert n \vert }=C_\Gamma \Vert Y_j \Vert_{\delta, \delta'}, 
\end{align*}
for any $\tau$ such that $\delta-\Gamma \tau>0$ and $j=1,2$.
We now have to deal with the case $j=3$. By definition of $(\mathrm{R}_{3,n},\mathrm{W}_{3,n})(\tau)=\e^{\mathcal{L}_n \tau}Y_{3,n}$, we have
\begin{align*}
    \begin{cases}
        (\partial_\tau + U_s \partial_\xi)\mathrm{R}_3 +U_s' \mathrm{R}_2-\rho_s' \partial_z \varphi[\mathrm{W}_2]-\rho_s'' \varphi[\mathrm{W}_2]=0,\\[2mm]
        (\partial_\tau + U_s \partial_\xi)\mathrm{W}_3 +U_s' \mathrm{W}_2-U_s'' \partial_z \varphi[\mathrm{W}_2]-U_s''' \varphi[\mathrm{W}_2]-\partial_\xi \mathrm{R}_3=0,
    \end{cases} 
\end{align*}
which can be written in Fourier in $\xi$ as follows: for all $n \in \Z$, 
\begin{align*}
      \begin{cases} 
        (\partial_\tau + inU_s)\mathrm{R}_{3,n} +U_s' \mathrm{R}_{2,n}-\rho_s' \partial_z \varphi[\mathrm{W}_{2,n}]-\rho_s'' \varphi[\mathrm{W}_{2,n}]=0,\\[2mm]
        (\partial_\tau + in U_s )\mathrm{W}_{3,n} +U_s' \mathrm{W}_{2,n}-U_s'' \partial_z \varphi[\mathrm{W}_{2,n}]-U_s''' \varphi[\mathrm{W}_{2,n}]-in \mathrm{R}_{3,n}=0.
    \end{cases} 
\end{align*}
Appealing to Lemma \ref{LM:analyticTransport} and summing the two estimates, we get
\begin{align*}
\frac{\mathrm{d}}{\mathrm{d}\tau}\left(\Vert \mathrm{R}_{3,n} \Vert_{\delta'}+\Vert \mathrm{W}_{3,n} \Vert_{\delta'} \right)  &\leq |n| \left(1 +\delta' N(U_s)_{\delta'}\right)\left(\Vert \mathrm{R}_{3,n} \Vert_{\delta'}+\Vert \mathrm{W}_{3,n} \Vert_{\delta'} \right) \\
&\quad +C_{\delta',\rho_s,U_s}\Big(\Vert  \mathrm{R}_{2,n} \Vert_{\delta'}+ \Vert \mathrm{W}_{2,n}\Vert_{\delta'} + \Vert \partial_z \varphi[\mathrm{W}_{2,n}] \Vert_{\delta'} + \Vert \varphi[\mathrm{W}_{2,n}]  \Vert_{\delta'} \Big),
\end{align*}
by the algebra property from Lemma \ref{LM:propXanalytic}. Using the straightforward inequality
\begin{align}
    \forall f=f(z), \qquad \Vert \partial_z f \Vert_{\delta'} \leq  \Vert \partial_z f \Vert_{\Ld^2}+ \vert \delta' \vert \Vert \partial_z^2 f \Vert_{\delta'},
\end{align}
and \eqref{eq:ineq-interpoSob/Analytique}, we can infer
\begin{align*}
\frac{\mathrm{d}}{\mathrm{d}s}\left(\Vert \mathrm{R}_{3,n} \Vert_{\delta'}+\Vert \mathrm{W}_{3,n} \Vert_{\delta'} \right)  &\leq \vert n \vert \left(1 +\delta' N(U_s)_{\delta'}\right) \left(\Vert \mathrm{R}_{3,n} \Vert_{\delta'}+\Vert \mathrm{W}_{3,n} \Vert_{\delta'} \right) \\
&\quad +C_{\delta',\rho_s,U_s}
\Big(\Vert  \mathrm{R}_{2,n} \Vert_{\delta'}
+ \Vert \mathrm{W}_{2,n}\Vert_{\delta'} 
+ (1+ \vert \delta'\vert)\Vert \varphi[\mathrm{W}_{2,n}] \Vert_{\H^1}
\\
& \qquad \qquad \qquad + \vert \delta'\vert(1+\vert \delta' \vert)\Vert \partial^2_z \varphi[\mathrm{W}_{2,n}] \Vert_{\delta'} 
 \Big) \\
&\leq \vert n \vert\left(1 +\delta' N(U_s)_{\delta'}\right) \left(\Vert \mathrm{R}_{3,n} \Vert_{\delta'}+\Vert \mathrm{W}_{3,n} \Vert_{\delta'} \right) \\
&\quad +C_{\delta',\rho_s,U_s}
\Big(\Vert  \mathrm{R}_{2,n} \Vert_{\delta'}
+ \Vert \mathrm{W}_{2,n}\Vert_{\delta'}  \Big),
 \end{align*}
 by elliptic regularity. By the previous bound \eqref{eq:boundRj12}, we get for all $\Gamma>1+\gamma_0$
 \begin{align*}
     \frac{\mathrm{d}}{\mathrm{d}\tau}\left(\Vert \mathrm{R}_{3,n} \Vert_{\delta'}+\Vert \mathrm{W}_{3,n} \Vert_{\delta'} \right)  &\leq \vert n \vert\left(1 +\delta' N(U_s)_{\delta'}\right) \left(\Vert \mathrm{R}_{3,n} \Vert_{\delta'}+\Vert \mathrm{W}_{3,n} \Vert_{\delta'} \right)+C_{\delta',\rho_s,U_s}C_\Gamma \e^{ \Gamma \vert n \vert \tau} \Vert  Y_{2,n} \Vert_{\delta'}.
 \end{align*}
We can now conclude as in the proof of Proposition \ref{Prop:SemigroupEstimL_n}, by using Grönwall lemma and by taking $\delta'>0$ small enough so that
\begin{align*}
\Vert (\mathrm{R}_{3,n},\mathrm{W}_{3,n})(\tau) \Vert_{\delta'} \leq C_\Gamma \e^{\Gamma \vert n \vert \tau} \Vert  Y_{n} \Vert_{\delta'},
\end{align*}
for all $\Gamma > 1+\gamma_0$. As for the case $j=1,2$, we obtain the bound on $(\mathrm{R}_{3},\mathrm{W}_{3})$. This concludes the proof of the proposition.
\end{proof}
This finally allows to obtain \ref{hypH3}, and therefore concludes the proof.

\appendix

\section{Unstable profile for the homogeneous hydrostatic   equations}\label{Appendix:ZeroFunction}

In this appendix, we demonstrate the existence of a smooth shear velocity profile $(U_s(z),0)$ around which the hydrostatic homogeneous Euler equations \eqref{eq:Hyd-E} exhibit an unbounded unstable spectrum.

By leveraging our perturbation argument presented in Section \ref{Section:GrowingMode} for the stratified case (see \eqref{eq:integHydrostaticHomog-profile}), or drawing upon the analysis in \cite{Renardy}, we establish that it is sufficient to find a profile $U_s$ such that the equation
\begin{align*}
    \int_{-1}^1 \frac{\mathrm{d}z}{(U_s(z)-c)^2}=0
\end{align*}
has a solution $c \in \lbrace \mathrm{Im}>0 \rbrace$. In \cite{Renardy}, it is asserted that some hyperbolic tangent $U_s(z)$ satisfies this property. While the instability of such profiles in homogeneous flows (at low frequency) is acknowledged in the physics literature (see \cites{rosenbluth1964necessary, chen1991sufficient, balmforth1999necessary}), we provide a proof for completeness.

We perform an analysis \textit{à la} Penrose-Nyquist (see \cite{Penrose}). Let us set \begin{align*}
F(c):=\int_{-1}^1 \frac{\mathrm{d}z}{(U_s(z)-c)^2}, \qquad c \in \lbrace \mathrm{Im}>0 \rbrace.
\end{align*}
We begin with the following result, showing that zeros of $F$ cannot exist outside a suitably large ball.
 
\begin{Lem}\label{Lem:IntegralFPositive}
Consider a smooth profile $U_s(z)$. There exists a constant $C(\Vert U_s \Vert_{\infty})>0$ such that for any $R > C(\Vert U_s \Vert_{\infty})$, it holds
\begin{align*}
\forall c \in \lbrace \vert z \vert =R \rbrace \cap \lbrace \mathrm{Im}\neq 0 \rbrace, \qquad \vert F(c) \vert >0.
\end{align*}

\begin{proof}
Let us write $c=R \e^{i \theta}$,  with $\theta \notin \pi \Z$ and  $R>\Vert U_s \Vert_{\infty}$. We compute
\begin{align*}
F(R \e^{i \theta})&=\frac{\e^{-2 i \theta}}{R^2}\int_{-1}^1 \frac{\mathrm{d}z}{\left(1-\frac{U_s(z)}{R}\e^{-i \theta}\right)^2}=\frac{\e^{-2 i \theta}}{R^2} \int_{-1}^1 \sum_{k \geq 1} k \left( \frac{U_s(z)}{R}\right)^{k-1} \e^{-i(k-1)\theta} \, \mathrm{d}z\\
&= 2\frac{\e^{-2 i \theta}}{R^2}  + \frac{1}{R} \sum_{k \geq 2} k \frac{\e^{-i(k+1)\theta}}{R^{k}} \int_{-1}^1 (U_s(z))^{k-1}  \, \mathrm{d}z.
\end{align*}
By the triangle inequality, we get
\begin{align*}
\vert F(R \e^{i \theta}) \vert>\frac{2}{R^2}-\frac{2}{\Vert U_s \Vert_{\infty}R}\sum_{k \geq 2} k \frac{\Vert U_s \Vert_{\infty}^k}{R^{k}}.
\end{align*}
To ensure that the right-hand side is positive, we need
\begin{align*}
\Vert U_s \Vert_{\infty}\left(1-\frac{\Vert U_s \Vert_{\infty}}{R}\right)^2>1-\left(1-\frac{\Vert U_s \Vert_{\infty}}{R}\right)^2.
\end{align*}
This condition is clearly satisfied if $R$ is large enough. 
\end{proof}
\end{Lem}

\begin{Prop}\label{Prop-Appendix}
Take $U_s(z):=\tanh(\beta z)$ with $z \in [-1,1]$ and $\beta>0$. For $\beta$ large enough, there exists $c \in \lbrace \mathrm{Im}>0 \rbrace$ such that 
\begin{align*}
\int_{-1}^1 \frac{\mathrm{d}z}{(U_s(z)-c)^2}=0.
\end{align*}
\end{Prop}
\begin{proof}
Let us set \begin{align*}
F(c):=\int_{-1}^1 \frac{\mathrm{d}z}{(U_s(z)-c)^2}, \qquad c \in \lbrace \mathrm{Im}>0 \rbrace.
\end{align*}
We observe that $F$ is holomorphic on $\lbrace \mathrm{Im}>0 \rbrace$. Now, let us define the half-disk and its boundary:
\begin{align*}
\Omega_{\eps, R}:&=\lbrace c \in \C \mid \mathrm{Im}(c)> \eps, \vert c -i\eps \vert \leq R \rbrace, \\
\partial \Omega_{\eps, R}:&=\partial \Omega_{\eps, R}\cap \left(\lbrace \vert c-i\eps  \vert = R \rbrace \cup \big( [-R,R]+i\eps \big) \right), \qquad \eps, R>0.
\end{align*}
By Cauchy's Theorem, a zero of the function $F$ exists in $\Omega_{\eps, R}$ if the winding number (around zero) of $F(\partial \Omega_{\eps, R})$ is strictly positive. 

For $c=i\eps+ R \e^{i\theta}$ with $\theta \in (0,\pi)$, we find that as $R \rightarrow + \infty$:
\begin{align*}
    F(i\eps+ R \e^{i\theta}) \underset{R \rightarrow + \infty}{=} \frac{\e^{-2 i \theta}}{R^2}+ \mathcal{O}\left( \frac{1}{R^2}\right).
\end{align*}

This result implies that the half-circle does not significantly contribute to the winding number, and the necessary crossings of $F(\partial \Omega_{\eps, R})$ with the positive real axis primarily come from the base of the half-disk. 
Thus we focus on the horizontal segment $\big( [-R,R]+i\eps \big)$, as the contribution from the half-circle becomes negligible for $R$ large enough. 
We note symmetry/antisymmetry properties:
\begin{align*}
    \forall a \in \R, \qquad \mathrm{Re} \, F(a+i\eps)=\mathrm{Re} \, F(-a+i\eps), \qquad \mathrm{Im} \, F(a+i\eps)=-\mathrm{Im} \, F(-a+i\eps),
\end{align*}
with $\mathrm{Im} \, F(i\eps)=0$ and a change of sign of $a \mapsto \mathrm{Im} \, F(a+i\eps)$ around $a=0$.

Let us study the function $F$. For all $c \in \lbrace \mathrm{Im}>0 \rbrace$, we have
\begin{align*}
F(c)&=\int_{-1}^1 \frac{1}{U_s'(z)}\frac{U_s'(z)}{(U_s(z)-c)^2} \, \mathrm{d}z =\int_{-1}^1 \frac{1}{U_s'(z)} \left( \frac{-1}{(U_s(z)-c)} \right)^{'} \, \mathrm{d}z \\
&=\left[ \frac{-1}{U_s' (U_s-c)} \right]^{z=1}_{z=-1}-\int_{-1}^1 \left(\frac{1}{U_s'(z)} \right)^{'}  \frac{1}{(c-U_s(z))}  \, \mathrm{d}z \\ &=\left[ \frac{1}{U_s' (c-U_s)} \right]^{z=1}_{z=-1}+\int_{-1}^1 \frac{U_s''(z)}{U_s'(z)^2}\frac{1}{(c-U_s(z))}  \, \mathrm{d}z.
\end{align*}
As $z \mapsto U_s(z)=\tanh (\beta z)$ is strictly monotone, we can change variable as follows
\begin{align*}
u=U_s(z), \qquad \mathrm{d}u= U_s'(z) \, \mathrm{d}z.
\end{align*}
Furthermore, there exists a continuous function $G: \R \rightarrow \R$ such that $G(U_s(z))= \frac{U_s''(z)}{U_s'(z)^3}$. In fact, we have
\begin{align*}
U_s''(z)=\beta^2\left(2 \tanh(\beta z)^3-2\tanh(\beta z) \right),
\end{align*}
hence
\begin{align*}
\frac{U_s''(z)}{U_s'(z)^3}=\frac{1}{\beta}\frac{2 \tanh(\beta z)^3-2\tanh(\beta z) }{(1-\tanh(\beta z)^2)^3}.
\end{align*} Thus, introducing
\begin{align*}
G(u):=\frac{1}{\beta}\frac{2u^3-2u}{(1-u^2)^3}=\frac{2}{\beta}\frac{u(u^2-1)}{(1-u^2)^3}=-\frac{2}{\beta}\frac{u}{(1-u^4)^2},
\end{align*}
we can write
\begin{align*}
F(c)=\frac{2U_s(1)}{U_s'(1)(c^2+U_s(1)^2)}+\int_{U_s(-1)}^{U_s(1)} \frac{G(u)}{c-u}  \, \mathrm{d}u.
\end{align*}
By Plemelj' formula (see \cite{lu1993boundary}), we also have for all $a \in [-R,R]$
\begin{align*}
F(a+i\eps) \underset{\eps \rightarrow 0}{\rightarrow}\widetilde{F}(a):=\frac{2U_s(1)}{U_s'(1)(a^2+U_s(1)^2)}+\int_{U_s(-1)}^{U_s(1)} \frac{G(u)}{a-u}  \, \mathrm{d}u-i \pi G(a),
\end{align*}
where the convergence is uniform on any compact set in $a$. The previous formula extends the definition of $F$ on $[-R,R]$. Observe that
\begin{align*}
\mathrm{Im} \,  \widetilde{F}(a)=0 \quad \Longleftrightarrow\quad G(a)=0 \quad \Longleftrightarrow\quad a=0,
\end{align*}
and note that $a \mapsto \mathrm{Im} \,  \widetilde{F}(a)$ is increasing around $a=0$, taking negative values before and positive values after $a=0$. By continuity and symmetries, it suffices to show that $\mathrm{Re}(F(i \eps))>0$ for small $\eps$. This will establish that the winding number around $0$ of $F(\partial \Omega_{\eps,R})$ is one, indicating one crossing from below with the real positive axis.
Note that
\begin{align*}
F(i \eps)&=\frac{2U_s(1)}{U_s'(1)(-\eps^2+U_s(1)^2)}+\int_{U_s(-1)}^{U_s(1)} \frac{G(u)}{i \eps-u}  \, \mathrm{d}u \\
&= \frac{2}{\tanh(\beta)\beta(1-\tanh(\beta)^2)}\left(1+ \frac{\eps^2}{\tanh(\beta)^2-\eps^2} \right) + \int_{-\tanh(\beta)}^{\tanh(\beta)} \frac{2}{\beta}\frac{1}{(1-u^4)^2}  \, \mathrm{d}u \\
& \quad-\eps^2 \int_{-\tanh(\beta)}^{\tanh(\beta)} \frac{2}{\beta}\frac{1}{(\eps^2+u^2)(1-u^4)^2} \, \mathrm{d}u
 +i \eps \int_{-\tanh(\beta)}^{\tanh(\beta)} \frac{2}{\beta}\frac{u}{(\eps^2+u^2)(1-u^4)^2}  \, \mathrm{d}u.
\end{align*}
Taking $\beta>0$ large enough and $\eps>0$ small enough, we observe that we can make the real part in the last expression strictly positive. This concludes the proof.
\end{proof}

\section*{Acknowledgements}
RB acknowledges funding from MUR (Italy) PRIN 2022HSSYPN. 
MCZ gratefully acknowledges support of the Royal Society through grant URF\textbackslash R1\textbackslash 191492 and the ERC/EPSRC through grant Horizon Europe Guarantee EP/X020886/1.
LE warmly thanks Yann Brenier, Daniel Han-Kwan and Marc Nualart for valuable discussions, and acknowledges financial support from the ERC/EPSRC through grant Horizon Europe Guarantee EP/X020886/1.
This work was partially funded by the Royal Society of London, Exchange Grant ICL-CNR 2020.

 \bibliography{biblio}
 \bibliographystyle{abbrv}

  \end{document}